\documentclass[ejs]{imsart}

\RequirePackage{amsthm,amsmath,amsfonts,amssymb}
\RequirePackage[numbers]{natbib}
\RequirePackage[colorlinks,citecolor=blue,urlcolor=blue]{hyperref}
\RequirePackage{graphicx}
\usepackage{cleveref}

\arxiv{2412.13020}
\startlocaldefs
\theoremstyle{plain}

\newtheorem{theorem}{Theorem}[section]
\newtheorem{lemma}[theorem]{Lemma}
\theoremstyle{definition}
\newtheorem{definition}[theorem]{Definition}

\theoremstyle{remark}

\endlocaldefs

\begin{document}
\begin{frontmatter}
\title{A Central Limit Theorem for the Permutation Importance Measure}
\runtitle{A Central Limit Theorem for the permutation importance measure}

\begin{aug}
\author[A]{\fnms{Nico}~\snm{Föge}\ead[label=e1]{nico.foege@ovgu.de}},
\author[B]{\fnms{Lena}~\snm{Schmid}\ead[label=e3]{lena.schmid@tu-dortmund.de}},
\author[A]{\fnms{Marc}~\snm{Ditzhaus$^\dagger$}\ead[label=e2]{}},
\and

\author[B,C]{\fnms{Markus}~\snm{Pauly}\ead[label=e4]{markus.pauly@tu-dortmund.de}}
\address[A]{Institut für Mathematische Stochastik,
Otto-von-Guericke Universität Magdeburg, \printead{e1,e2}}

\address[B]{Department of Statistics,
TU Dortmund University, \printead{e3,e4}}
\address[C]{Research Center Trustworthy Data Science and Security, University Alliance Ruhr, \printead{e4}}

\runauthor{N. Föge, L. Schmid, M. Ditzhaus, M. Pauly.}

\end{aug}

\begin{abstract}
Random Forests have become a widely used tool in machine learning since their introduction in 2001, known for their strong performance in classification and regression tasks. One key feature of Random Forests is the Random Forest Permutation Importance Measure (RFPIM), an internal, non-parametric measure of variable importance. While widely used, theoretical work on RFPIM is sparse, and most research has focused on empirical findings. However, recent progress has been made, such as establishing consistency of the RFPIM, although a mathematical analysis of its asymptotic distribution is still missing.\\
In this paper, we provide a formal proof of a Central Limit Theorem for RFPIM using U-statistics theory. Our approach deviates from the conventional Random Forest model by assuming a random number of trees and imposing conditions on the regression functions and error terms, which must be bounded and additive, respectively.
Our result aims at improving the theoretical understanding of RFPIM rather than conducting comprehensive hypothesis testing. However, our contributions provide a solid foundation and demonstrate the potential for future work to extend to practical applications which we also highlight with a small simulation study.
\end{abstract}



\textit{$^\dagger$\small We would like to thank the late Marc Ditzhaus, a bright mind and wonderful mentor, collaboration partner and mentee for his contribution to that paper.} 

\end{frontmatter}

\section{Introduction}

Since their introduction in 2001, Random Forests \cite{breiman2001random} quickly gained popularity due to their robust performance in both classification and regression tasks. Over the years, their application has been extended to right-censored data \cite{ishwaran2008random} and the imputation of missing values \cite{steckhoven2012forest,tang2017random}. Despite these empirical successes, the mathematical and statistical properties of Random Forests reamined poorly understood for some time. The complex structure and data-dependent weights of Random Forests pose significant challenges for rigorous mathematical analysis.

The first studies on the consistency of Random Forests began with simplified versions \cite{biau2008consistency, biau2012analysis} until \cite{scornet2015consistency} ultimately established the $L^2$-consistency of the (almost) original Random Forest algorithm introduced by \cite{breiman2001random}. Further strengthening the theoretical foundation, \cite{mentch2014ensemble} derived a Central Limit Theorem using a $U$-statistics framework, while \cite{wager2018estimation} established asymptotic normality and confidence intervals for tree ensembles under the "honesty" condition, also providing an upper bound on their bias.

Beyond predictive performance, Random Forests offer interpretability 
in the form of an internal non-parametric measure of variable importance (VIMP). Among various importance measures, we focus on the Random Forest Permutation Importance (RFPIM). The RFPIM leverages the Out-of-Bag (OOB) sample, which includes all observations not used in training a given tree during the bootstrap step. This allows for an internal validation mechanism. To compute RFPIM, the values of the feature of interest are randomly permuted within the OOB sample, and the resulting decrease in accuracy is averaged over all trees.

In contrast to the theoretical work on Random Forests themselves, results for RFPIM are sparse. However, some foundational work has been done, starting with \cite{gregorutti2017correlation}, which explored the setting of correlated features. \cite{BurimundMarkus} advanced the theoretical understanding by proving the asymptotic unbiasedness and weak consistency of RFPIM, while \cite{ishwaran2019standard} constructed confidence intervals for RFPIM assuming asymptotic normality. Although simulations supported this  assumption, a formal proof of asymptotic normality is still missing.
This is where our work makes an
important contribution. We apply 
U-statistic theory from \cite{peng2022rates} and leverage convergence rates for the CART algorithm from \cite{mazumder2024convergence} to establish and rigorously prove a Central Limit Theorem for RFPIM. Thus, providing formal support for the findings of \cite{ishwaran2019standard}.\\

The paper is organized as follows: In Section 2, we introduce our mathematical framework and formally define Permutation Importance. Section 3 presents the necessary theory of $U$-statistics and connects it to RFPIM. Our main results, including the asymptotic normality of RFPIM are provided in Section 4. In Section 5 we contribute a simulation study investigating the practical relevance of our findings.
 Finally, Section 6 discusses the implications of our findings.

\section{Mathematical framework}
We assume a sparse and non-parametric learning problem with regression function $\widetilde{m}:[0,1]^p\longrightarrow\mathbb{R}$ and model equation
\begin{align}
    Y=\widetilde m(\mathbf{X})+\varepsilon, \label{Modellgleichunf}
\end{align}
with a uniformly distributed feature vector $\mathbf{X}=\left(X^{(1)},\ldots,X^{(p)}\right)\in[0,1]^p$ and some random
error $\varepsilon$ with $\mathbb{E}\left[\varepsilon\right]=0$ and 
$\mathbb{V}\text{ar}\left(\varepsilon\right)=\sigma^2>0$. Moreover, $\mathbf{X}$ and $\varepsilon$ 
are assumed to be independent. 
With sparsity, we mean there is a subset of features $\mathcal{S}\subseteq \{1,\ldots,p\}$
that contains all information about \(Y\). Assuming without loss of generality that \(\mathcal{S} = \{1, \ldots, s\}\),
this means that there is a function \(m: [0,1]^s \rightarrow \mathbb{R}\) such that 
\(m(\mathbf{X}_{\mathcal{S}}) = \widetilde m(\textbf{X})\), where $\mathbf{X}_{\mathcal{S}}=\left(X^{(1)},\ldots,X^{(s)}\right)$, see  \cite{BurimundMarkus} for the same assumption.
It is widely known, that Random Forests adapt well to sparse frameworks because the choice of optimal splits at each node
avoids splittings on non-informative features \cite{scornet2015consistency}.

We are interested in the "Importance" of single features for $Y$.
To extract the importance of a single feature, we need to exclude interactions that shed the isolated influence of a feature on $Y$. To this end, we assume an additive structure, i.e there are functions $m_j:[0,1]\rightarrow\mathbb{R}$, $j=1,\ldots,p$ such that
\begin{align}
\label{additive Funktion}
    \widetilde m(\mathbf{X})=\sum_{j=1}^pm_j(X^{(j)}).
\end{align}
Let 
$$\mathcal{D}_n=\Big\{\big[\textbf{X}_i^T,Y_i\big]^T\in[0,1]^p\times\mathbb{R}:i=1,...,n\Big\},$$ 
with $\big[\textbf{X}_i^T,Y_i\big]^T$ be a training set containing i.i.d. copies of $[\textbf{X}^T,Y]^T$
also containing the i.i.d. errors $\varepsilon_1,\ldots,\varepsilon_n$.
A theoretical measure for the feature importance for $j\in\{1,\ldots, p\}$ is given by
\begin{align*}
    I(j):=
    \mathbb{E}\left[\Big(Y_1-\widetilde m \left(\mathbf{X}_{j,1}\right)\Big)^2\right]
    -
    \sigma^2,
\end{align*}
where
$\mathbf{X}_{j,1}=\left[X_{1}^{(1)},\ldots,X_{1}^{(j-1)},X^{(j)},X_{1}^{(j+1)},
\ldots,X_1^{(p)}\right]$ and $X^{(j)}$ an independent copy of $X_1^{(j)}$ introduced by \cite{gregorutti2017correlation}. In informal terms, \(I(j)\) quantifies the contribution of the variability of the $j$-th feature to the variability of the response variable \(Y\), by measuring the increase in prediction error when the feature \(X^{(j)}\) is excluded from the model.
 Due to our model assumptions $I(j)$
simplifies to
 \begin{align}
        I(j)&= \mathbb{E}\left[
         \left(\nonumber
        Y_1-\widetilde m(\mathbf{X}_{j,1})
    \right)^2\right]-\sigma^2
    =
      \mathbb{E}\left[
         \left(
        m_j\left(X_1^{(j)}\right)+\varepsilon_1-m_j\left(X^{(j)}\right)
    \right)^2\right]-\sigma^2\\
    &=
    \mathbb{V}\text{ar}\left(
        m_j\left(X_1^{(j)}\right)+\varepsilon_1-m_j\left(X^{(j)}\right)\right)
    -\sigma^2=
    2\mathbb{V}\text{ar}\left(
        m_j\left(X_1^{(j)}\right)\right). \label{I als var(m)}
    \end{align}
    
\subsection{Random Forest Permutation Importance}
 A Random Forest is an ensemble of
tree-based learners, where each tree is trained on a subsample of $\mathcal{D}_n$ of size $a_n\in\{1,\ldots, n\}$. As in \cite{BurimundMarkus}
we restrict ourselves to sampling without replacement. Moreover at each node of each tree $m_{try}\in\{1,\ldots , p\}$ features
 are randomly chosen as candidates for the split direction. 
 To apply the results of \cite{peng2022rates}, we 
 adapt to their framework and  
draw the number of trees
 $\widehat{M}\in\mathbb{N}$
 at random assuming 
 $\widehat{M} \sim Bin\left({n \choose n-a_n},p_M\right)$ with 
 $\mathbb{E}[\widehat M]=M$
 \cite{peng2022rates}.
 Let $\mathbf{\Theta}$ be the generic random vector
 determining the subsample construction and the feature sub-spacing.
 For a Random Forest, $\widehat M$ independent copies $\mathbf{\Theta}_1,\ldots,\mathbf{\Theta}_{\widehat{M}}$ of  $\mathbf{\Theta}$ are drawn independent of $\mathcal{D}_n$. The piecewise constant
 estimation for $\widetilde{m}$ of the $t$-th tree, $t\in\{1,\ldots,\widehat M\}$, is denoted by 
 $m_{n,1}(\cdot,\mathbf{\Theta}_t, \mathcal{D}_n)$ and the Random Forest prediction for a fixed $\mathbf{x}\in[0,1]^p$
 is defined as 
\begin{align*}
m_{n,\widehat M}(\mathbf{x},\mathbf{\Theta}_1,\ldots,\mathbf{\Theta}_{\widehat M},\mathcal{D}_n)
=
    \frac{1}{\widehat M}\sum_{t=1}^{\widehat M}m_{n,1}(\mathbf{x},\mathbf{\Theta}_t, \mathcal{D}_n).
\end{align*}
The subsample determined by $\mathbf{\Theta}_t$ is denoted by $\mathcal{D}_n^{(t)}$.
To provide a more formal description of the selection mechanism, we define the generating vector $$\mathbf{\Theta}_t=\left[\left(\mathbf{\Theta}_t^{(1)}\right)^{\top},\left(\mathbf{\Theta}_t^{(2)}\right)^{\top}\right]^{\top},$$ which can be decomposed into subvectors $\mathbf{\Theta}_t^{(1)}$ and $\mathbf{\Theta}_t^{(2)}$.
The components of $\mathbf{\Theta}_t^{(1)}=\left[\Theta_{1,t}^{(1)},\ldots,\Theta_{n,t}^{(1)}\right]^{\top}$ indicate whether the
$i$-th observation, $i=1,\ldots,n$ has been selected ($\Theta_{i,t}^{(1)}=1$) or not ($\Theta_{i,t}^{(1)}=0$), while $\mathbf{\Theta}_t^{(1)}$
models the feature sup-spacing.
The set of observation-indices
$\mathcal{D}_n^{-(t)}$
not used for
the training of the $t$-th tree is called out-of-bag (OOB) sample of the $t$-the tree. This delivers an internal validation, that can also be used for the estimation of $I(j)$
by measuring the increase of the mean squared error on the OOB-sample after permuting its observations along 
the $j$-th feature (\cite{BurimundMarkus}). 

Let $\mathbf{X}_i^{\pi_{j,t}}=(X_i^{(1)},\ldots,X_i^{(j-1)},X_{\pi_{j,t}(i)}^{(j)},X_i^{(j+1)},\ldots,X_i^{(p)})^{\top}$, 
where $\pi_{j,t}$ is a permutation of all observations in $ \mathcal{D}_n^{-(t)}$.
The Random Forest Permutation Importance Measure (RFPIM) for a feature \(j\) then is defined as
\begin{align*}
&I^{OOB}_{n,\widehat M}(j) = \frac{1}{\widehat M \gamma_n} \sum_{t=1}^{\widehat M} \sum_{i \in \mathcal{D}_n^{-(t)}} \left\{ \left( Y_i - m_{n,1}(\mathbf{X}_i^{\pi_{j,t}}, \mathbf{\Theta}_t) \right)^2 - \left( Y_i - m_{n,1}(\mathbf{X}_i, \mathbf{\Theta}_t) \right)^2 \right\} \\
&= \!\frac{1}{\widehat M \gamma_n} \sum_{t=1}^{\widehat M} \sum_{i=1}^n \!\left\{ \left( Y_i \!- \!m_{n,1}(\mathbf{X}_i^{\pi_{j,t}}\!, \mathbf{\Theta}_t) \right)^2 \!- \!\left( Y_i \!- \!m_{n,1}(\mathbf{X}_i, \mathbf{\Theta}_t) \right)^2 \right\}\! \mathbbm{1} \left\{ \!X_i \in \mathcal{D}_{n,\mathbf{X}}^{-(t)} \!\right\}\!,
\end{align*}
for \(j \in \{1, \ldots, p\}\), where \(\gamma_n\) is the cardinality of \(\mathcal{D}_n^{-(t)}\).
 As we assume sampling without replacement ($a_n$ times) during the bootstrap step \(\gamma_n = n- a_n\) is constant. 
 Furthermore, we have
 \(\mathcal{D}_{n,\mathbf{X}}^{-(t)} = \{ \mathbf{X}_i : i \in \mathcal{D}_n^{-(t)} \}\).  Note that \(\pi_{j,t}\) is independent of \(\mathcal{D}_n\).
 While \cite{BurimundMarkus} proved, that $I^{OOB}_{n, M}(j)$
is a weakly consistent and asymptotically unbiased estimator for $I(j)$,
 we 
 investigate the asymptotic distribution of $I^{OOB}_{n,\widehat M}(j)$.
Simulations of \cite{ishwaran2019standard} suggest that $I^{OOB}_{n, M}(j)$ is asymptotically normal.
In the following, we are going to formally prove, that $I^{OOB}_{n,\widehat M}(j)$ is indeed asymptotically normal under some technical assumptions. To the best of our knowledge this is the first approach for a theoretical validation of the confidence intervals for $I(j)$ proposed in \cite{ishwaran2019standard}. Note that our results are for a randomly chosen number of trees, while the previous results for permutation importance were based on a fixed number of trees.
 The assumption of a randomly chosen number of trees is necessary to fit into the $U$ statistic framework of \cite{peng2022rates}, 
who already discussed this assumption.

\section{$U$-statistics and permutation importance}
We first prove that a simplified version of ${I}^{OOB}_{n,M}(j)$ is asymptotically normal via $U$-statistics and then show that ${I}^{OOB}_{n,M}(j)$ and it's simplified version have the same asymptotic distribution. 

We need a simplified version, because ${I}^{OOB}_{n,M}(j)$ does not fit into the general definition of a $U$-statistics, since the summation only goes over
$M<{n \choose \gamma_n}$ subsets. Moreover, the permutation puts additional randomness to the kernel. 
Luckily there is a solution for both issues known as generalized $U$-statistics.
\begin{definition}[Generalized $U$-statistics]
\label{Def 1}
    Let $\mathbf{Z}_1,\ldots,\mathbf{Z}_n$ be i.i.d. random vectors and let $h:\mathbb{R}^s\rightarrow\mathbb{R}$ be a (possibly randomized)
    permutation symmetric function.
    Then $U$-statistics with kernel $h_s$ are defined as estimators of the form,
    \begin{align}
          U_{n,s,M,\omega}=\frac{1}{\widehat M}\sum_{(t)}\rho_t h_s\left(\mathbf{Z}_{i_1},\ldots\mathbf{Z}_{i_s};\omega_t\right)
          \label{generalized U-Statistic}
    \end{align}
    where the summation is over all possible subsests from $\{\mathbf{Z}_1,\ldots,\mathbf{Z}_n\}$ of size $s$. Moreover, $\omega_t$ denotes i.i.d. randomness independent of     $(\mathbf{Z}_i)_i$ and $(\rho_t)_t$, where $\widehat M=\sum_t\rho_t$ and
    $\rho_t$ are i.i.d. Bernoulli random variables with     
    $\mathbb{P}\left(\rho_t = 1\right)=M/{n \choose s}$
    determining  which subsamples are selected. If $M={n \choose s}$, then (\ref{generalized U-Statistic}) is called a generalized complete $U$-statistics and if $M<{n \choose s}$ it 
    is called generalized
 incomplete $U$-statistic \cite{peng2022rates}.
\end{definition}
\noindent $U$-statistics already were used to establish normality and rates of convergence for the Random Forest regression estimator $m_{n,M}$
in \cite{mentch2016quantifying} and \cite{peng2022rates}. We want to apply the theory to the RFPIM. However, $I_{n,\widehat M}^{OOB}(j)$ cannot be put directly into a form like $U_{n,s,M,\omega}$.
The situation is particularly complex because, for each tree, a subset of $\mathcal{D}n$ is used to train $m_{n,1}$, while the remaining observations are utilized to compute the OOB-error.
That is why we start with a simplification of $I_{n,M}^{OOB}(j)$
by replacing the tree-estimation by the true regression function and thus define
\begin{align}
    \widetilde{I}^{OOB}_{n,M}(j)
    &=
    \frac{1}{M}
    \sum_{t=1}^M
        \frac{1}{\gamma_n}
    \sum_{i\in\mathcal{D}_n^{-(t)}}
    \Big\{
        \big(
        Y_i-\widetilde{m}(\textbf{X}_i^{\pi_{j,t}})
        \big)^2
        -
        \big(
        Y_i-\widetilde{m}(\textbf{X}_i)
        \big)^2
    \Big\}
    \Big\}.\label{I mit wahren m}
\end{align}
\noindent Our simplified version (\ref{I mit wahren m}) of permutation importance fits into \Cref{Def 1}, when we choose
 $s=\gamma_n$, $\omega_t=\pi_{j,t}$ and $(\mathbf{Z}_1,\ldots,\mathbf{Z}_n):= ((\mathbf{X}_{1},Y_{1}),\ldots,(\mathbf{X}_{n},Y_{{n}}))$.
Thus $\widetilde{I}^{OOB}_{n,\widehat M}(j)$ is a generalized $U$-statistic
with kernel 
$$
h_{\gamma_n}\left(\mathbf{Z}_{i_1},\ldots\mathbf{Z}_{i_{\gamma_n}};\pi_{j,t}\right)
=
\frac{1}{\gamma_n}
    \sum_{i\in\mathcal{D}_n^{-(t)}}
    \Big\{
        \big(
        Y_i-\widetilde{m}(\textbf{X}_i^{\pi_{j,t}})
        \big)^2
        -
        \big(
        Y_i-\widetilde{m}(\textbf{X}_i)
        \big)^2
    \Big\}
    \Big\}.
$$
Below we first show that $\widetilde{I}^{OOB}_{n,M}(j)$ is asymptotically normal and then show that its difference with ${I}^{OOB}_{n,M}(j)$ asymptotically vanishes in probability to deduce the main result.

\section{Asymptotic results}
\subsection{Central Limit Theorem for RFPIM}
This section contains our main asymptotic results for the Permutation Importance Measure.
To prove these, we need some additional assumptions
\begin{enumerate}
    \item[(A1)]  We only consider permutations $\pi_{j;t}\in\mathcal{V}_n
    =\left\lbrace \pi\in\mathcal{S}_{\gamma_n}:\pi(i)\neq i~\text{for all}~i\right\rbrace$, where $\mathcal{S}_{\gamma_n}$ is the symmetric group.
    \item[(A2)] $\widetilde m$ is an additive function, i.e. there are functions $m_{\ell}:[0,1]\rightarrow\mathbb{R}$, $\ell=1,\ldots ,p$ such that $\widetilde m(\mathbf{X})
    =\sum\limits_{\ell=1}^p m_{\ell}\left(X^{(l)}\right)$.
    \item[(A3)] $\widetilde m$ is bounded, i.e. there is a $K>0$, such that $\sup\limits_{\mathbf{x}\in[0,1]^p}\vert\widetilde{m}(\mathbf{x})\vert\leq K$. 
    \item[(A4)] The error terms have finite fourth moments, i.e. $\mathbb{E}[\varepsilon_1^4]<\infty$. 
    \item[(A5)]  $\widehat M\sim Bin\left({n \choose \gamma_n},M/{n \choose \gamma_n}\right)$ is chosen at random
    \item[(A6)] The features are mutually independent.
\end{enumerate}
The Assumption (A1) is needed for technical reasons and was already used in \cite{BurimundMarkus} to establish consistency of the RFPIM .
(A2) is needed to make an isolated measurement of the influence of each feature possible. Condition (A3) is needed for technical reasons, but we believe that the existence of an upper bound for the regression function is not too restrective for most practical applications.
The moment assumption (A4) is required for the consistency of the tree prediction \cite{BurimundMarkus}.
For the proof of the consistency of $m_{n,M}$ in \cite{scornet2015consistency} assumption (A2) is also needed. \cite{scornet2015consistency}
assumes the component functions to be continuous and instead of (A4) requires gaussian errors.
The random number of trees (A5) is needed to fit in the generalized $U$-statistics framework of
\cite{peng2022rates}. Condition (A6) also implies uncorrelated features, which is important since correlated features can lead to significant bias in RFPIM as discussed in \cite{strobl2008conditional}.
To derive asymptotic normality, we need to make sure that we are in a non-degenerated case. To this end we additionally assume that 
$\mathbb{V}\text{ar}\left(m_j(X_1^{(j)})\right)>0$. In fact, 
if $\mathbb{V}\text{ar}\left(m_j\left(X_1^{(j)}\right)\right)=0$, then $m_j(X_1^{(j)})$ would be constant almost surely 
meaning the permutation along the $j$-th component of $\mathbf{X}_1$ would not change $\widetilde m\left(\mathbf{X}_1\right)$,
which implies
$\widetilde{I}^{OOB}_{n,M}(j)=0$ almost surely. 
Under these assumptions, $\widetilde{I}^{OOB}_{n,M}(j)$ fits into the framework of generalized incomplete $U$-statistics
and we can make use of Theorem 2 of \cite{peng2022rates} to achieve its asymptotic normality. To better follow the prove ideas, we repeat their theorem in
\begin{theorem}\label{Theorem 2.1 von Peng}
    Let $\mathbf{Z}_1,\ldots,\mathbf{Z}_n$ be i.i.d. from $F_{\mathbf{Z}}$
    and $U_{n,\gamma_n,M,\pi}$ be a generalized incomplete 
    $U$-statistic with kernel $h_{\gamma_n}(\mathbf{Z}_1,\ldots,\mathbf{Z}_n;\pi_j)$.
    Let $\mu=\mathbb{E}[ h_{\gamma_n}]$, $\zeta_{1,\gamma_n}=\mathbb{V}\text{ar}(\mathbb{E}[ h_{\gamma_n}|\mathbf{Z}_1])$
    and $\zeta_{\gamma_n}=\mathbb{V}\text{ar}( h_{\gamma_n})$.
    Suppose that $$\frac{\mathbb{E}\left[\vert h_{\gamma_n}-\mu \vert^{2k}
    \right]}{
    \mathbb{E}\left[\vert h_{\gamma_n}-\mu \vert^{k}
    \right]^2}$$
    is uniformly bounded in $\gamma_n$ for $k=2,3$.
    If $\frac{\gamma_n}{n}\frac{\zeta_{\gamma_n}}{\gamma_n\zeta_{1,\gamma_n}}\rightarrow 0$, then
    \begin{align}
        \frac{U_{n,\gamma_n,M,\pi}-\mu}{\sqrt{\gamma_n^2\zeta_{1,\gamma_n}/n+\zeta_{\gamma_n}/M}}
        \overset{d}{\longrightarrow} \mathcal{N}(0,1) \nonumber
    \end{align}
    as $M\rightarrow\infty$ and $n\rightarrow\infty$ \cite{peng2022rates}.
\end{theorem}
\noindent To apply \Cref{Theorem 2.1 von Peng} to $\widetilde{I}^{OOB}_{n,M}(j)$, the main part of the work is to show that
$\mathbb{E}\left[\vert  h_{\gamma_n}-\mu \vert^{2k}
    \right]/
    \mathbb{E}\left[\vert h_{\gamma_n}-\mu \vert^{k}
    \right]^2$
    is uniformly bounded in $\gamma_n$ for $k=2,3$ and to prove that $\frac{\gamma_n}{n}\frac{\zeta_{\gamma_n}}{\gamma_n\zeta_{1,\gamma_n}}\rightarrow 0$.
    This is done in the Appendix (\Cref{Lemma 1} and \Cref{Lemma 2} and some lengthy calculations) and we thus receive the following
\begin{theorem}\label{Theorem 3}
    Let (A1)-(A6) be fulfilled and $\mathbf{Z}_1,\ldots,\mathbf{Z}_{n}$ independent and identically distributed.
    If $\mathbb{V}\text{ar}\left(m_j(X_1^{(j)})\right)>0$, $\gamma_n\rightarrow\infty$ as $n\rightarrow\infty$, and $\gamma_n/\sqrt{n}\rightarrow 0$  as $n\rightarrow\infty$,
    then
    \begin{align*}
   \frac{\widetilde I_{n,M}^{OOB}(j)-I(j)}{\sqrt{\gamma_n^2\zeta_{1,\gamma_n}/n+\zeta_{\gamma_n}/M}}
        \overset{d}{\longrightarrow }\mathcal{N}(0,1)
    \end{align*}
    as $M\rightarrow\infty$ and $n\rightarrow\infty$.
\end{theorem}
\noindent To transfer the result of \Cref{Theorem 3} 
to $I_{n,M}^{OOB}(j)$,
a sufficiently fast convergence of the base learner $m_{n,1}$ is needed. \cite{BurimundMarkus} already pointed out, that the 
consistency of the individual trees implies the consistency of $I_{n,m}^{OOB}(j)$. We need the stronger assumption that the approximation error of $m_{n,1}$
for $\widetilde{m}$ is negligible. Therefore we need to make the so called
'Sufficient Impurity Decrease' assumption and use results from \cite{mazumder2024convergence}. 
Let $A\subset [0,1]^p$ be some hyperrectangle, that is partitioned into $A_L$ and $A_R$ via axis-aligned splitting. More formally for $b\in[0,1]$ and $j\in\{1,\ldots,p\}$
\begin{align*}
    A_L&=A_L(j,b):=A\cap\{v\in[0,1]^p:v^{(j)}\leq b\},\\
    A_R&=A_R(j,b):=A\cap\{v\in[0,1]^p:v^{(j)}> b\},\,\
\end{align*}
where $v^{(j)}$ denotes the $j$-th component of $v$. Then the theoretical version of the CART splitting criterium \cite{scornet2015consistency} is defined as
\begin{align*}
    \Delta(A,j,b):=&\mathbb{P}(\mathbf{X}\in A)\mathbb{V}\text{ar}(\widetilde{m}(\mathbf{X})|\mathbf{X} \in A)
    -\mathbb{P}(\mathbf{X}\in A_L)\mathbb{V}\text{ar}(\widetilde{m}(\mathbf{X})|\mathbf{X} \in A_L)\\
    &-\mathbb{P}(\mathbf{X}\in A_R)\mathbb{V}\text{ar}(\widetilde{m}(\mathbf{X})|\mathbf{X} \in A_R).
\end{align*}
Equipped with this we can formulate the 'Sufficient Impurity Decrease' (SID) assumption 
\begin{itemize}
    \item[(A7)] There is a constant $\lambda\in(0,1)$ such that for all hyperrectangles $A\subset[0,1]^p$
    \begin{align*}
        \sup_{1\leq j\leq p,b\in[0,1]}\Delta(A,j,b)\geq\lambda\cdot\mathbb{P}(\mathbf{X}\in A)\mathbb{V}\text{ar}(\widetilde{m}(\mathbf{X})|\mathbf{X} \in A).
    \end{align*}
\end{itemize}
Additionally we need a stricter version of (A4), to give a uniform upper bound for the tree predictions.
\begin{itemize}[leftmargin=3em]
    \item[(A4*)] The error terms are bounded by some $K_{\varepsilon}>0$.
\end{itemize}
Moreover, we need to link the growth-rate of $M$ to the sample size, i.e. $M=M_n$
Equipped with these assumptions we can now prove our main theorem, which delivers the asymptotic normality of $I_{n,\widehat M}^{OOB}(j)$.
\begin{theorem}\label{Theorem 4}
    Let (A1)-(A3), (A4*), (A5)-(A7) be fulfilled and $\mathbf{Z}_1,\ldots,\mathbf{Z}_{n}$ independent and identically distributed and  $\mathbb{V}\text{ar}\left(m_j(X_1^{(j)})\right)>0$. 
    If $\gamma_n\rightarrow\infty$,  
    $\frac{M_n\gamma_n\log(a_n)^2\log(a_np)}{a_n^{\phi(\lambda)}}\rightarrow 0$ for $\phi(\lambda)=\frac{-\log_2(1-\lambda)}{1-\log_2(1-\lambda)}$ as $M_n\rightarrow\infty$ and $n\rightarrow\infty$.
    Then
    \begin{align*}
\frac{ I_{n,M}^{OOB}(j)-I(j)}{\sqrt{\gamma_n^2\zeta_{1,\gamma_n}/n+\zeta_{\gamma_n}/M_n}}
        \overset{d}{\longrightarrow }\mathcal{N}(0,1)
    \end{align*}
    as $M_n\rightarrow\infty$ and $n\rightarrow\infty$, if the trees are grown to depth $d=\lceil\log_2(n)/(1-\log_2(1-\lambda))\rceil$.
\end{theorem}
 \noindent   To our knowledge this is the first CLT for the RFPIM and thus the first mathematical reasoning for 
 the asymptotic normality assumed in \cite{ishwaran2019standard}.\\
    The assumption (A7) is an assumption on the regression function $\widetilde{m}$, that is not very intuitive. Function classes satisfying this condition were investigated
    by \cite{mazumder2024convergence}. Proposition 3.1 of \cite{mazumder2024convergence} tells us that (A7) is fulfilled if $m_1,\ldots,m_p$ are from the Locally Reverse Poincaré Class
    with parameter $\tau$,
    i.e. $m_1,\ldots,m_p$ are differentiable on $(0,1)$ and for any subinterval $[a,b]\subseteq (0,1)$ and
    $$
    \left(\int_a^b\vert m_j'(t)\vert dt\right)^2
    \leq
    \frac{\tau^2}{b-a}\inf_{w\in\mathbb{R}}\int_a^b\vert m_j(t)-w\vert^2dt
    $$
    for every $j\in\{1,\ldots,p\}$. In this case $\lambda=4/p\tau^2$.
    This is fulfilled if for instance each $m_j$ is a piecewise defined function, whoses pieces are polynomial, strictly monotonous functions or smooth and strongly convex.
    
The values for $\lambda$ in the examples given in \cite{mazumder2024convergence} are rather small. Thus the desired depth $d$ is close to $\log_2(n)$ which corresponds to the depth of fully grown binary tree. The assumption of $$\frac{M_n\gamma_n\log(a_n)^2\log(a_np)}{a_n^{\phi(\lambda)}}\rightarrow 0$$ as $n\rightarrow \infty$ is 
quite restrictive and is not fulfilled by typical choices like $a_n\approx 0.632\cdot n$. Our assumption is for example fulfilled by choices like $M_n=\lfloor\log(n)^{q_1}\rfloor$ and
$\gamma_n=\lfloor n-\log(n)^{q_2}\rfloor$ for some arbitrary $q_1,q_2>0$.
\section{Simulation Study}
In this section, we conduct a simulation study to evaluate the practical relevance of our central limit theorem. Specifically, we compare the empirical distribution of RFPIM under various conditions. Furthermore, we examine the robustness of our assumptions by analyzing how the distribution of RFPIM changes when these assumptions are violated. All simulations are performed using R version 4.4.1 \cite{R}.

\subsection{Simulation settings}
For each of the $N=1000$ simulation runs we generate $n=2500$ i.i.d. copies $(\mathbf{X_1},Y_1),\ldots,(\mathbf{X}_n,Y_n)$ of $(\mathbf{X},Y)$ following Model  (\ref{Modellgleichunf}).
The sample size equals the sample size in \cite{ishwaran2019standard}, who already illustrated via simulation, that the distribution of RFPIM can often be approximated by a normal distribution.
We choose $p=10$, i.e. $\mathbf{X}\sim\mathcal{U}\left([0,1]^{10}\right)$.

The Random Forest algorithm is implemented in R using the $\texttt{ranger}$ package \cite{ranger}. For each $k = 1, \ldots, N$, we train a Random Forest model with $M = 500$ trees and a subsample size of $a_n = \lfloor 2/3 \cdot n \rfloor$. The parameter $m_{try}$ is set to the default value provided by the $\texttt{ranger}$ package. Using the trained Random Forest, we compute the RFPIM $I_{n,\widehat M}^{OOB}(j)$ for $j = 1, \ldots, 10$, and denote the results from the $k$-th simulation run by $I_{n,\widehat M}^{OOB,k}(j)$. Finally, we compare the empirical quantiles of the standardized values $I_{n,\widehat M}^{OOB,1}(j), \ldots, I_{n,\widehat M}^{OOB,N}(j)$ with the quantiles of the standard normal distribution using QQ-plots. To investigate robustness of our results to derivations from our assumptions, 
we vary the error distribution, the regression function $\widetilde m$, and whether permutations or derrangements are used for the feature importance as described below.\\
\textit{Error Distribution}.
In \Cref{Theorem 4} we require (A4*), which is that the error terms are bounded. This is fulfilled for a truncated normal distribution. We use the R function \texttt{rtrunc} from the package \texttt{truncdist} (\cite{truncdist})
to generate the truncated normal distributed error terms. We choose $-2$ as lower bound and $2$ as upper bound. The mean is set to zero and the standard deviation is set to one for the truncated normal distribution.
To test how crucial this stricter version of (A4) is, we also use standard normal distributed error terms fulfilling (A4), but not (A4*).\\ 
\textit{Regression Functions}.
For the regression function $\widetilde m:[0,1]^p\longrightarrow \mathbb{R}$ our theory needs an additive structure (A2). We use a total of four different regression functions, where the first fulfills (A2) and the other three do not.
The first two regression functions are inspired by \cite{BurimundMarkus}, while the other non-additive functions are taken from \cite{ishwaran2019standard}. 
In contrast to \cite{ishwaran2019standard} we use less noisy features in oder to stay consistent with the first function fulfilling (A2). The regression functions are specified as follows:
\begin{enumerate}
    \item First, we use a polynomial regression function 
    that satisfies (A2):
    \begin{align*}
         \widetilde m(\mathbf{X})=2 X^{(1)} + 4 \left(X^{(2)}\right)^2 + 2 \left(X^{(3)}\right)^3 - 3\left(X^{(4)}\right)^4 + \left(X^{(5)}\right)^5
    \end{align*}
    \item We investigate a non-additive trigonometric function as link-function that connects $Y$ with a lineare/additive predictor:
    \begin{align*}
          \widetilde m(\mathbf{X})=\sin\left(2 X^{(1)} + 4 X^{(2)} + 2 X^{(3)} - 3X^{(4)} + X^{(5)}\right)
    \end{align*}
     \item We use a function with a trigonometric and a polynomial part:
    \begin{align*}
          \widetilde m(\mathbf{X})=10\sin\left(\pi X^{(1)} X^{(2)}\right) + 20\left(X^{(3)} - 0.5\right)^2
          + 10 X^{(4)} + 5 X^{(5)}
    \end{align*}
    \item Finally, we consider a function with three interaction effects:
    \begin{align*}
           \widetilde m(\mathbf{X})=X^{(1)}X^{(2)} + \left(X^{(3)}\right)^2 - X^{(4)} X^{(7)} +X^{(8)}X^{(10)} -\left(X^{(6)}\right)^2
    \end{align*}
\end{enumerate}
\textit{Permutation and Derrangement}. Our assumption (A1) restricts ourselves to derangements, which is needed for technical reasons and has previously been assumed in \cite{BurimundMarkus} to prove consisteny of the RFPIM. To get an impression whether this is actually needed, we consider both scenarios and run the simulation under ordinary permutations and derangements.

\subsection{Simulation Results}
In \Cref{Plot1}, we present the results for each of the four regression functions, focusing on the non-noise features, i.e., those where \( I(j) > 0 \). Under the assumption (A2), this corresponds to the condition \( \mathbb{V}\text{ar}(m_j(X_1^{(j)})) > 0 \) as stated in \Cref{Theorem 3} and \Cref{Theorem 4}.
We begin by examining potential violations of the additive structure of the regression function as per assumption (A2). The first plot in \Cref{Plot1} is the only one where all the assumptions of \Cref{Theorem 4} hold. Here, we observe that the quantiles from the simulation closely align with the quantiles of a normal distribution, suggesting that the sample size of \( n = 2500 \) and the number of trees \( M = 500 \) are sufficient for a good approximation.
For the second and third function , the normal distributions also appears to fit the distribution of the standardized RFPIM well, with the quantiles aligning closely. However, for the fourth function, the interaction effects introduce severe discrepancies. In this case, the normal distribution exhibits heavier tails than the RFPIM distribution, indicating a potential violation of the normality assumption.
The usage of arbitrary permutations instead of derangements does not change the QQ-Plots much as we can see in \Cref{Plot 2}. One could even argue, that the empirical quantiles fit slightly better to the normal distributions compared to \Cref{Plot1}.

\begin{figure}[h]
    \includegraphics[width=1\linewidth]{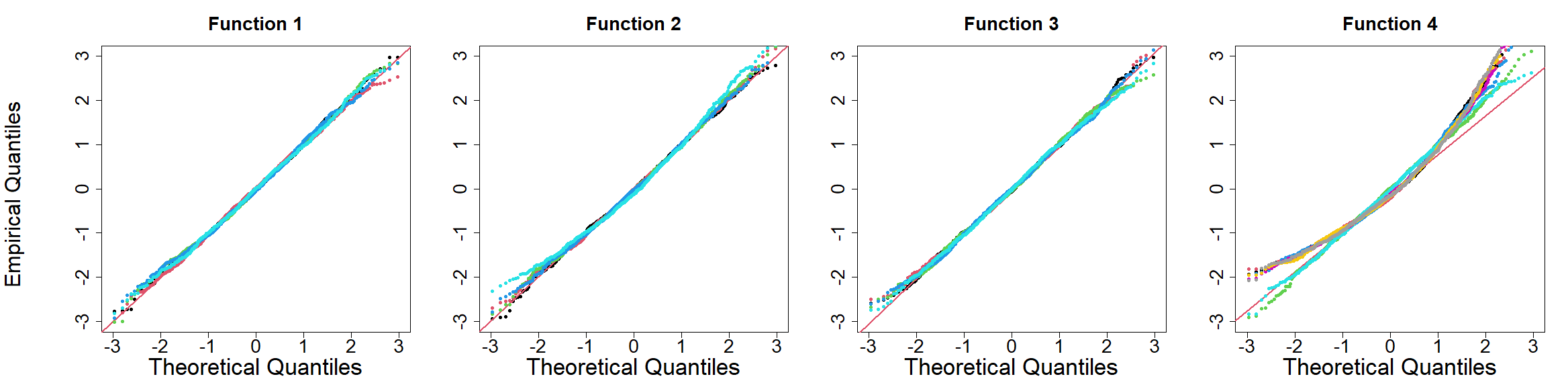}
    \caption{QQ-Plots for truncated normal errors and derrangements under the four different regression models.}
    \label{Plot1}
\end{figure}

\newpage

\begin{figure}[h]
    \centering
    \includegraphics[width=1\linewidth]{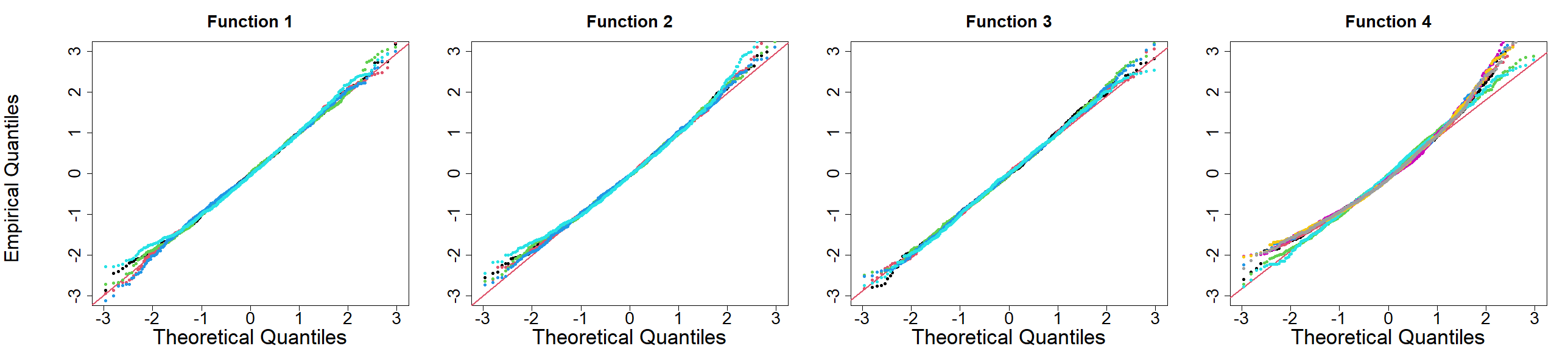}
    \caption{QQ-Plots for truncated normal errors and full permutation under the four different regression models.}
    \label{Plot 2}
\end{figure}

\begin{figure}[h]
    \centering
    \includegraphics[width=1\linewidth]{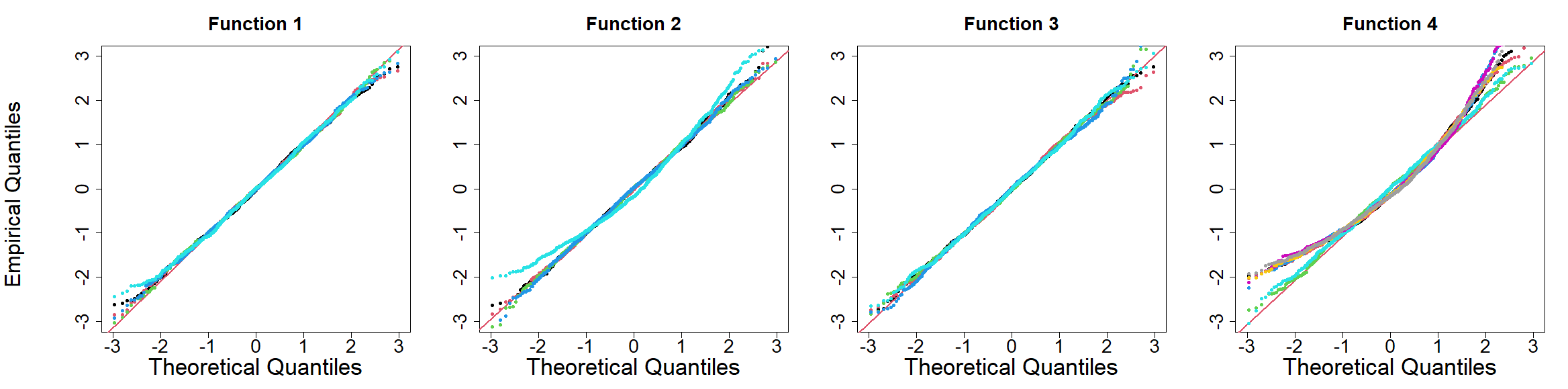}
    \caption{QQ-Plots for normal errors and derrangements under the four different regression models.}
    \label{Plot 3}
\end{figure}

\noindent Thus the simulation suggests, that (A1) is not that crucial for the CLT given in \Cref{Theorem 4}.
The results for normal (instead of truncated normal) errors combined with derangements are given in \Cref{Plot 3}. We again observe some deviation from the normal distribution for the fourth function and now also for the second one. For the two other regression functions, asymptotic normality again seems to be plausible. 

\section{Discussion}

In this paper, we employ $U$-statistics theory to establish, to the best of our knowledge, the first Central Limit Theorem for the Random Forest Permutation Importance (RFPIM). This result provides a mathematical guarantee that the standardized RFPIM is asymptotically normally distributed, supporting the assumption made by  \cite{ishwaran2019standard}.
Our theory applies to all bounded and additive regression functions whose components satisfy the 'Sufficient Impurity Decrease' assumption
and bounded errors. This includes smooth and convex functions, piecewise polynomials, and piecewise monotone functions. However, the CLT needs a few other technical restrictions that deviate from the original Random Forest framework: we select the number of trees at random, to fit in the $U$-statistics framework of \cite{peng2022rates}.  
 Another restriction is in the choice of hyperparameters because we need a huge sample size of the bootstrap sample in the bootstrap step of the Random Forests in comparison to 
the Out-Of-Bag sample, to control the approximation errors of the individual trees. Moreover, as in the consistency proofs of the RFPIM in \cite{BurimundMarkus}, we restrict the results to derangements instead of all permutations. 

In a simulation study we investigated how the asymptotic distribution behaves if these assumptions are violated. It turned out that the restriction to derangements do not seem to have a huge impact on the asympotic RFPIM distribution. The same holds true for the assumption of bounded error terms. However, the additive structure of the regression function turned out to be important, since asymptotic normality does not seem to hold if there are multiple multiplicative effects. The result nevertheless 
delivers an important first step to the theoretical validation of the confidence intervals proposed in  \cite{ishwaran2019standard}. 

\newpage


\begin{appendix}
\section{Proof of Theorem 3}\label{appA}
\begin{lemma}\label{Lemma 1}
    Let (A1)-(A4) be fulfilled. If $\mathbb{V}\text{ar}\left(m_j(X_1^{(j)})\right)>0$, then there is a $\widetilde\zeta>0$, such that
    $\gamma_n^2\zeta_{1,\gamma_n}\longrightarrow \widetilde\zeta>0$ as $n\rightarrow\infty$.
\end{lemma}
\begin{proof}
We start with 
\begin{align}
        &\mathbb{E}\left[
h|\mathbf{Z}_1
\right]
=
 \mathbb{E}\left[
\frac{1}{\gamma_n}\nonumber
    \sum_{i=1}^{\gamma_n}
    \Big\{
        \big(
        Y_{i}-\widetilde{m}
        (\textbf{X}_{i}^{\pi_{j,t}})
        \big)^2
        -
        \big(
        Y_{i}-\widetilde{m}
        (\textbf{X}_{i})
        \big)^2
    \Big\}\bigg|\mathbf{Z}_1
\right]\\
&=
\frac{1}{\gamma_n}
    \sum_{i=1}^{\gamma_n}
 \mathbb{E}\left[
        \big(
        Y_{i}-\widetilde{m}
        (\textbf{X}_{i}^{\pi_{j,t}})
        \big)^2
        \bigg|\mathbf{Z}_1
\right]
        -
        \frac{1}{\gamma_n}
    \sum_{i=1}^{\gamma_n}
         \mathbb{E}\left[
        \big(
        Y_{i}-\widetilde{m}
        (\textbf{X}_{i})
        \big)^2
    \bigg|\mathbf{Z}_1
\right]. \label{Lemma 1 h.bed.erw}
\end{align}
The second term can be rewritten as
\begin{align}
       &\frac{1}{\gamma_n}\nonumber
    \sum_{i=1}^{\gamma_n}
         \mathbb{E}\left[
        \big(
        Y_{i}-\widetilde{m}
        (\textbf{X}_{i})
        \big)^2
    \bigg|\mathbf{Z}_1
\right]\\
&=
 \frac{1}{\gamma_n}
 \mathbb{E}\left[
        \big(
        Y_{1}-\widetilde{m}
        (\textbf{X}_{1})
        \big)^2
    \bigg|\mathbf{Z}_1
\right]
+
       \frac{1}{\gamma_n}\nonumber
    \sum_{i=2}^{\gamma_n}
         \mathbb{E}\left[
        \big(
        Y_{i}-\widetilde{m}
        (\textbf{X}_{i})
        \big)^2
    \bigg|\mathbf{Z}_1
\right]\\
&=
 \frac{1}{\gamma_n}
        \big(
        Y_{1}-\widetilde{m}
        (\textbf{X}_{1})
        \big)^2
+
       \frac{1}{\gamma_n}\nonumber
    \sum_{i=2}^{\gamma_n}
         \mathbb{E}\left[
        \big(
        Y_{i}-\widetilde{m}
        (\textbf{X}_{i})
        \big)^2
\right]\\
&=
\frac{1}{\gamma_n}
        \varepsilon_1^2
+
       \frac{\gamma_n-1}{\gamma_n}
         \mathbb{E}\left[
        \varepsilon_1^2 \label{lemma 1 2. teil}
\right]
\end{align}
almost surely.
We use (A2) and get
\begin{align*}
    \left(
        Y_i-\widetilde m \left(X_i^{\pi_{j,1}} \right)
    \right)^2
    &=
    \left(
    \sum_{\ell=1}^pm_{\ell}\left(X_i^{(\ell)}\right)+\varepsilon_i
   \! -\!
     \sum_{\ell=1,\ell\neq j}^pm_{\ell}\left(X_i^{(\ell)}\right)
     -
     m_j\left(X_{\pi_j(i)}^{(j)}\right)\!
    \right)^2\\
    &=
    \left(
        m_j\left(X_i^{(j)}\right)+\varepsilon_i-m_j\left(X_{\pi_j(i)}^{(j)}\right)
    \right)^2
\end{align*}    
for $i=1,\ldots ,n$.
Thus, for the first term in (\ref{Lemma 1 h.bed.erw})
we get
\begin{align}
    &\frac{1}{\gamma_n}
    \sum_{i=1}^{\gamma_n}
 \mathbb{E}\left[
        \big(
        Y_{i}-\widetilde{m}
        (\textbf{X}_{i}^{\pi_{j,t}})
        \big)^2
        \bigg|\mathbf{Z}_1
\right]\nonumber\\
&=
\frac{1}{\gamma_n}\nonumber
    \sum_{i=1}^{\gamma_n}
 \mathbb{E}\left[
         \left(
        m_j\left(X_i^{(j)}\right)+\varepsilon_i-m_j\left(X_{\pi_j(i)}^{(j)}\right)
    \right)^2
        \bigg|\mathbf{Z}_1
\right]\\
&=
\frac{1}{\gamma_n}
    \sum_{i=1}^{\gamma_n}
 \mathbb{E}\left[
         \left(
        m_j\left(X_i^{(j)}\right)+\varepsilon_i
    \right)^2
        \bigg|\mathbf{Z}_1
\right]
+
\frac{1}{\gamma_n}
    \sum_{i=1}^{\gamma_n}
 \mathbb{E}\left[
         \left(
        m_j\left(X_{\pi_j(i)}^{(j)}\right)
    \right)^2\nonumber
        \bigg|\mathbf{Z}_1
\right]\\
&+
\frac{2}{\gamma_n}
    \sum_{i=1}^{\gamma_n}
 \mathbb{E}\left[
 m_j\left(X_{\pi_j(i)}^{(j)}\right)
         \left(
        m_j\left(X_i^{(j)}\right)+\varepsilon_i
    \right)
        \bigg|\mathbf{Z}_1
\right]\label{lemma 1 y-pi}
\end{align}
almost surely.
For the first sum we get
\begin{align}
   & \frac{1}{\gamma_n}\nonumber
    \sum_{i=1}^{\gamma_n}
 \mathbb{E}\left[
         \left(
        m_j\left(X_i^{(j)}\right)+\varepsilon_i
    \right)^2
        \bigg|\mathbf{Z}_1
\right]\\
&=
\frac{1}{\gamma_n}
\mathbb{E}\left[
         \left(
        m_j\left(X_1^{(j)}\right)+\varepsilon_1
    \right)^2
        \bigg|\mathbf{Z}_1
\right]
+
    \frac{1}{\gamma_n}
    \sum_{i=2}^{\gamma_n}
 \mathbb{E}\left[
         \left(
        m_j\left(X_i^{(j)}\right)+\varepsilon_i
    \right)^2
        \bigg|\mathbf{Z}_1
\right]\nonumber\\
&=
\frac{1}{\gamma_n}
         \left(
        m_j\left(X_1^{(j)}\right)+\varepsilon_1
    \right)^2
+
    \frac{1}{\gamma_n}
    \sum_{i=2}^{\gamma_n}
 \mathbb{E}\left[
         \left(
        m_j\left(X_i^{(j)}\right)+\varepsilon_i
    \right)^2
\right]\nonumber\\
&=
\frac{1}{\gamma_n}\label{lemma 2 1. binom}
         \left(
        m_j\left(X_1^{(j)}\right)+\varepsilon_1
    \right)^2
+
    \frac{\gamma_n-1}{\gamma_n}
 \mathbb{E}\left[
         \left(
        m_j\left(X_1^{(j)}\right)+\varepsilon_1
    \right)^2
\right]
\end{align}
almost surely.
The law of total expectation delivers
\begin{align}
&\frac{1}{\gamma_n}\nonumber
    \sum_{i=1}^{\gamma_n}
    \mathbb{E}\left[
    m_j\left(X_{\pi(i)}^{(j)}\right)^2\Big|\mathbf{Z}_1
    \right]\\
    &=
      \frac{1}{\gamma_n}
    \sum_{i=1}^{\gamma_n}\nonumber
    \mathbb{P}\left(
        \pi(i)=1
    \right)
    \mathbb{E}\left[
        m_j\left(X_{1}^{(j)}\right)^2\Big|\mathbf{Z}_1
    \right]\\
    &+
    \frac{1}{\gamma_n}
    \sum_{i=1}^{\gamma_n}
     \mathbb{P}\left(
        \pi(i)\neq 1
    \right)
    \mathbb{E}\left[
    m_j\left(X_{2}^{(j)}\right)^2\Big|\mathbf{Z}_1
    \right]
    \nonumber\\
    &=
    \nonumber
    \frac{1}{\gamma_n}
    \sum_{i=1}^{\gamma_n}
    \mathbb{P}\left(
        \pi(i)=1
    \right)
    m_j\left(X_1^{(j)}\right)^2+
    \frac{1}{\gamma_n}\nonumber
    \sum_{i=1}^{\gamma_n}
     \mathbb{P}\left(
        \pi(i)\neq 1
    \right)
    \mathbb{E}\left[
    m_j\left(X_2^{(j)}\right)^2
    \right]
\end{align}

almost surely, where the i.i.d. assumption was used in the first and second equation.
For the remaining part of (\ref{lemma 1 y-pi})
we can see

\begin{align}
 &=
    \frac{1}{\gamma_n-1}\label{lemma 2 binom 2}
    m_j\left(X_1^{(j)}\right)^2
    +
      \frac{\gamma_n -2}{\gamma_n-1}
    \mathbb{E}\left[
    m_j\left(X_1^{(j)}\right)^2
    \right]\\
    &\frac{2}{\gamma_n}
    \sum_{i=1}^{\gamma_n}
 \mathbb{E}\left[
 m_j\left(X_{\pi_j(i)}^{(j)}\right)
         \left(
        m_j\left(X_i^{(j)}\right)+\varepsilon_i
    \right)
        \bigg|\mathbf{Z}_1
\right]\nonumber\\
&=
 \frac{2}{\gamma_n}\nonumber
    \sum_{i=1}^{\gamma_n}
 \mathbb{E}\left[
 m_j\left(X_{\pi_j(i)}^{(j)}\right)
        m_j\left(X_i^{(j)}\right)
        \bigg|\mathbf{Z}_1
\right]
+
 \frac{2}{\gamma_n}\nonumber
    \sum_{i=1}^{\gamma_n}
 \mathbb{E}\left[
 m_j\left(X_{i}^{(j)}\right)\varepsilon_i
        \Big|\mathbf{Z}_1
\right]\\
&=
\frac{2}{\gamma_n}
 \bigg(\nonumber
    \sum_{i=1}^{\gamma_n}
 \mathbb{E}\left[
 m_j\left(X_{\pi_j(i)}^{(j)}\right)
        m_j\left(X_i^{(j)}\right)
        \bigg|\mathbf{Z}_1
\right]
+
 \mathbb{E}\left[
 m_j\left(X_{1}^{(j)}\right)\varepsilon_1
        \Big|\mathbf{Z}_1
\right]\\
&+
 \sum_{i=2}^{\gamma_n}\nonumber
 \mathbb{E}\left[
 m_j\left(X_{i}^{(j)}\right)\varepsilon_i
\right]
\bigg)\\
&=
\frac{2}{\gamma_n}
 \bigg(\nonumber
    \sum_{i=1}^{\gamma_n}
 \mathbb{E}\left[
 m_j\left(X_{\pi_j(i)}^{(j)}\right)
        m_j\left(X_i^{(j)}\right)
        \bigg|\mathbf{Z}_1
\right]
+
 m_j\left(X_{1}^{(j)}\right)\varepsilon_1
\end{align}

\begin{align}
&+
 \sum_{i=2}^{\gamma_n}
 \mathbb{E}\left[
 m_j\left(X_{i}^{(j)}\right)\right] \mathbb{E}\left[\varepsilon_i
\right]
\bigg)\\
&=
 \frac{2}{\gamma_n}
 \bigg(\nonumber
    \sum_{i=2}^{\gamma_n}
 \mathbb{E}\left[
 m_j\left(X_{\pi_j(i)}^{(j)}\right)
        m_j\left(X_i^{(j)}\right)
        \bigg|\mathbf{Z}_1
\right]\\
&
+
\mathbb{E}\left[\nonumber
 m_j\left(X_{\pi_j(1)}^{(j)}\right)
        m_j\left(X_1^{(j)}\right)
        \bigg|\mathbf{Z}_1
\right]+
 m_j\left(X_{1}^{(j)}\right)\varepsilon_1\bigg)\\
&=
 \frac{2}{\gamma_n}\nonumber
 \bigg(
    \sum_{i=2}^{\gamma_n}
 \mathbb{E}\left[
 m_j\left(X_{\pi_j(i)}^{(j)}\right)
        m_j\left(X_i^{(j)}\right)
        \bigg|\mathbf{Z}_1
\right]\\
&+
        m_j\left(X_1^{(j)}\right)
\mathbb{E}\left[\nonumber
 m_j\left(X_{\pi_j(1)}^{(j)}\right)
        \bigg|\mathbf{Z}_1
\right]
+
 m_j\left(X_{1}^{(j)}\right)
\varepsilon_1\bigg)
\end{align}
almost surely. Next we have analogously to (\ref{lemma 2 binom 2})
\begin{align}
    \mathbb{E}\left[
    m_j\left(X_{\pi(1)}^{(j)}\right)\Big|\mathbf{Z}_1
    \right]\nonumber
    &=
    \frac{1}{\gamma_n -1}\sum_{k=2}^{\gamma_n}
     \mathbb{E}\left[
 m_j\left(X_{k}^{(j)}\right)
        \bigg|\mathbf{Z}_1
\right]
+
\frac{1}{\gamma_n-1}
 m_j\left(X_{1}^{(j)}\right)\\
&=
\frac{\gamma_n-2}{\gamma_n -1}\nonumber
     \mathbb{E}\left[
 m_j\left(X_{1}^{(j)}\right)
\right]
+
\frac{1}{\gamma_n-1}
 m_j\left(X_{1}^{(j)}\right)
\end{align}
and analogously for $i\geq 2$
\begin{align}
     &\mathbb{E}\left[
 m_j\left(X_{\pi_j(i)}^{(j)}\right)
        m_j\left(X_i^{(j)}\right)
        \bigg|\mathbf{Z}_1
\right]\nonumber\\
&=
\sum_{k=1,k\neq i}^{\gamma_n}
\mathbb{P}\left(\nonumber
\pi_j(i)=k
\right)
     \mathbb{E}\left[
 m_j\left(X_{k}^{(j)}\right)
        m_j\left(X_i^{(j)}\right)
        \bigg|\mathbf{Z}_1
\right]\\
&=
\frac{1}{\gamma_n\!-\!1}\!\nonumber
\Bigg(
\mathbb{E}\left[
 m_j\left(X_{k}^{(j)}\right)
        m_j\left(X_i^{(j)}\right)
        \bigg|\mathbf{Z}_1
\right]\\
&\!+\!
\sum_{k=2,k\neq i}^{\gamma_n}\!\nonumber
\mathbb{E}\left[
 m_j\left(X_{k}^{(j)}\right)
        m_j\left(X_i^{(j)}\right)
        \bigg|\mathbf{Z}_1
\right]\Bigg)\\
&=
\frac{1}{\gamma_n\!-\!1}\!\nonumber
\Bigg(
 m_j\left(X_{1}^{(j)}\right)
\mathbb{E}\left[
        m_j\left(X_i^{(j)}\right)
        \bigg|\mathbf{Z}_1
\right]+
\sum_{k=2,k\neq i}^{\gamma_n}
\mathbb{E}\left[
 m_j\left(X_{k}^{(j)}\right)
        m_j\left(X_i^{(j)}\right)
\right]\Bigg)\\
&=
\frac{1}{\gamma_n\!-\!1}\!\nonumber
\left(
 m_j\left(X_{1}^{(j)}\right)
\mathbb{E}\left[
        m_j\left(X_i^{(j)}\right)
\right]
+
\sum_{k=2,k\neq i}^{\gamma_n}
\mathbb{E}\left[
 m_j\left(X_{k}^{(j)}\right)\right]\mathbb{E}\left[
        m_j\left(X_i^{(j)}\right)
\right]\right)\\
&=
\frac{1}{\gamma_n\!-\!1}\nonumber
\left(
 m_j\left(X_{1}^{(j)}\right)
\mathbb{E}\left[
        m_j\left(X_i^{(j)}\right)
\right]
+
(\gamma_n-2)
\mathbb{E}\left[
       m_j\left(X_1^{(j)}\right)
\right]^2\right)\\
&=
\frac{1}{\gamma_n-1}
     \mathbb{E}\left[
       m_j\left(X_1^{(j)}\right)
\right]
\nonumber
m_j\left(X_1^{(j)}\right)
+
\frac{\gamma_n\!-\!2}{\gamma_n\!-\!1}
\mathbb{E}\left[
       m_j\left(X_1^{(j)}\right)
\right]^2.
\end{align}
Hence,
\begin{align}
      &\frac{2}{\gamma_n}
    \sum_{i=2}^{\gamma_n}
 \mathbb{E}\left[
 m_j\left(X_{\pi_j(i)}^{(j)}\right)
         \left(
        m_j\left(X_i^{(j)}\right)+\varepsilon_i
    \right)
        \bigg|\mathbf{Z}_1
\right]\nonumber\\
&=
\frac{2}{\gamma_n}\bigg(
     \mathbb{E}\left[
       m_j\left(X_1^{(j)}\right)
\right]
\nonumber
m_j\left(X_1^{(j)}\right)
+
(\gamma_n-2)
\mathbb{E}\left[
       m_j\left(X_1^{(j)}\right)
\right]^2\\
&
+
    \frac{\gamma_n\!-\!2}{\gamma_n \!-\!1}
m_j\left(X_1^{(j)}\right)
\mathbb{E}\left[
m_j\left(X_1^{(j)}\right)
\right]+
\frac{1}{\gamma_n \!-\!1}
 m_j\left(X_{1}^{(j)}\right)^2\!
+
\! m_j\left(X_{1}^{(j)}\right)
\varepsilon_1
\bigg)\label{lemma 1 binom 3}
\end{align}
holds true.
Combining (\ref{lemma 1 y-pi}),(\ref{lemma 2 1. binom}),
(\ref{lemma 2 binom 2}) and (\ref{lemma 1 binom 3}) leads to
\begin{align}
   & \frac{1}{\gamma_n}
    \sum_{i=1}^{\gamma_n}
 \mathbb{E}\left[
        \big(
        Y_{i}-\widetilde{m}
        (\textbf{X}_{i}^{\pi_{j,t}})
        \big)^2
        \bigg|\mathbf{Z}_1
\right]\nonumber\\
&=
\nonumber
\frac{1}{\gamma_n}
         \left(
        m_j\left(X_1^{(j)}\right)+\varepsilon_1
    \right)^2
+
    \frac{\gamma_n-1}{\gamma_n}
 \mathbb{E}\left[
         \left(
        m_j\left(X_1^{(j)}\right)+\varepsilon_1
    \right)^2
\right]\\
&+
\frac{1}{\gamma_n-1}
    m_j\left(X_1^{(j)}\right)^2
    +
      \frac{\gamma_n -2}{\gamma_n-1}
    \mathbb{E}\left[
    m_j\left(X_1^{(j)}\right)^2
    \right]\nonumber\\
    &+
    \frac{2}{\gamma_n}
     \mathbb{E}\left[
       m_j\left(X_1^{(j)}\right)
\right]
\nonumber
m_j\left(X_1^{(j)}\right)
+
2\frac{\gamma_n-2}{\gamma_n}
\mathbb{E}\left[
       m_j\left(X_1^{(j)}\right)
\right]^2\\
&+
2m_j\left(X_1^{(j)}\right)
\frac{\gamma_n-2}{\gamma_n(\gamma_n -1)}
     \mathbb{E}\left[\nonumber
 m_j\left(X_{1}^{(j)}\right)
\right]\\
&+
\frac{2}{\gamma_n(\gamma_n-1)}
m_j\left(X_1^{(j)}\right)^2
+
\frac{2}{\gamma_n}
 m_j\left(X_{1}^{(j)}\right)
\varepsilon_1. \label{lemma1 1.teil}
\end{align}
Note that there are several constant summands in (\ref{lemma1 1.teil}), that
vanish inside the variance. Hence (\ref{Lemma 1 h.bed.erw}), (\ref{lemma 1 2. teil}) and (\ref{lemma1 1.teil})
result in
\begin{align*}
\zeta_{1,\gamma_n}=&
\mathbb{V}\text{ar}\left(\mathbb{E}\left[
h|\mathbf{Z}_1
\right]\right)
=
\mathbb{V}\text{ar}\bigg(
\frac{1}{\gamma_n}
         \left(
        m_j\left(X_1^{(j)}\right)+\varepsilon_1
    \right)^2
    +
    \frac{1}{\gamma_n-1}
    m_j\left(X_1^{(j)}\right)^2\\
    &
+
\frac{2}{\gamma_n}
     \mathbb{E}\left[
       m_j\left(X_1^{(j)}\right)
\right]
\nonumber
m_j\left(X_1^{(j)}\right)
\frac{2(\gamma_n-2)}{\gamma_n(\gamma_n -1)}
m_j\left(X_1^{(j)}\right)
     \mathbb{E}\left[
 m_j\left(X_{1}^{(j)}\right)
\right]\\
&+
\frac{2}{\gamma_n(\gamma_n-1)}
m_j\left(X_1^{(j)}\right)^2
+
\frac{2}{\gamma_n}
 m_j\left(X_{1}^{(j)}\right)
\varepsilon_1
-\frac{1}{\gamma_n}\varepsilon_1^2
\bigg),
\end{align*}
which implies with (A3), (A4) and the dominated convergence theorem,
that $\gamma_n^2\zeta_{1,\gamma_n}$ is
asymptotically 
equivalent to
\begin{align*}
    &\mathbb{V}\text{ar}\bigg(
         \left(
        m_j\left(X_1^{(j)}\right)+\varepsilon_1
    \right)^2
    +
    m_j\left(X_1^{(j)}\right)^2
+
4m_j\left(X_1^{(j)}\right)
     \mathbb{E}\left[
 m_j\left(X_{1}^{(j)}\right)
\right]\\
&+
 2m_j\left(X_{1}^{(j)}\right)
\varepsilon_1
-\varepsilon_1^2
\bigg)\\
&=
    \mathbb{V}\text{ar}\bigg(
   2 m_j\left(X_1^{(j)}\right)^2
+
4m_j\left(X_1^{(j)}\right)
     \mathbb{E}\left[
 m_j\left(X_{1}^{(j)}\right)
\right]
+
4 m_j\left(X_{1}^{(j)}\right)
\varepsilon_1
\bigg)\\
&=
      \mathbb{V}\text{ar}\bigg(
     m_j\left(X_1^{(j)}\right) \left(
   2 m_j\left(X_1^{(j)}\right)
+
4
     \mathbb{E}\left[
 m_j\left(X_{1}^{(j)}\right)
\right]
\right)
+
4 m_j\left(X_{1}^{(j)}\right)
\varepsilon_1
\bigg)=:\widetilde\zeta.\\
\end{align*}
Due to the independence of $\varepsilon_1$ and $\mathbf{X}_1$,
\begin{align*}
    &\mathbb{C}\text{ov}\left(
     m_j\left(X_1^{(j)}\right) \left(
   2 m_j\left(X_1^{(j)}\right)
+
4
     \mathbb{E}\left[
 m_j\left(X_{1}^{(j)}\right)
\right]
\right)
    ,
    4 m_j\left(X_{1}^{(j)}\right)
\varepsilon_1
    \right)\\
    &=
    \mathbb{E}\left[
      4m_j\left(X_1^{(j)}\right)^2 \left(
   2 m_j\left(X_1^{(j)}\right)
+
4
     \mathbb{E}\left[
 m_j\left(X_{1}^{(j)}\right)
\right]
\right)\varepsilon_1
    \right]\\
    &=
 \mathbb{E}\left[
      4m_j\left(X_1^{(j)}\right)^2 \left(
   2 m_j\left(X_1^{(j)}\right)
+
4
     \mathbb{E}\left[
 m_j\left(X_{1}^{(j)}\right)
\right]
\right)\right]
\mathbb{E}\left[\varepsilon_1
    \right]=0
\end{align*}
holds true and hence
\begin{align*}
    \widetilde \zeta
    &=
     \mathbb{V}\text{ar}\bigg(
     m_j\left(X_1^{(j)}\right) \left(
   2 m_j\left(X_1^{(j)}\right)
+
4
     \mathbb{E}\left[
 m_j\left(X_{1}^{(j)}\right)
\right]
\right)
+
4 m_j\left(X_{1}^{(j)}\right)
\varepsilon_1
\bigg)\\
&=
 \mathbb{V}\text{ar}\bigg(
     m_j\left(X_1^{(j)}\right) \left(
   2 m_j\left(X_1^{(j)}\right)
\!+\!
4
     \mathbb{E}\left[
 m_j\left(X_{1}^{(j)}\right)
\right]
\right)
\bigg)
\!+\!
 \mathbb{V}\text{ar}\bigg(
4 m_j\left(X_{1}^{(j)}\right)
\varepsilon_1
\bigg)\\
&\geq
 16\mathbb{V}\text{ar}\bigg(
m_j\left(X_{1}^{(j)}\right)
\varepsilon_1
\bigg)
=
16\mathbb{E}\left[
m_j\left(X_1^{(j)}\right)^2
\varepsilon_1^2
\right]=16\sigma^2
\mathbb{E}\left[
m_j\left(X_1^{(j)}\right)^2
\right]>0,
\end{align*}
which completes the proof.
\end{proof}

\begin{lemma}\label{Lemma 2}
    Let (A3) and (A4) be fulfilled. Then there is a $\widetilde K>0$, such that $\zeta_{\gamma_n}=\mathbb{V}\text{ar}(h)<\widetilde K$ for all $n\in\mathbb{N}$.
\end{lemma}
\begin{proof}
    At first we apply Jensens inequality for sums and by noting that $(a+b)^2\leq2a^2+2b^2$ for all $a,b\in\mathbb{R}$ we recieve
    \begin{align*}
            \mathbb{E}\left[h_{\gamma_n}^2\right]
            &=
            \mathbb{E}\left[\left( 
            \frac{1}{\gamma_n}\sum_{i=1}^{\gamma_n}
    \Big\{
        \big(
        Y_{i}-\widetilde{m}
        (\textbf{X}_{i}^{\pi_{j,t}})
        \big)^2
        -
        \big(
        Y_{i}-\widetilde{m}
        (\textbf{X}_{i})
        \big)^2
    \Big\}\right)^2\right]\\
    &\leq
    \frac{1}{\gamma_n}\sum_{i=1}^{\gamma_n}
     \mathbb{E}\left[
    \left( \Big\{
        \big(
        Y_{i}-\widetilde{m}
        (\textbf{X}_{i}^{\pi_{j,t}})
        \big)^2
        -
        \big(
        Y_{t_i}-\widetilde{m}
        (\textbf{X}_{i})
        \big)^2
    \Big\}\right)^2\right]\\
    &\leq
     \frac{2}{\gamma_n}\sum_{i=1}^{\gamma_n}
     \mathbb{E}\left[
        \big(
        Y_{i}-\widetilde{m}
        (\textbf{X}_{i}^{\pi_{j,t}})
        \big)^4\right]
        +
        \frac{2}{\gamma_n}\sum_{i=1}^{\gamma_n}
     \mathbb{E}\left[
        \big(
        Y_{i}-\widetilde{m}
        (\textbf{X}_{i})
        \big)^4\right]\\
&=
   \frac{2}{\gamma_n}\sum_{i=1}^{\gamma_n}
     \mathbb{E}\left[
         \left(
        m_j\left(X_i^{(j)}\right)+\varepsilon_i-m_j\left(X_{\pi_j(i)}^{(j)}\right)
    \right)^4\right]
    +
       \frac{2}{\gamma_n}\sum_{i=1}^{\gamma_n}
     \mathbb{E}\left[
        \varepsilon_i
        ^4\right].
    \end{align*}
    Because of (A4), it's sufficient to show that for all $i\in\{1, \ldots,n\}$,
    \begin{align*}
          \mathbb{E}\left[
         \left(
        m_j\left(X_i^{(j)}\right)+\varepsilon_i-m_j\left(X_{\pi_j(i)}^{(j)}\right)
    \right)^4\right]
    \end{align*}
        is finite.
        Note that $\vert m_j(x)\vert <K$ for all $x\in[0,1]$, due to (A3). Hence, by the independence of $\mathbf{X}_i$ and $\varepsilon_i$ and (A4), we get
        \begin{align*}
                      &\mathbb{E}\left[
         \left(
        m_j\left(X_i^{(j)}\right)+\varepsilon_i-m_j\left(X_{\pi_j(i)}^{(j)}\right)
    \right)^4\right]\\
              &=
              \mathbb{E}\left[
         \left(
        m_j\left(X_i^{(j)}\right)-m_j\left(X_{\pi_j(i)}^{(j)}\right)
    \right)^4\right]
    +
      4\mathbb{E}\left[
         \left(
        m_j\left(X_i^{(j)}\right)-m_j\left(X_{\pi_j(i)}^{(j)}\right)
    \right)^3\varepsilon_i\right]\\
    &+
      6\mathbb{E}\left[
         \left(
        m_j\left(X_i^{(j)}\right)-m_j\left(X_{\pi_j(i)}^{(j)}\right)
    \right)^2\varepsilon_i^2\right]
    +
      4\mathbb{E}\left[
         \left(
        m_j\left(X_i^{(j)}\right)-m_j\left(X_{\pi_j(i)}^{(j)}\right)
    \right)\varepsilon_i^3\right]\\
    &+
    \nonumber
    \mathbb{E}\left[
    \varepsilon_i^4
    \right]\\
    &=
       \mathbb{E}\left[
         \left(
        m_j\left(X_i^{(j)}\right)-m_j\left(X_{\pi_j(i)}^{(j)}\right)
    \right)^4\right]
    +
    6\sigma^2
    \mathbb{E}\left[
         \left(
        m_j\left(X_i^{(j)}\right)-m_j\left(X_{\pi_j(i)}^{(j)}\right)
    \right)^2\right] \\
    &+
    \nonumber
    \mathbb{E}\left[
    \varepsilon_i^4
    \right]\\
    &\leq
    16K^4+24\sigma^2K^2+
    \mathbb{E}\left[
    \varepsilon_i^4
    \right]<\infty.
        \end{align*}
        Therefore the proof is completed.
\end{proof}

\begin{lemma}\label{Lemma 3}
Let $\mathbf{X}_{j,1}=\left[X_1^{(1)},\ldots,X_1^{(j-1)},X,X_1^{(j+1)},\ldots,X_1^{(p)}\right]^\top$, where $X$ is an independent copy of $X_1^{(j)}$, then

    \begin{align*}
        \mathbb{E}\left[h\right]=
        \mathbb{E}\left[(Y_1-\widetilde m(\mathbf{X}_{j,1}))^2\right]-\sigma^2=:I(j)
    \end{align*}
    if (A1)-(A4) is fulfilled
\end{lemma}

\begin{proof}
    Since $(\mathbf{X}_1,Y_1),\ldots,(\mathbf{X}_n,Y_n)$ are i.i.d., we can conclude from (A2), that
    \begin{align*}
            &\mathbb{E}\left[h((\mathbf{X}_1,Y_1),\ldots,(\mathbf{X}_n,Y_n))\right]\nonumber\\
            &=
            \mathbb{E}\left[\frac{1}{\gamma_n}\sum_{i=1}^{\gamma_n}
    \Big\{
        \big(
        Y_{i}-\widetilde{m}
        (\textbf{X}_{i}^{\pi_{j,t}})
        \big)^2
        -
        \big(
        Y_{i}-\widetilde{m}
        (\textbf{X}_{i})
        \big)^2
    \Big\}\right]\\
    &=
          \mathbb{E}\left[
        \big(
        Y_{1}-\widetilde{m}
        (\textbf{X}_{1}^{\pi_{j,t}})
        \big)^2    \right]
        -
         \mathbb{E}\left[
        \varepsilon_1^2
        \right]
        =
                 \mathbb{E}\left[
        \big(
        Y_{1}-\widetilde{m}
        (\textbf{X}_{1}^{\pi_{j,t}})
        \big)^2    \right]
        -
        \sigma^2\\
        &=
         \mathbb{E}\left[
         \left(
        m_j\left(X_1^{(j)}\right)+\varepsilon_1-m_j\left(X_{\pi_j(1)}^{(j)}\right)
    \right)^2\right]-\sigma^2
    \end{align*}
    holds true. With (A1) and the law of total expectation we get
    \begin{align*}
        &\mathbb{E}\left[
         \left(
        m_j\left(X_1^{(j)}\right)+\varepsilon_1-m_j\left(X_{\pi_j(1)}^{(j)}\right)
    \right)^2\right]\\
    &=
    \sum_{k=2}^{\gamma_n}\mathbb{P}\left(\pi_j(1)=k\right)
    \mathbb{E}\left[
         \left(
        m_j\left(X_1^{(j)}\right)+\varepsilon_1-m_j\left(X_{k}^{(j)}\right)
    \right)^2\right]\\
    &=
    \sum_{k=2}^{\gamma_n}
    \frac{1}{\gamma_n-1}
     \mathbb{E}\left[
         \left(
        m_j\left(X_1^{(j)}\right)+\varepsilon_1-m_j\left(X\right)
    \right)^2\right]=
    \mathbb{E}\left[
         \left(
        Y_1-\widetilde m(\mathbf{X}_{j,1})
    \right)^2\right].
    \end{align*}
\end{proof}

\subsection*{Proof of Theorem 3}
\begin{proof}
   The proof idea for the theorem is applying Theorem 2.1. The main task of the proof will be to check whether
    \begin{align}
        \frac{\mathbb{E}\left[\vert h_{\gamma_n}-I(j) \vert^{2k}
    \right]}{
    \mathbb{E}\left[\vert h_{\gamma_n}-I(j) \vert^{k} \label{theorem 2.2 ljapunov}
    \right]^2}
    \end{align}
    is uniformly bounded in $n$ for $k=2,3$.
Since we are dealing with permutations we start by establishing some properties.
The elements $\pi_j\in\mathcal{V}_n$ are so called derangements.
The number of derangements of an $\gamma_n$-element set is given as 
\begin{align*}
    \vert\mathcal{V}_n\vert=!\gamma_n:=\gamma_n!\sum_{i=0}^n\frac{(-1)^i}{i!}
    =
    \left[\frac{\gamma_n!}{e}\right]
    =
    \left\lfloor
    \frac{\gamma_n!+1}{e}
    \right\rfloor.
\end{align*}
For $i\neq k$
\begin{align}
        &\mathbb{P}\left(\nonumber\nonumber
\left\{\pi_j(k)=i\right\} \cup \left\{\pi_j(i)=k\right\} 
    \right)\\
    &=
    \nonumber
     \mathbb{P}\left(
\pi_j(k)=i
    \right)
    +
     \mathbb{P}\left(
\pi_j(i)=k 
    \right)
    -
            \mathbb{P}\left(
\left\{\pi_j(k)=i\right\} , \left\{\pi_j(i)=k\right\} 
    \right)\\
    &=
    \frac{2}{\gamma_n-1}
    -
    \frac{!(\gamma_n-2)}{!\gamma_n}=:p_{\gamma_n}\label{theorem 2.2 permu vereinigung}
\end{align}
and
\begin{align}
     \mathbb{P}\left(\nonumber
\pi_j(k)=i, \pi_j(i)\neq k
     \right)
     &=
          \mathbb{P}\left(
\pi_j(k)=i
     \right)\label{theorem 2.2 permu schnitt}
     -
          \mathbb{P}\left(
\pi_j(k)=i, \pi_j(i)= k
     \right)\\
     &=
     \frac{1}{\gamma_n-1}-\frac{!(\gamma_n-2)}{!\gamma_n}
     =\frac{1}{\gamma_n}+\mathcal{O}\left(\frac{1}{\gamma_n^2}\right).
\end{align}
Using (A2), \Cref{Lemma 3}, (\ref{theorem 2.2 permu vereinigung}) and the i.i.d. assumption
as well as $\mathbb{E}\left[\varepsilon_i\right]=0$ $(*)$ leads to
\begin{align*}
    &\mathbb{E}\left[\left|h_{\gamma_n}-I(j)\right|^2\right] \\
    &= \mathbb{E}\left[\left|\frac{1}{\gamma_n}\sum_{i=1}^{\gamma_n}\left(Y_i - \widetilde{m}\left(X_i^{\pi_{j}}\right)\right)^2 - \varepsilon_i^2 - I(j)\right|^2\right] \\
    &= \mathbb{E}\Bigg[\bigg|\frac{1}{\gamma_n} \sum_{i=1}^{\gamma_n} \Bigg\{
        \left(m_j\left(X_i^{(j)}\right) - m_j\left(X_{\pi(i)}^{(j)}\right)\right)^2\\
        &\quad\quad\quad\quad\quad\quad\quad\quad\quad\quad\quad\quad\quad\quad
        + 2\varepsilon_i\left(m_j\left(X_i^{(j)}\right) - m_j\left(X_{\pi(i)}^{(j)}\right)\right)
    \Bigg\} - I(j)\bigg|^2\Bigg]\\
        &\overset{(*)}{=}
     \frac{1}{\gamma_n}\sum_{i=1}^{\gamma_n}\mathbb{E}\Bigg[
     \Bigg\{
     \left(\nonumber
     m_j\left(X_i^{(j)}\right)
    - m_j\left(X_{\pi(i)}^{(j)}\right)\right)^2\\
        &\quad\quad\quad\quad\quad\quad\quad\quad\quad\quad\quad\quad\quad\quad+2\varepsilon_i\left(\nonumber
     m_j\left(X_i^{(j)}\right)
    - m_j\left(X_{\pi(i)}^{(j)}\right)\right)
    -I(j)\Bigg\}^2\Bigg]\\
            &+
    \nonumber
    \frac{1}{\gamma_n^2}\sum_{i=1}^{\gamma_n}\sum_{k=1,k\neq i}^{\gamma_n}
    \mathbb{E}\Bigg[
      \left\{
     \left(\nonumber
     m_j\left(X_i^{(j)}\right)
    - m_j\left(X_{\pi(i)}^{(j)}\right)\right)^2
    -I(j)\right\}\\
          &\quad\quad\quad\quad\quad\quad\quad\quad\quad\quad\quad\quad\quad\quad\
          \left\{
     \left(\nonumber
     m_j\left(X_k^{(j)}\right)
    - m_j\left(X_{\pi(k)}^{(j)}\right)\right)^2
    -I(j)\right\}
    \Bigg]\\[0.5em]&=
    \frac{1}{\gamma_n}
    \mathbb{V}\text{ar}\left[
     \left(\nonumber
     m_j\left(X_1^{(j)}\right)
    - m_j\left(X_{2}^{(j)}\right)\right)^2+2\varepsilon_1\left(\nonumber
     m_j\left(X_1^{(j)}\right)
    - m_j\left(X_{2}^{(j)}\right)\right)
    \right]\\
        &+
    \nonumber
    \frac{\gamma_n(\gamma_n-1)}{\gamma_n^2}
       \mathbb{P}\left(\nonumber
\pi_j(k)\neq i, \pi_j(i)\neq k
    \right)
    \mathbb{E}\left[
      \left(m_j\left(X_1^{(j)}\right)
    - m_j\left(X_{2}^{(j)}\right)\right)^2
    -I(j)
    \right]^2\\
    &+
    \frac{\gamma_n(\gamma_n-1)}{\gamma_n^2}
        \mathbb{P}\left(\nonumber
\pi_j(k)\neq i , \pi_j(i)=k
    \right)\\
        &
        \cdot\mathbb{E}\left[
      \left\{\!
     \left(\nonumber
     m_j\left(X_1^{(j)}\right)
    \!-\! m_j\left(X_{2}^{(j)}\right)\right)^2
    \!-\!I(j)\right\}\!
          \left\{\!
     \left(\nonumber
     m_j\left(X_1^{(j)}\right)
    \!-\! m_j\left(X_{3}^{(j)}\right)\right)^2
    \!-\!I(j)\right\}
    \right]\\
    &+
        \nonumber \mathcal{O}\left(\frac{1}{\gamma_n^2}\right)\\
            &=
    \frac{1}{\gamma_n}
    \mathbb{V}\text{ar}\left[
     \left(\nonumber
     m_j\left(X_1^{(j)}\right)
    - m_j\left(X_{2}^{(j)}\right)\right)^2+2\varepsilon_1\left(\nonumber
     m_j\left(X_1^{(j)}\right)
    - m_j\left(X_{2}^{(j)}\right)\right)
    \right]\\
\end{align*}

\newpage
\begin{align}
    &+
    \frac{\gamma_n(\gamma_n-1)}{\gamma_n^2}\frac{2}{\gamma_n-1}
        \mathbb{E}\Bigg[
      \left\{
     \left(\nonumber
     m_j\left(X_1^{(j)}\right)
    - m_j\left(X_{2}^{(j)}\right)\right)^2
    -I(j)\right\}\\
    &
    \quad\quad\quad\quad\quad
    \quad\quad\quad\quad\quad
    \quad\quad
          \left\{
     \left(\nonumber
     m_j\left(X_1^{(j)}\right)
    - m_j\left(X_{3}^{(j)}\right)\right)^2
    -I(j)\right\}
    \Bigg]\\
    &+
    \nonumber
    \mathcal{O}\left(\frac{1}{\gamma_n^2}\right)\\
     &=
    \frac{1}{\gamma_n}
    \mathbb{V}\text{ar}\left[
     \left(\nonumber
     m_j\left(X_1^{(j)}\right)
    - m_j\left(X_{2}^{(j)}\right)\right)^2+2\varepsilon_1\left(\nonumber
     m_j\left(X_1^{(j)}\right)
    - m_j\left(X_{2}^{(j)}\right)\right)
    \right] + 0\\
    &+
    \frac{2}{\gamma_n}
    \mathbb{C}\text{ov}\left[ \left(\nonumber
     m_j\left(X_1^{(j)}\right)
    - m_j\left(X_{2}^{(j)}\right)\right)^2,
     \left(\nonumber
     m_j\left(X_1^{(j)}\right)
    - m_j\left(X_{3}^{(j)}\right)\right)^2\right]\\
    &+
    \mathcal{O}\left(\frac{1}{\gamma_n}\right).
\end{align}

\noindent Without loss of generality we can assume $\mathbb{E}\left[m_j\left(X_1^{(j)}\right)\right]=0$, because
as in previous calculations we get

\begin{align*}
    h_{\gamma_n}
    &= 
    \frac{1}{\gamma_n}\sum_{i=1}^{\gamma_n}
    \left\{
    \left(Y_i
    -\widetilde m\left(X_i^{\pi_{j}}\right)\right)^2-\varepsilon_i^2
    \right\}
    \end{align*}
    \begin{align*}
    &=
      \frac{1}{\gamma_n}\sum_{i=1}^{\gamma_n}
    \left\{
    \left(m_j\left(X_i^{(j)}\right)+\varepsilon_i
    - m_j\left(X_{\pi(i)}^{(j)}\right)\right)^2-\varepsilon_i^2
    \right\}.
\end{align*}
Thus, the expectation of $m_j\left(X_i^{(j)}\right)$ and $m_j\left(X_{\pi(i)}^{(j)}\right)$ are cancelling each other.
For random variables $U,V$ and $W$ with $U,V\indep W$
\begin{align}
     \mathbb{C}\text{ov}\left(UW,V\right)=
     \mathbb{E}\left[UWV\right]-\mathbb{E}\left[UW\right]\mathbb{E}\left[V\right]
     =
     \mathbb{E}\left[W\right]\mathbb{C}\text{ov}\left(U,V\right)\label{theroem 2.2 Kovarianzregel}
\end{align}
holds,
which implies
\begin{align}
       &\mathbb{C}\text{ov}\left[ \left(\nonumber
     m_j\left(X_1^{(j)}\right)
    - m_j\left(X_{2}^{(j)}\right)\right)^2,
     \left(\nonumber
     m_j\left(X_1^{(j)}\right)
    - m_j\left(X_{3}^{(j)}\right)\right)^2\right]\\
    &=
           \mathbb{C}\text{ov}\left[\nonumber
     m_j\left(X_1^{(j)}\right)^2\!,\!
    m_j\left(X_1^{(j)}\right)^2\right]
    -4
     \mathbb{C}\text{ov}\left[\nonumber
     m_j\left(X_1^{(j)}\right)^2,
     m_j\left(X_1^{(j)}\right)
    m_j\left(X_2^{(j)}\right)\right]\\
    &+
    4 \mathbb{C}\text{ov}\left[\nonumber
      m_j\left(X_1^{(j)}\right)
    m_j\left(X_2^{(j)}\right)\!,\!
     m_j\left(X_1^{(j)}\right)
    m_j\left(X_2^{(j)}\right)\right]\\
    &=
    \nonumber
    \mathbb{V}\text{ar}\left[ m_j\left(X_1^{(j)}\right)^2\right]
\end{align}
due to the i.i.d. assumption. Using (\ref{theroem 2.2 Kovarianzregel}) 
and $\mathbb{E}\left[\varepsilon_1\right]=0$ leads to
\begin{align}
    &\mathbb{V}\text{ar}\left[
     \left(\nonumber
     m_j\left(X_1^{(j)}\right)
    - m_j\left(X_{2}^{(j)}\right)\right)^2+2\varepsilon_1\left(\nonumber
     m_j\left(X_1^{(j)}\right)
    - m_j\left(X_{2}^{(j)}\right)\right)
    \right]\\
     &=
     \mathbb{V}\text{ar}\left[
     \left(\nonumber
     m_j\left(X_1^{(j)}\right)
    - m_j\left(X_{2}^{(j)}\right)\right)^2\right]
    +
    \mathbb{V}\text{ar}\left[2\varepsilon_1\left(\nonumber
     m_j\left(X_1^{(j)}\right)
    - m_j\left(X_{2}^{(j)}\right)\right)
    \right]\\
    &+
    2
    \mathbb{C}\text{ov}\left[
         \left(\nonumber
     m_j\left(X_1^{(j)}\right)
    - m_j\left(X_{2}^{(j)}\right)\right)^2,
    2\varepsilon_1\left(\nonumber
     m_j\left(X_1^{(j)}\right)
    - m_j\left(X_{2}^{(j)}\right)\right)
    \right]\\
    &=
    \mathbb{V}\text{ar}\left[
     \left(\nonumber
     m_j\left(X_1^{(j)}\right)
    - m_j\left(X_{2}^{(j)}\right)\right)^2
    \right]
    +
   4 \mathbb{V}\text{ar}\left[
    \varepsilon_1\left(\nonumber
     m_j\left(X_1^{(j)}\right)
    - m_j\left(X_{2}^{(j)}\right)\right)
    \right].
\end{align}
Hence,
\begin{align}
    &\mathbb{E}\left[\left|h_{\gamma_n}-I(j)\right|^2\right]\nonumber\\
    &=
    \frac{2}{\gamma_n}
    \mathbb{V}\text{ar}\left[ m_j\left(X_1^{(j)}\right)^2\right]
    +
    \frac{1}{\gamma_n}
        \mathbb{V}\text{ar}\left[
     \left(\nonumber
     m_j\left(X_1^{(j)}\right)
    - m_j\left(X_{2}^{(j)}\right)\right)^2
    \right]\\
    &+
    \frac{4}{\gamma_n}
    \mathbb{V}\text{ar}\left[
    \varepsilon_1\left(\label{theorem 2.2 2. Moment}
     m_j\left(X_1^{(j)}\right)
    - m_j\left(X_{2}^{(j)}\right)\right)
    \right]
    +
      \mathcal{O}\left(\frac{1}{\gamma_n^2}\right).
\end{align}
Note that
\begin{align}
    \mathbb{E}\left[\left|h_{\gamma_n}-I(j)\right|^2\right]\geq
      \frac{2}{\gamma_n}
        \mathbb{V}\text{ar}\left[
     m_j\left(X_1^{(j)}\right)^2
    \right]
    +
       \mathcal{O}\left(\frac{1}{\gamma_n^2}\right).
       \label{theorem 2.2 zentr. moment größer als var}
\end{align}
   Since $\mathbb{V}\text{ar}\left[
   \nonumber
     m_j\left(X_1^{(j)}\right)
    \right]>0,$ it also holds true that  $\mathbb{V}\text{ar}\left[
   \nonumber
     m_j\left(X_1^{(j)}\right)
    \right]>0$ and with (\ref{theorem 2.2 zentr. moment größer als var}) there is a constant $c_1>0$
    such that $\mathbb{E}\left[\left|h_{\gamma_n}-I(j)\right|^2\right]\sim\frac{c_1}{\gamma_n}$
    and $\mathbb{E}\left[\left|h_{\gamma_n}-I(j)\right|^2\right]^2\sim\frac{c_1^2}{\gamma_n^2}$.
    With similar calculations as before, we get
    
\begin{align*}
        &\mathbb{E}\left[\left|h_{\gamma_n}-I(j)\right|^4\right]\nonumber\\
    &=
    \nonumber
        \mathbb{E}\Bigg[\bigg|\frac{1}{\gamma_n}\sum_{i=1}^{\gamma_n}
     \bigg\{
     \left(\nonumber
     m_j\left(X_i^{(j)}\right)
    - m_j\left(X_{\pi(i)}^{(j)}\right)\right)^2\\
    &\quad\quad\quad\quad\quad\quad\quad\quad\quad\quad\quad\quad\quad\quad
    +2\varepsilon_i\left(\nonumber
     m_j\left(X_i^{(j)}\right)
    - m_j\left(X_{\pi(i)}^{(j)}\right)\right)
\bigg\}
    -I(j)\bigg|^4\Bigg]\\[0.5em]
       &= \nonumber \mathcal{O}\left(\frac{1}{\gamma_n^2}\right) + \frac{1}{\gamma_n^4}\sum_{i_1=1}^{\gamma_n}\underset{i_2 \neq i_1}{\sum_{i_2=1}^{\gamma_n}} \underset{i_3 \neq i_2 \neq i_1}{\sum_{i_3=1}^{\gamma_n}} \\
          &\mathbb{E}\bigg[ \!\bigg\{  \!\left( m_j\left(X_{i_1}^{(j)}\right) - m_j\left(X_{\pi(i_1)}^{(j)}\right)\right)^2
   \nonumber
   +2\varepsilon_{i_1}\left( m_j\left(X_{i_1}^{(j)}\right) - m_j\left(X_{\pi(i_1)}^{(j)}\right)\right) -I(j)\bigg\}^2
\end{align*}

\begin{align}
   &\nonumber \left\{  \!\left( m_j\left(X_{i_2}^{(j)}\right) - m_j\left(X_{\pi(i_2)}^{(j)}\right)\right)^2+2\varepsilon_{i_2}\left( m_j\left(X_{i_2}^{(j)}\right) - m_j\left(X_{\pi(i_2)}^{(j)}\right)\right) -I(j)\right\}\\ & \nonumber \left\{  \!\left( m_j\left(X_{i_3}^{(j)}\right) - m_j\left(X_{\pi(i_3)}^{(j)}\right)\right)^2+2\varepsilon_{i_3}\left( m_j\left(X_{i_3}^{(j)}\right) - m_j\left(X_{\pi(i_3)}^{(j)}\right)\right) -I(j)\right\}\bigg]\\
   &+ \frac{1}{\gamma_n^4}\sum_{i_1=1}^{\gamma_n}\underset{i_2 \neq i_1}{\sum_{i_2=1}^{\gamma_n}} \underset{i_3 \neq i_2 \neq i_1}{\sum_{i_3=1}^{\gamma_n}} \underset{i_4 \neq i_1 \neq i_3}{\sum_{i_4=1;i_4 \neq i_2}^{\gamma_n}} \nonumber\\
   &\mathbb{E}\bigg[\! \left\{ \!\left( m_j\left(X_{i_1}^{(j)}\right) - m_j\left(X_{\pi(i_1)}^{(j)}\right)\right)^2+2\varepsilon_{i_1} \!\left( \! m_j\left(X_{i_1}^{(j)}\right) - m_j\left(X_{\pi(i_1)}^{(j)}\right)\right) -I(j)\right\}\nonumber\\ 
   &\nonumber \left\{ \! \left( m_j\left(X_{i_2}^{(j)}\right) - m_j\left(X_{\pi(i_2)}^{(j)}\right)\right)^2+2\varepsilon_{i_2}\!\left( m_j\left(X_{i_2}^{(j)}\right) - m_j\left(X_{\pi(i_2)}^{(j)}\right)\right) -I(j)\right\}\\ &\nonumber \left\{\! \left( m_j\left(X_{i_3}^{(j)}\right) - m_j\left(X_{\pi(i_3)}^{(j)}\right)\right)^2+2\varepsilon_{i_3}\!\left( m_j\left(X_{i_3}^{(j)}\right) - m_j\left(X_{\pi(i_3)}^{(j)}\right)\right) -I(j)\right\}\\ &\nonumber \left\{ \!\left( m_j\left(X_{i_4}^{(j)}\right) - m_j\left(X_{\pi(i_4)}^{(j)}\right)\right)^2+2\varepsilon_{i_4}\!\left( m_j\left(X_{i_4}^{(j)}\right) - m_j\left(X_{\pi(i_4)}^{(j)}\right)\right) -I(j)\right\} \bigg]\\[0.5em]
        &=
        \nonumber
    \mathcal{O}\left(\frac{1}{\gamma_n^2}\right)
    +
    \frac{3!\gamma_n(\gamma_n-1)(\gamma_n-2)}{\gamma_n^4}\nonumber\\
    &\mathbb{E}\bigg[
    \!
     \left\{    \!
     \left(\nonumber
     m_j\left(X_1^{(j)}\right)
    - m_j\left(X_{\pi(1)}^{(j)}\right)\right)^2+2\varepsilon_i    \!\left(\nonumber
     m_j\left(X_1^{(j)}\right)
    - m_j\left(X_{\pi(1)}^{(j)}\right)\right)
    -I(j)\right\}^2\\
        &
    \nonumber
     \left\{    \!
     \left(\nonumber
     m_j\left(X_{2}^{(j)}\right)
    - m_j\left(X_{\pi(2)}^{(j)}\right)\right)^2
    \!-\!I(j)\right\}    
         \left\{
     \left(\nonumber
     m_j\left(X_{3}^{(j)}\right)
    - m_j\left(X_{\pi(3)}^{(j)}\right)\right)^2
        \!-    \!I(j)\right\}      \!
    \bigg]\\
    &+
    \frac{4!\gamma_n(\gamma_n\!-\!1)(\gamma_n\!-\!2)(\gamma_n\!-\!3)}{\gamma_n^4}\nonumber\\
    &
     \mathbb{E}\bigg[
     \!\left\{\!
     \left(\nonumber
     m_j\left(X_1^{(j)}\right)
    - m_j\left(X_{\pi(1)}^{(j)}\right)\right)^2
    \!-\!I(j)\right\}\!
     \left\{
     \left(\nonumber
     m_j\left(X_2^{(j)}\right)
    - m_j\left(X_{\pi(2)}^{(j)}\right)\right)^2
    \!-\!I(j)\!\right\}\\
      &
    \nonumber
     \left\{\!
     \left(\nonumber
     m_j\left(X_3^{(j)}\right)
    - m_j\left(X_{\pi(3)}^{(j)}\right)\right)^2
    \!-\!I(j)\right\}
     \left\{
     \left(\nonumber
     m_j\left(X_{4}^{(j)}\right)
    - m_j\left(X_{\pi(4)}^{(j)}\right)\right)^2
    \!-\!I(j)\right\}    \!
    \bigg],
\end{align}
where the second equality holds true because all summands of the remaining $\mathcal{O}(\gamma_n^2)$ summand are uniformly bounded.
With the factor $1/\gamma_n^4$ the remainder is $\mathcal{O}(1/\gamma_n^2)$.\\
To calculate the expectations, we require some basis arguments of combinatorics. In particular, we only need to consider derangements such that we cannot isolate one single factor of the product regarding independence to the other factors. The latter is because each single factor itself is centered. Additionally, we can make use of the symmetry of the product leading to respective constant factor in front of the probabilities. To keep the proofs as simple as possible, we do not specify these factors more than that they are finite and positive. In this spirit, for constants $c_2,\ldots,c_7\in(0,\infty)$ we have:
\begin{align}
       &\mathbb{E}\bigg[\!
     \left\{\!
     \left(\nonumber
     m_j\left(X_1^{(j)}\right)
    - m_j\left(X_{\pi(1)}^{(j)}\right)\right)^2+2\varepsilon_i\left(\nonumber
     m_j\left(X_1^{(j)}\right)
    - m_j\left(X_{\pi(1)}^{(j)}\right)\right)
    -I(j)\right\}^2\\
        &
    \nonumber\!
     \left\{\!
     \left(\nonumber
     m_j\left(X_{2}^{(j)}\right)
    - m_j\left(X_{\pi(2)}^{(j)}\right)\right)^2
    -I(j)\right\}\!    
         \left\{
     \left(\nonumber
     m_j\left(X_{3}^{(j)}\right)
    - m_j\left(X_{\pi(3)}^{(j)}\right)\right)^2
    -I(j)\!\right\}  \!
    \bigg]\\
    &=
    \nonumber
    c_2\mathbb{P}\left(
\pi(1)=2,\pi(2)=3,\pi(3)=1
    \right)\\
      &
    \mathbb{E}\bigg[
     \bigg\{
     \left(\nonumber
     m_j\left(X_1^{(j)}\right)
    - m_j\left(X_{2}^{(j)}\right)\right)^2+2\varepsilon_i\left(\nonumber
     m_j\left(X_1^{(j)}\right)
    - m_j\left(X_{2}^{(j)}\right)\right)
    -I(j)\bigg\}^2\\
        &
    \nonumber
     \left\{
     \left(\nonumber
     m_j\left(X_{2}^{(j)}\right)
    - m_j\left(X_{3}^{(j)}\right)\right)^2
    -I(j)\right\}  \!
     \left\{
     \left(\nonumber
     m_j\left(X_{3}^{(j)}\right)
    - m_j\left(X_{1}^{(j)}\right)\right)^2
    -I(j)\right\}  
    \bigg]\\
    &+
        \nonumber
    c_3\mathbb{P}\left(
\pi(1)=2,\pi(2)=3,\pi(3)>3
    \right)\\
    &
    \mathbb{E}\bigg[
     \bigg\{
     \left(\nonumber
     m_j\left(X_1^{(j)}\right)
    - m_j\left(X_{2}^{(j)}\right)\right)^2+2\varepsilon_i\left(\nonumber
     m_j\left(X_1^{(j)}\right)
    - m_j\left(X_{2}^{(j)}\right)\right)
    -I(j)\bigg\}^2\\
           &
    \nonumber
     \left\{
     \left(\nonumber
     m_j\left(X_{2}^{(j)}\right)
    - m_j\left(X_{3}^{(j)}\right)\right)^2
    -I(j)\right\}  \!
     \left\{
     \left(\nonumber
     m_j\left(X_{3}^{(j)}\right)
    - m_j\left(X_{4}^{(j)}\right)\right)^2
    -I(j)\right\}  
    \bigg]\\
    &+
        \nonumber
    c_4\mathbb{P}\left(
\pi(1)=2,\pi(2)>3,\pi(3)=1
    \right)\\
    &
    \mathbb{E}\bigg[
     \bigg\{
     \left(\nonumber
     m_j\left(X_1^{(j)}\right)
    - m_j\left(X_{2}^{(j)}\right)\right)^2+2\varepsilon_i\left(\nonumber
     m_j\left(X_1^{(j)}\right)
    - m_j\left(X_{2}^{(j)}\right)\right)
    -I(j)\bigg\}^2\\
           &
    \nonumber
     \left\{
     \left(\nonumber
     m_j\left(X_{2}^{(j)}\right)
    - m_j\left(X_{4}^{(j)}\right)\right)^2
    -I(j)\right\}  
     \left\{
     \left(\nonumber
     m_j\left(X_{3}^{(j)}\right)
    - m_j\left(X_{1}^{(j)}\right)\right)^2
    -I(j)\right\}  
    \bigg]\\
      &+
        \nonumber
    c_5\mathbb{P}\left(
\pi(1)>3,\pi(2)>3,\pi(3)=2
    \right)\\
    &
    \mathbb{E}\bigg[
     \bigg\{
     \left(\nonumber
     m_j\left(X_1^{(j)}\right)
    - m_j\left(X_{4}^{(j)}\right)\right)^2+2\varepsilon_i\left(\nonumber
     m_j\left(X_1^{(j)}\right)
    - m_j\left(X_{4}^{(j)}\right)\right)
    -I(j)\bigg\}^2\\
           &
    \nonumber
     \left\{
     \left(\nonumber
     m_j\left(X_{2}^{(j)}\right)
    - m_j\left(X_{5}^{(j)}\right)\right)^2
    -I(j)\right\}  
     \left\{
     \left(\nonumber
     m_j\left(X_{3}^{(j)}\right)
    - m_j\left(X_{2}^{(j)}\right)\right)^2
    -I(j)\right\}  
    \bigg]\\
    &=\mathcal{O}\left(\frac{1}{\gamma_n}\right),\label{theorem 2.2 viertes moment erste Summe}
\end{align}
since $m_j$, and therefore also the expectation, is bounded,
\begin{align}
    &\mathbb{P}\left(\nonumber
\pi(1)=2,\pi(2)=3,\pi(3)=1
    \right)=\mathcal{O}\left(\frac{1}{\gamma_n^3}\right),\\
    &\mathbb{P}\left(
\pi(1)=2,\pi(2)=3,\pi(3)>3
    \right)=\mathcal{O}\left(\frac{1}{\gamma_n^2}\right),\nonumber\\
    &
    \mathbb{P}\left(
\pi(1)=2,\pi(2)>3,\pi(3)=1
    \right)=\mathcal{O}\left(\frac{1}{\gamma_n^2}\right)\nonumber
\end{align}
and
\begin{align*}
    \mathbb{P}\left(
\pi(1)>3,\pi(1)>3,\pi(3)=2
    \right)=\mathcal{O}\left(\frac{1}{\gamma_n}\right).
\end{align*}  
Furthermore
\begin{align}
       &
     \mathbb{E}\bigg[
     \!\left\{\!
     \left(\nonumber
     m_j\left(X_1^{(j)}\right)
    - m_j\left(X_{\pi(1)}^{(j)}\right)\right)^2
    \!-\!I(j)\right\}\!
     \left\{
     \left(\nonumber
     m_j\left(X_2^{(j)}\right)
    - m_j\left(X_{\pi(2)}^{(j)}\right)\right)^2
    \!-\!I(j)\right\}\\
      &
    \nonumber
     \left\{\!
     \left(\nonumber
     m_j\left(X_3^{(j)}\right)
    - m_j\left(X_{\pi(3)}^{(j)}\right)\right)^2
    \!-\!I(j)\right\}\!
     \left\{
     \left(\nonumber
     m_j\left(X_{4}^{(j)}\right)
    - m_j\left(X_{\pi(4)}^{(j)}\right)\right)^2
    \!-\!I(j)\right\}\bigg]\\
    &=
    c_6\mathbb{P}\left(\pi(1)=2,\pi(2)\notin\{3,4\},\pi(3)=4,\pi(4)\notin\{1,2\}\right)
    \nonumber\\
    &
    \mathbb{E}\bigg[
     \left\{
     \left(\nonumber
     m_j\left(X_1^{(j)}\right)
    - m_j\left(X_{2}^{(j)}\right)\right)^2
    -I(j)\right\}\\
      &
    \left\{
     \left(\nonumber
     m_j\left(X_2^{(j)}\right)
    - m_j\left(X_{\pi(2)}^{(j)}\right)\right)^2
    -I(j)\right\}
    \nonumber
     \left\{
     \left(\nonumber
     m_j\left(X_3^{(j)}\right)
    - m_j\left(X_{4}^{(j)}\right)\right)^2
    -I(j)\right\}\\
      &
     \left\{
     \left(\nonumber
     m_j\left(X_{4}^{(j)}\right)
    - m_j\left(X_{\pi(4)}^{(j)}\right)\right)^2
    -I(j)\right\}\bigg|\pi(2)\notin\{3,4\}
,\pi(4)\notin\{1,2\}\bigg]\\[0.3em]
    &+
    c_7\mathbb{P}\left(\pi(1)=2,\pi(2)=3,\pi(3)=4,\pi(4)>4\right)\nonumber\\
    &
    \mathbb{E}\bigg[
     \!\left\{\!
     \left(\nonumber
     m_j\left(X_1^{(j)}\right)
    - m_j\left(X_{2}^{(j)}\right)\right)^2
    -I(j)\right\}
    \left\{
     \left(\nonumber
     m_j\left(X_2^{(j)}\right)
    - m_j\left(X_{3}^{(j)}\right)\right)^2
    \!-\!I(j)\right\}\\
      &
    \nonumber
     \left\{\!
     \left(\nonumber
     m_j\left(X_3^{(j)}\right)
    - m_j\left(X_{4}^{(j)}\right)\right)^2
    \!-\!I(j)\right\}\!
     \left\{
     \left(\nonumber
     m_j\left(X_{4}^{(j)}\right)
    - m_j\left(X_{5}^{(j)}\right)\right)^2
    \!-\!I(j)\right\}\!\bigg]\\
    &+\nonumber
     \mathcal{O}\left(\frac{1}{\gamma_n^4}\right)\\
    &=
    \mathcal{O}\left(\frac{1}{\gamma_n^2}\right),\label{theorem 2.2 viertes moment zweite Summe}
\end{align}
because
\begin{align*}
    &\mathbb{P}\left(\pi(1)=2,\pi(2)\notin\{3,4\},\pi(3)=4,\pi(4)\notin\{1,2\}\right)\\
    &\leq
    \mathbb{P}\left(\pi(1)=2,\pi(3)=4,\pi(4)>4\right)
    =\frac{!(\gamma_n-2)}{!\gamma_n}
    =
    \frac{1}{\gamma_n^2}+o(1)
\end{align*}
and
\begin{align*}
    &\mathbb{P}\left(\pi(1)=2,\pi(2)=3,\pi(3)=4,i(4)>4\right)\\&\leq\mathbb{P}\left(\pi(1)=2,\pi(3)=4\right)
     =\frac{!(\gamma_n-2)}{!\gamma_n}  =
    \frac{1}{\gamma_n^2}+o(1).
\end{align*}
Hence, with (\ref{theorem 2.2 viertes moment erste Summe}) and (\ref{theorem 2.2 viertes moment zweite Summe})
\begin{align*}
    \mathbb{E}\left[\left|h_{\gamma_n}-I(j)\right|^4\right]=\mathcal{O}\left(\frac{1}{\gamma_n^2}\right)
\end{align*}
holds true. 
Since $\mathbb{E}\left[\left|h_{\gamma_n}-I(j)\right|^2\right]^2\sim\frac{c_1^2}{\gamma_n^2}$, it also holds that
\[
\frac{ \mathbb{E}\left[\left|h_{\gamma_n}-I(j)\right|^4\right]}{\mathbb{E}\left[\left|h_{\gamma_n}-I(j)\right|^2\right]^2}
=
\mathcal{O}(1)
\]
is uniformly bounded.
With Hölder inequality, we also have
\begin{align*}
    \mathbb{E}\left[\left|h_{\gamma_n}-I(j)\right|^3\right]^2
    \leq
    \mathbb{E}\left[\left|h_{\gamma_n}-I(j)\right|^4\right]^{3/2}
    \leq
    c_8\left(\frac{1}{\gamma_n^2}\right)^{3/2}=\frac{c_8}{\gamma_n^3}
\end{align*}
and
\begin{align*}
    \mathbb{E}\left[\left|h_{\gamma_n}-I(j)\right|^3\right]^2
    \geq
    \mathbb{E}\left[\left|h_{\gamma_n}-I(j)\right|^2\right]^{3}
    \geq =\frac{c_9}{\gamma_n^3}
\end{align*}
for constants $c_8,c_9>0$. To show that (\ref{theorem 2.2 ljapunov})
is uniformly bounded, it only remains to show that
\begin{align*}
     \mathbb{E}\left[\left|h_{\gamma_n}-I(j)\right|^6\right]=\mathcal{O}\left(\frac{1}{\gamma_n^3}\right).
\end{align*}
In the same way as for the fourth moment, (A2), \Cref{Lemma 3},  and the i.i.d assumptions
can be used to recieve
\begin{align*}
      &\mathbb{E}\left[\left|h_{\gamma_n}-I(j)\right|^6\right]\nonumber\\
              &=
    \nonumber
        \mathbb{E}\Bigg[\bigg|\frac{1}{\gamma_n}\sum_{i=1}^{\gamma_n}
     \bigg\{
     \left(\nonumber
     m_j\left(X_i^{(j)}\right)
    - m_j\left(X_{\pi(i)}^{(j)}\right)\right)^2\\
    &
    \quad\quad\quad\quad\quad\quad\quad\quad\quad\quad\quad\quad\quad
    +2\varepsilon_i\left(\nonumber
     m_j\left(X_i^{(j)}\right)
    - m_j\left(X_{\pi(i)}^{(j)}\right)\right)
\bigg\}
    -I(j)\bigg|^6\Bigg]\\[0.5em]
        &=
    \nonumber
    \mathcal{O}\left(\frac{1}{\gamma_n^3}\right)   +
    \frac{1}{\gamma_n^6}\underset{\text{pairwise distinct}}{\sum_{i,i_2,{i_3},i_4}}\\
      & \mathbb{E}\bigg[\!
     \left\{\!
     \left(\nonumber
     m_j\left(X_{i_1}^{(j)}\right)
    - m_j\left(X_{\pi({i_1})}^{(j)}\right)\right)^2+2\varepsilon_{i_1}\left(\nonumber
     m_j\left(X_{i_1}^{(j)}\right)
    - m_j\left(X_{\pi({i_1})}^{(j)}\right)\right)
    -I(j)\right\}^3\\
          &
    \nonumber
     \left\{
     \left(\nonumber
     m_j\left(X_{i_3}^{(j)}\right)
    - m_j\left(X_{\pi({i_3})}^{(j)}\right)\right)^2+2\varepsilon_{i_3}\left(\nonumber
     m_j\left(X_{i_3}^{(j)}\right)
    - m_j\left(X_{\pi(i_4)}^{(j)}\right)\right)
    -I(j)\right\}\\
      &
    \nonumber
     \left\{
     \left(\nonumber
     m_j\left(X_{i_4}^{(j)}\right)
    - m_j\left(X_{\pi(i_4)}^{(j)}\right)\right)^2+2\varepsilon_{i_4}\left(\nonumber
     m_j\left(X_{i_4}^{(j)}\right)
    - m_j\left(X_{\pi(i_4)}^{(j)}\right)\right)
    -I(j)\right\}
    \bigg]
\end{align*}

\begin{align}
&+
    \nonumber\frac{1}{\gamma_n^6}\underset{\text{pairwise distinct}}{\sum_{{i_1},i_2,{i_3},i_4}}\\
      &\mathbb{E}\bigg[\!
     \left\{\!
     \left(\nonumber
     m_j\left(X_{i_1}^{(j)}\right)
    - m_j\left(X_{\pi({i_1})}^{(j)}\right)\right)^2+2\varepsilon_{i_1}\left(\nonumber
     m_j\left(X_{i_1}^{(j)}\right)
    - m_j\left(X_{\pi{i_1})}^{(j)}\right)\right)
    -I(j)\right\}^2\\
    &
    \nonumber
     \left\{\!
     \left(\nonumber
     m_j\left(X_{i_2}^{(j)}\right)
    - m_j\left(X_{\pi({i_2})}^{(j)}\right)\right)^2+2\varepsilon_{i_2}\left(\nonumber
     m_j\left(X_{i_2}^{(j)}\right)
    - m_j\left(X_{\pi({i_2})}^{(j)}\right)\right)
    -I(j)\right\}^2\\  
      &
    \nonumber
     \left\{\!
     \left(\nonumber
     m_j\left(X_{i_3}^{(j)}\right)
    - m_j\left(X_{\pi({i_3})}^{(j)}\right)\right)^2+2\varepsilon_{i_3}\left(\nonumber
     m_j\left(X_{i_3}^{(j)}\right)
    - m_j\left(X_{\pi(i_4)}^{(j)}\right)\right)
    -I(j)\right\}\\
      &
    \nonumber
     \left\{\!
     \left(\nonumber
     m_j\left(X_{i_4}^{(j)}\right)
    - m_j\left(X_{\pi(i_4)}^{(j)}\right)\right)^2+2\varepsilon_{i_4}\left(\nonumber
     m_j\left(X_{i_4}^{(j)}\right)
    - m_j\left(X_{\pi(i_4)}^{(j)}\right)\right)
    -I(j)\right\}
    \bigg]\\    
    &\quad+
        \frac{1}{\gamma_n^6}
        \underset{\text{pairwise distinct}}{
        \sum_{{i_1},{i_2},{i_3},i_4,i_5} }\nonumber\\
        &
      \mathbb{E}\bigg[
     \left\{\!
     \left(\nonumber
     m_j\left(X_{i_1}^{(j)}\right)
    - m_j\left(X_{\pi({i_1})}^{(j)}\right)\right)^2+2\varepsilon_{i_1}\left(\nonumber
     m_j\left(X_{i_1}^{(j)}\right)
    - m_j\left(X_{\pi({i_1})}^{(j)}\right)\right)
    -I(j)\right\}^2\\
    &
    \nonumber
     \left\{\!
     \left(\nonumber
     m_j\left(X_{i_2}^{(j)}\right)
    - m_j\left(X_{\pi({i_2})}^{(j)}\right)\right)^2+2\varepsilon_{i_2}\left(\nonumber
     m_j\left(X_{i_2}^{(j)}\right)
    - m_j\left(X_{\pi({i_2})}^{(j)}\right)\right)
    -I(j)\right\}\\
      &
    \nonumber
     \left\{\!
     \left(\nonumber
     m_j\left(X_{i_3}^{(j)}\right)
    - m_j\left(X_{\pi({i_3})}^{(j)}\right)\right)^2+2\varepsilon_{i_3}\left(\nonumber
     m_j\left(X_{i_3}^{(j)}\right)
    - m_j\left(X_{\pi({i_3})}^{(j)}\right)\right)
    -I(j)\right\}\\
       &
    \nonumber
     \left\{\!
     \left(\nonumber
     m_j\left(X_{i_4}^{(j)}\right)
    - m_j\left(X_{\pi(i_4)}^{(j)}\right)\right)^2+2\varepsilon_{i_4}\left(\nonumber
     m_j\left(X_{i_4}^{(j)}\right)
    - m_j\left(X_{\pi(i_4)}^{(j)}\right)\right)
    -I(j)\right\} \\
       &
    \nonumber
     \left\{\!
     \left(\nonumber
     m_j\left(X_{i_5}^{(j)}\right)
    - m_j\left(X_{\pi(i_5)}^{(j)}\right)\right)^2+2\varepsilon_{i_5}\left(\nonumber
     m_j\left(X_{i_5}^{(j)}\right)
    - m_j\left(X_{\pi(i_5)}^{(j)}\right)\right)
    -I(j)\right\} 
    \bigg]\\
    &\quad+
        \frac{1}{\gamma_n^6}\nonumber
        \underset{\text{pairwise distinct}}{
        \sum_{{i_1},{i_2},{i_3},i_4,i_5,i_6}}\\
        &
      \mathbb{E}\bigg[\!
     \left\{\!
     \left(\nonumber
     m_j\left(X_{i_1}^{(j)}\right)
    - m_j\left(X_{\pi({i_1})}^{(j)}\right)\right)^2+2\varepsilon_{i_1}\left(\nonumber
     m_j\left(X_{i_1}^{(j)}\right)
    - m_j\left(X_{\pi({i_1})}^{(j)}\right)\right)
    -I(j)\right\}\\
    &
    \nonumber
     \left\{\!
     \left(\nonumber
     m_j\left(X_{i_2}^{(j)}\right)
    - m_j\left(X_{\pi({i_2})}^{(j)}\right)\right)^2+2\varepsilon_{i_2}\left(\nonumber
     m_j\left(X_{i_2}^{(j)}\right)
    - m_j\left(X_{\pi({i_2})}^{(j)}\right)\right)
    -I(j)\right\}\\
      &
    \nonumber
     \left\{\!
     \left(\nonumber
     m_j\left(X_{i_3}^{(j)}\right)
    - m_j\left(X_{\pi({i_3})}^{(j)}\right)\right)^2+2\varepsilon_{i_3}\left(\nonumber
     m_j\left(X_{i_3}^{(j)}\right)
    - m_j\left(X_{\pi({i_3})}^{(j)}\right)\right)
    -I(j)\right\}\\
       &
    \nonumber
     \left\{\!
     \left(\nonumber
     m_j\left(X_{i_4}^{(j)}\right)
    - m_j\left(X_{\pi(i_4)}^{(j)}\right)\right)^2+2\varepsilon_{i_4}\left(\nonumber
     m_j\left(X_{i_4}^{(j)}\right)
    - m_j\left(X_{\pi(i_4)}^{(j)}\right)\right)
    -I(j)\right\} \\
       &
    \nonumber
     \left\{\!
     \left(\nonumber
     m_j\left(X_{i_5}^{(j)}\right)
    - m_j\left(X_{\pi(i_5)}^{(j)}\right)\right)^2+2\varepsilon_{i_5}\left(\nonumber
     m_j\left(X_{i_5}^{(j)}\right)
    - m_j\left(X_{\pi(i_5)}^{(j)}\right)\right)
    -I(j)\right\} \\
       &
    \nonumber
     \left\{\!
     \left(\nonumber
     m_j\left(X_{i_6}^{(j)}\right)
    - m_j\left(X_{\pi(i_6)}^{(j)}\right)\right)^2+2\varepsilon_{i_6}\left(\nonumber
     m_j\left(X_{i_6}^{(j)}\right)
    - m_j\left(X_{\pi(i_6)}^{(j)}\right)\right)
    -I(j)\right\} 
    \bigg]\\
          &=\mathcal{O}\left(\frac{1}{\gamma_n^3}\right)\nonumber
    \end{align}

    \begin{align}
      &+
      \frac{4!(\gamma_n\!-\!1)(\gamma_n\!-\!2)(\gamma_n\!-\!3)}{\gamma_n^5}\nonumber\\
      &
       \mathbb{E}\bigg[
     \left\{
     \left(\nonumber
     m_j\left(X_1^{(j)}\right)
    - m_j\left(X_{\pi(1)}^{(j)}\right)\right)^2+2\varepsilon_1\left(\nonumber
     m_j\left(X_1^{(j)}\right)
    - m_j\left(X_{\pi(1)}^{(j)}\right)\right)
    \!-\!I(j)\right\}^3\\
    &
    \nonumber
     \left\{
     \left(\nonumber
     m_j\left(X_2^{(j)}\right)
    - m_j\left(X_{\pi(2)}^{(j)}\right)\right)^2
    \!-\!I(j)\right\}\!
     \left\{
     \left(\nonumber
     m_j\left(X_3^{(j)}\right)
    - m_j\left(X_{\pi(3)}^{(j)}\right)\right)^2
    \!-\!I(j)\right\}\\
      &
\quad\quad\quad\quad\quad
    \quad\quad\quad\quad\quad
    \quad\quad\quad\quad
    \nonumber
     \left\{
     \left(\nonumber
     m_j\left(X_{4}^{(j)}\right)
    - m_j\left(X_{\pi(4)}^{(j)}\right)\right)^2
    \!-\!I(j)\right\}\!
    \bigg]\\
        &+
      \frac{4!(\gamma_n\!-\!1)(\gamma_n\!-\!2)(\gamma_n\!-\!3)}{\gamma_n^5}\nonumber\\
      &       \mathbb{E}\bigg[
     \left\{
     \left(\nonumber
     m_j\left(X_1^{(j)}\right)
    - m_j\left(X_{\pi(1)}^{(j)}\right)\right)^2+2\varepsilon_1\left(\nonumber
     m_j\left(X_1^{(j)}\right)
    - m_j\left(X_{\pi(1)}^{(j)}\right)\right)
    \!-\!I(j)\right\}^2\\
    &
    \nonumber
       \left\{
     \left(\nonumber
     m_j\left(X_2^{(j)}\right)
    - m_j\left(X_{\pi(2)}^{(j)}\right)\right)^2+2\varepsilon_2\left(\nonumber
     m_j\left(X_2^{(j)}\right)
    - m_j\left(X_{\pi(2)}^{(j)}\right)\right)
    \!-\!I(j)\right\}^2\\
      &
    \nonumber
       \left\{
     \left(\nonumber
     m_j\left(X_3^{(j)}\right)
    - m_j\left(X_{\pi(3)}^{(j)}\right)\right)^2
    \!-\!I(j)\right\}\!
     \left\{
     \left(\nonumber
     m_j\left(X_{4}^{(j)}\right)
    - m_j\left(X_{\pi(4)}^{(j)}\right)\right)^2
    \!-\!I(j)\right\}\!
    \bigg]\\
    &+
    \nonumber
    \frac{5!(\gamma_n\!-\!1)(\gamma_n\!-\!2)(\gamma_n\!-\!3)(\gamma_n\!-\!4)}{\gamma_n^5}\\
     &
      \mathbb{E}\bigg[
     \!\left\{\!
     \left(\nonumber
     m_j\left(X_1^{(j)}\right)
    - m_j\left(X_{\pi(1)}^{(j)}\right)\right)^2+2\varepsilon_1\left(\nonumber
     m_j\left(X_1^{(j)}\right)
    - m_j\left(X_{\pi(1)}^{(j)}\right)\right)
    \!-\!I(j)\right\}^2\\
    &
    \nonumber
     \left\{
     \left(\nonumber
     m_j\left(X_2^{(j)}\right)
    - m_j\left(X_{\pi(2)}^{(j)}\right)\right)^2
    \!-\!I(j)\right\}\!
     \left\{
     \left(\nonumber
     m_j\left(X_3^{(j)}\right)
    - m_j\left(X_{\pi(3)}^{(j)}\right)\right)^2
    \!-\!I(j)\right\}\\
      &
    \nonumber
     \left\{
     \left(\nonumber
     m_j\left(X_{4}^{(j)}\right)
    - m_j\left(X_{\pi(4)}^{(j)}\right)\right)^2
    \!-\!I(j)\right\}\!
     \left\{
     \left(\nonumber
     m_j\left(X_{5}^{(j)}\right)
    - m_j\left(X_{\pi(5)}^{(j)}\right)\right)^2
    \!-\!I(j)\right\}\!
    \bigg]\\
    &+
    \nonumber
    \frac{6!(\gamma_n\!-\!1)(\gamma_n\!-\!2)(\gamma_n\!-\!3)(\gamma_n\!-\!4)
    (\gamma_n\!-\!5)}{\gamma_n^5}\\
     &
      \mathbb{E}\bigg[
    \! \left\{\!
     \left(\nonumber
     m_j\left(X_1^{(j)}\right)
    - m_j\left(X_{\pi(1)}^{(j)}\right)\right)^2
    \!-\!I(j)\right\}\!\left\{
     \left(\nonumber
     m_j\left(X_2^{(j)}\right)
    - m_j\left(X_{\pi(2)}^{(j)}\right)\right)^2
    \!-\!I(j)\right\}\\
        &
     \left\{
     \left(\nonumber
     m_j\left(X_3^{(j)}\right)
    - m_j\left(X_{\pi(3)}^{(j)}\right)\right)^2
    \!-\!I(j)\right\}\!
     \left\{
     \left(\nonumber
     m_j\left(X_{4}^{(j)}\right)
    - m_j\left(X_{\pi(4)}^{(j)}\right)\right)^2
   \! -\!I(j)\right\}\\
      &
    \nonumber
     \left\{\!
     \left(\nonumber
     m_j\left(X_{5}^{(j)}\right)
    - m_j\left(X_{\pi(5)}^{(j)}\right)\right)^2
    \!-\!I(j)\right\}\!
         \left\{
     \left(\nonumber
     m_j\left(X_{6}^{(j)}\right)
    - m_j\left(X_{\pi(6)}^{(j)}\right)\right)^2
    \!-\!I(j)\right\}\!
    \bigg].
\end{align}

\noindent Continuing with the calculation of expectations, we again rely on some basic principles of combinatorics. In particular, we consider derangements to ensure that no single factor of the product can be isolated, as this would affect the independence of the other factors. This is because each individual factor is centered. Moreover, we 
exploit the symmetry of the product, which implies a constant factor in front of the probabilities. To keep the proofs simple, we only specify these factors as finite and positive. In this context, the following applies to the constants $\widetilde c_1,\ldots\widetilde c_{31}$:

\begin{align*}
        &\mathbb{E}\bigg[
     \left\{
     \left(\nonumber
     m_j\left(X_1^{(j)}\right)
    - m_j\left(X_{\pi(1)}^{(j)}\right)\right)^2+2\varepsilon_1\left(\nonumber
     m_j\left(X_1^{(j)}\right)
    - m_j\left(X_{\pi(1)}^{(j)}\right)\right)
    -I(j)\right\}^3\\
      &\nonumber
     \left\{
     \left(\nonumber
     m_j\left(X_2^{(j)}\right)
    - m_j\left(X_{\pi(2)}^{(j)}\right)\right)^2
    -I(j)\right\}
     \left\{
     \left(\nonumber
     m_j\left(X_3^{(j)}\right)
    - m_j\left(X_{\pi(3)}^{(j)}\right)\right)^2
    -I(j)\right\}\\
          &
\quad\quad\quad\quad
    \quad\quad\quad\quad
    \quad\quad\quad\quad
    \quad\quad\quad\quad
    \nonumber
     \left\{
     \left(\nonumber
     m_j\left(X_{4}^{(j)}\right)
    - m_j\left(X_{\pi(4)}^{(j)}\right)\right)^2
    -I(j)\right\}
    \bigg]\\
        &=
    \widetilde c_1
    \mathbb{P}\left(
    \pi(1)=2,\pi(2)\notin\{3,4\},\pi(3)=4,\pi(4)\notin\{1,2\}
    \right)\nonumber\\
      &
      \mathbb{E}\bigg[
     \left\{
     \left(\nonumber
     m_j\left(X_1^{(j)}\right)
    - m_j\left(X_{2}^{(j)}\right)\right)^2+2\varepsilon_1\left(\nonumber
     m_j\left(X_1^{(j)}\right)
    - m_j\left(X_{2}^{(j)}\right)\right)
    -I(j)\right\}^3\\
    &
    \nonumber
     \left\{
     \left(\nonumber
     m_j\left(X_2^{(j)}\right)
    - m_j\left(X_{\pi(2)}^{(j)}\right)\right)^2
    -I(j)\right\}
     \left\{
     \left(\nonumber
     m_j\left(X_3^{(j)}\right)
    - m_j\left(X_{4}^{(j)}\right)\right)^2
    -I(j)\right\}\\
      &
    \quad\quad\quad\quad
    \nonumber
     \left\{
     \left(\nonumber
     m_j\left(X_{4}^{(j)}\right)
    - m_j\left(X_{\pi(4)}^{(j)}\right)\right)^2
    -I(j)\right\}
    \bigg| \pi(2)\notin\{3,4\}
,\pi(4)\notin\{1,2\}
    \bigg]\\[0.3em]
      &+
    \widetilde c_2
    \mathbb{P}\left(
    \pi(1)\notin\{3,4\},\pi(2)=1,\pi(3)=4,\pi(4)\notin\{1,2\}
    \right)\nonumber\\
     &
      \mathbb{E}\bigg[
     \left\{
     \left(\nonumber
     m_j\left(X_1^{(j)}\right)
    - m_j\left(X_{\pi(1)}^{(j)}\right)\right)^2+2\varepsilon_1\left(\nonumber
     m_j\left(X_1^{(j)}\right)
    - m_j\left(X_{\pi(1)}^{(j)}\right)\right)
    -I(j)\right\}^3\\
    &
    \nonumber
     \left\{
     \left(\nonumber
     m_j\left(X_2^{(j)}\right)
    - m_j\left(X_{1}^{(j)}\right)\right)^2
    -I(j)\right\}
     \left\{
     \left(\nonumber
     m_j\left(X_3^{(j)}\right)
    - m_j\left(X_{4}^{(j)}\right)\right)^2
    -I(j)\right\}\\
      &
    \nonumber
        \quad\quad\quad\quad
     \left\{
     \left(\nonumber
     m_j\left(X_{4}^{(j)}\right)
    - m_j\left(X_{\pi(4)}^{(j)}\right)\right)^2
    -I(j)\right\}
    \bigg| \pi(1)\notin\{3,4\}
,\pi(4)\notin\{1,2\}
    \bigg]\\
         &+
     \widetilde c_3
    \mathbb{P}\left(
    \pi(1)\notin\{2,3,4\},\pi(2)=3,\pi(3)=4,\pi(4)\neq 1
    \right)\nonumber\\
     &
      \mathbb{E}\bigg[
     \left\{
     \left(\nonumber
     m_j\left(X_1^{(j)}\right)
    - m_j\left(X_{\pi(1)}^{(j)}\right)\right)^2+2\varepsilon_1\left(\nonumber
     m_j\left(X_1^{(j)}\right)
    - m_j\left(X_{\pi(1)}^{(j)}\right)\right)
    -I(j)\right\}^3\\
    &
    \nonumber
     \left\{
     \left(\nonumber
     m_j\left(X_2^{(j)}\right)
    - m_j\left(X_{3}^{(j)}\right)\right)^2
    -I(j)\right\}
     \left\{
     \left(\nonumber
     m_j\left(X_3^{(j)}\right)
    - m_j\left(X_{4}^{(j)}\right)\right)^2
    -I(j)\right\}\\
      &
\quad\quad\quad\quad
    \nonumber
     \left\{
     \left(\nonumber
     m_j\left(X_{4}^{(j)}\right)
    - m_j\left(X_{\pi(4)}^{(j)}\right)\right)^2
    -I(j)\right\}
    \bigg| \pi(1)\notin\{3,4\}
,\pi(4)\notin\{1,2\}
    \bigg]\\
       &+
     \widetilde c_4
    \mathbb{P}\left(
    \pi(1)=2,\pi(2)=3,\pi(3)\neq 4,\pi(4)= 1
    \right)\nonumber\\
     &
      \mathbb{E}\bigg[
     \left\{
     \left(\nonumber
     m_j\left(X_1^{(j)}\right)
    - m_j\left(X_{2}^{(j)}\right)\right)^2+2\varepsilon_1\left(\nonumber
     m_j\left(X_1^{(j)}\right)
    - m_j\left(X_{2}^{(j)}\right)\right)
    -I(j)\right\}^3\\
        &
    \nonumber
     \left\{
     \left(\nonumber
     m_j\left(X_2^{(j)}\right)
    - m_j\left(X_{3}^{(j)}\right)\right)^2
    -I(j)\right\}
     \left\{
     \left(\nonumber
     m_j\left(X_3^{(j)}\right)
    - m_j\left(X_{\pi(3)}^{(j)}\right)\right)^2
    -I(j)\right\}\\
      &
\end{align*}

\begin{align*}
&\quad\quad\quad\quad
    \nonumber
     \left\{
     \left(\nonumber
     m_j\left(X_{4}^{(j)}\right)
    - m_j\left(X_{1}^{(j)}\right)\right)^2
    -I(j)\right\}
    \bigg| \pi(3)\neq 4
    \bigg]\\
     &+
     \widetilde c_5
  \mathbb{P}\left(
    \pi(1)=4,\pi(2)=3,\pi(3)=1,\pi(4)\notin\{2,3\}
    \right)\nonumber\\
     &
      \mathbb{E}\bigg[
     \left\{
     \left(\nonumber
     m_j\left(X_1^{(j)}\right)
    - m_j\left(X_{4}^{(j)}\right)\right)^2+2\varepsilon_1\left(\nonumber
     m_j\left(X_1^{(j)}\right)
    - m_j\left(X_{4}^{(j)}\right)\right)
    -I(j)\right\}^3\\
    &
    \nonumber
     \left\{
     \left(\nonumber
     m_j\left(X_2^{(j)}\right)
    - m_j\left(X_{3}^{(j)}\right)\right)^2
    -I(j)\right\}
     \left\{
     \left(\nonumber
     m_j\left(X_3^{(j)}\right)
    - m_j\left(X_{1}^{(j)}\right)\right)^2
    -I(j)\right\}\\
      &
\quad\quad\quad\quad
    \nonumber
     \left\{
     \left(\nonumber
     m_j\left(X_{4}^{(j)}\right)
    - m_j\left(X_{\pi(4)}^{(j)}\right)\right)^2
    -I(j)\right\}
    \bigg|\pi(4)\notin\{2,3\}
    \bigg]\\
    &+
     \widetilde c_6
  \mathbb{P}\left(
    \pi(1)=2,\pi(2)=3,\pi(3)=4,\pi(4)\notin\{2,3\}
    \right)\nonumber\\
     &
      \mathbb{E}\bigg[
     \left\{
     \left(\nonumber
     m_j\left(X_1^{(j)}\right)
    - m_j\left(X_{2}^{(j)}\right)\right)^2+2\varepsilon_1\left(\nonumber
     m_j\left(X_1^{(j)}\right)
    - m_j\left(X_{2}^{(j)}\right)\right)
    -I(j)\right\}^3\\
    &
    \nonumber
     \left\{
     \left(\nonumber
     m_j\left(X_2^{(j)}\right)
    - m_j\left(X_{3}^{(j)}\right)\right)^2
    -I(j)\right\}
     \left\{
     \left(\nonumber
     m_j\left(X_3^{(j)}\right)
    - m_j\left(X_{4}^{(j)}\right)\right)^2
    -I(j)\right\}\\
      &
\quad\quad\quad\quad
    \nonumber
     \left\{
     \left(\nonumber
     m_j\left(X_{4}^{(j)}\right)
    - m_j\left(X_{\pi(4)}^{(j)}\right)\right)^2
    -I(j)\right\}
    \bigg|\pi(4)\notin\{2,3\}
    \bigg]\\
        &+
     \widetilde c_7
  \mathbb{P}\left(
    \pi(1)\neq 4,\pi(2)=1,\pi(3)=2,\pi(4)=3
    \right)\nonumber\\
     &
      \mathbb{E}\bigg[
     \left\{
     \left(\nonumber
     m_j\left(X_1^{(j)}\right)
    - m_j\left(X_{\pi(1)}^{(j)}\right)\right)^2+2\varepsilon_1\left(\nonumber
     m_j\left(X_1^{(j)}\right)
    - m_j\left(X_{\pi(1)}^{(j)}\right)\right)
    -I(j)\right\}^3\\
    &
    \nonumber
     \left\{
     \left(\nonumber
     m_j\left(X_2^{(j)}\right)
    - m_j\left(X_{1}^{(j)}\right)\right)^2
    -I(j)\right\}
     \left\{
     \left(\nonumber
     m_j\left(X_3^{(j)}\right)
    - m_j\left(X_{2}^{(j)}\right)\right)^2
    -I(j)\right\}\\
      &
\quad\quad\quad\quad
    \nonumber
     \left\{
     \left(\nonumber
     m_j\left(X_{4}^{(j)}\right)
    - m_j\left(X_{3}^{(j)}\right)\right)^2
    -I(j)\right\}
    \bigg|\pi(1)\neq 4
    \bigg]\\
     &+
     \widetilde c_8
  \mathbb{P}\left(
    \pi(1)= 4,\pi(2)=1,\pi(3)=2,\pi(4)=3
    \right)\nonumber\\
     &
      \mathbb{E}\bigg[
     \left\{
     \left(\nonumber
     m_j\left(X_1^{(j)}\right)
    - m_j\left(X_{4}^{(j)}\right)\right)^2+2\varepsilon_1\left(\nonumber
     m_j\left(X_1^{(j)}\right)
    - m_j\left(X_{4}^{(j)}\right)\right)
    -I(j)\right\}^3\\
    &
    \nonumber
     \left\{
     \left(\nonumber
     m_j\left(X_2^{(j)}\right)
    - m_j\left(X_{1}^{(j)}\right)\right)^2
    -I(j)\right\}
     \left\{
     \left(\nonumber
     m_j\left(X_3^{(j)}\right)
    - m_j\left(X_{2}^{(j)}\right)\right)^2
    -I(j)\right\}\\
      &
\quad\quad\quad\quad
    \nonumber
     \left\{
     \left(\nonumber
     m_j\left(X_{4}^{(j)}\right)
    - m_j\left(X_{3}^{(j)}\right)\right)^2
    -I(j)\right\}
    \bigg]\\
    &=\mathcal{O}\left(\frac{1}{\gamma_n^2}\right),\nonumber
\end{align*}

\noindent since

\begin{align*}
    &\max\Big\{\mathbb{P}\left(
    \pi(1)=2,\pi(2)\notin\{3,4\},\pi(3)=4,\pi(4)\notin\{1,2\}
    \right),~\\
    &~
    \mathbb{P}\left(
    \pi(1)\notin\{3,4\},\pi(2)=1,\pi(3)=4,\pi(4)\notin\{1,2\}
    \right),\\
    &~
    \mathbb{P}\left(
    \pi(1)\notin\{2,3,4\},\pi(2)=3,\pi(3)=4,\pi(4)\neq 1
    \right)\nonumber,
    ~\\
    &~\mathbb{P}\left(
    \pi(1)=2,\pi(2)=3,\pi(3)\neq 4,\pi(4)= 1
    \right),\nonumber\\
    &
       ~ \mathbb{P}\left(
    \pi(1)=4,\pi(2)=3,\pi(3)=1,\pi(4)\notin\{2,3\}
    \right),\\
    &~
     \mathbb{P}\left(
    \pi(1)=2,\pi(2)=3,\pi(3)=4,\pi(4)\notin\{2,3\}
    \right),\nonumber\\
    &~
     \mathbb{P}\left(
    \pi(1)\neq 4,\pi(2)=1,\pi(3)=2,\pi(4)=3
    \right),
    ~ \\
    &~ \mathbb{P}\left(
    \pi(1)= 4,\pi(2)=1,\pi(3)=2,\pi(4)=3
    \right)\Big\}\\
    &\leq
    \mathbb{P}\left(
    \pi(1)=2,\pi(2)=3
    \right)=\frac{!(\gamma_n-2)}{!\gamma_n}
    =
    \mathcal{O}\left(\frac{1}{\gamma_n^2}\right)
\end{align*}

\noindent as in (\ref{theorem 2.2 permu vereinigung}) or (\ref{theorem 2.2 permu schnitt}).
Analogously, the second four fold sum  is $\mathcal{O}\left(\frac{1}{\gamma_n}\right)$,
as we can see with

\begin{align*}
      &\mathbb{E}\bigg[
     \left\{
     \left(\nonumber
     m_j\left(X_1^{(j)}\right)
    - m_j\left(X_{\pi(1)}^{(j)}\right)\right)^2+2\varepsilon_1\left(\nonumber
     m_j\left(X_1^{(j)}\right)
    - m_j\left(X_{\pi(1)}^{(j)}\right)\right)
    -I(j)\right\}^2\\
    &
    \nonumber
     \left\{
     \left(\nonumber
     m_j\left(X_2^{(j)}\right)
    - m_j\left(X_{\pi(2)}^{(j)}\right)\right)^2+2\varepsilon_2\left(\nonumber
     m_j\left(X_2^{(j)}\right)
    - m_j\left(X_{\pi(2)}^{(j)}\right)\right)
    -I(j)\right\}^2\\  
      &
    \nonumber
     \left\{
     \left(\nonumber
     m_j\left(X_3^{(j)}\right)
    - m_j\left(X_{\pi(3)}^{(j)}\right)\right)^2
    \!-\!I(j)\right\}\!
     \left\{
     \left(\nonumber
     m_j\left(X_{4}^{(j)}\right)
    - m_j\left(X_{\pi(4)}^{(j)}\right)\right)^2
    \!-\!I(j)\right\}\!
    \bigg]\\
  &=
   \widetilde c_{10}\mathbb{P}\left(\pi(1)=3,\pi(2)=4,
    \pi(3)\notin\{2,4\},\pi(4)\notin\{1,3\}
    \right)\\
    &  
    \mathbb{E}\bigg[
     \left\{
     \left(\nonumber
     m_j\left(X_1^{(j)}\right)
    - m_j\left(X_{3}^{(j)}\right)\right)^2+2\varepsilon_1\left(\nonumber
     m_j\left(X_1^{(j)}\right)
    - m_j\left(X_{3}^{(j)}\right)\right)
    -I(j)\right\}^2\\
    &
    \nonumber
     \left\{
     \left(\nonumber
     m_j\left(X_2^{(j)}\right)
    - m_j\left(X_{4}^{(j)}\right)\right)^2+2\varepsilon_2\left(\nonumber
     m_j\left(X_2^{(j)}\right)
    - m_j\left(X_{4}^{(j)}\right)\right)
    -I(j)\right\}^2\\  
      &
    \nonumber
     \left\{
     \left(\nonumber
     m_j\left(X_3^{(j)}\right)
    - m_j\left(X_{\pi(3)}^{(j)}\right)\right)^2
    -I(j)\right\}\\
    &
    \quad\quad\quad
     \left\{
     \left(\nonumber
     m_j\left(X_{4}^{(j)}\right)
    - m_j\left(X_{\pi(4)}^{(j)}\right)\right)^2
    -I(j)\right\} \bigg| \pi(3)\notin\{2,4\},\pi(4)\notin\{1,3\}
    \bigg]\\
          &+
   \widetilde c_{11}\mathbb{P}\left(\pi(1)=3,\pi(2)\notin\{1,3\},
    \pi(3)\notin\{2,4\},\pi(4)=2
    \right)\\
    &  
    \mathbb{E}\bigg[
     \left\{
     \left(\nonumber
     m_j\left(X_1^{(j)}\right)
    - m_j\left(X_{3}^{(j)}\right)\right)^2+2\varepsilon_1\left(\nonumber
     m_j\left(X_1^{(j)}\right)
    - m_j\left(X_{3}^{(j)}\right)\right)
    -I(j)\right\}^2\\
    &
    \nonumber
     \left\{
     \left(\nonumber
     m_j\left(X_2^{(j)}\right)
    - m_j\left(X_{\pi(2)}^{(j)}\right)\right)^2+2\varepsilon_2\left(\nonumber
     m_j\left(X_2^{(j)}\right)
    - m_j\left(X_{\pi(2)}^{(j)}\right)\right)
    -I(j)\right\}^2\\  
          &
    \nonumber
     \left\{
     \left(\nonumber
     m_j\left(X_3^{(j)}\right)
    - m_j\left(X_{\pi(3)}^{(j)}\right)\right)^2
    -I(j)\right\}
\end{align*}

\begin{align*}  
    &
    \quad\quad\quad
     \left\{
     \left(\nonumber
     m_j\left(X_{4}^{(j)}\right)
    - m_j\left(X_{2}^{(j)}\right)\right)^2
    -I(j)\right\} \bigg|\pi(2)\notin\{1,3\}, \pi(3)\notin\{2,4\}
    \bigg]\\[0.5em]
       &+
   \widetilde c_{12}\mathbb{P}\left(\pi(1)\notin\{2,4\},\pi(2)\notin\{1,3\},
    \pi(3)=1,\pi(4)=2
    \right)\\
    &
    \mathbb{E}\bigg[
     \left\{
     \left(\nonumber
     m_j\left(X_1^{(j)}\right)
    - m_j\left(X_{\pi(1)}^{(j)}\right)\right)^2+2\varepsilon_1\left(\nonumber
     m_j\left(X_1^{(j)}\right)
    - m_j\left(X_{\pi(1)}^{(j)}\right)\right)
    -I(j)\right\}^2\\
    &
    \nonumber
     \left\{
     \left(\nonumber
     m_j\left(X_2^{(j)}\right)
    - m_j\left(X_{\pi(2)}^{(j)}\right)\right)^2+2\varepsilon_2\left(\nonumber
     m_j\left(X_2^{(j)}\right)
    - m_j\left(X_{\pi(2)}^{(j)}\right)\right)
    -I(j)\right\}^2\\  
      &
    \nonumber
     \left\{
     \left(\nonumber
     m_j\left(X_3^{(j)}\right)
    - m_j\left(X_{1}^{(j)}\right)\right)^2
    -I(j)\right\}\\
    &
    \quad\quad\quad
     \left\{
     \left(\nonumber
     m_j\left(X_{4}^{(j)}\right)
    - m_j\left(X_{2}^{(j)}\right)\right)^2
    -I(j)\right\} \bigg|\pi(1)\notin\{2,4\},\pi(2)\notin\{1,3\}
    \bigg]\\
           &+
   \widetilde c_{13}\mathbb{P}\left(\pi(1)\notin\{2,3,4\},\pi(2)=3,
    \pi(3)=4,\pi(4)\neq 1
    \right)\\
    & 
    \mathbb{E}\bigg[
     \left\{
     \left(\nonumber
     m_j\left(X_1^{(j)}\right)
    - m_j\left(X_{\pi(1)}^{(j)}\right)\right)^2+2\varepsilon_1\left(\nonumber
     m_j\left(X_1^{(j)}\right)
    - m_j\left(X_{\pi(1)}^{(j)}\right)\right)
    -I(j)\right\}^2\\
    &
    \nonumber
     \left\{
     \left(\nonumber
     m_j\left(X_2^{(j)}\right)
    - m_j\left(X_{3}^{(j)}\right)\right)^2+2\varepsilon_2\left(\nonumber
     m_j\left(X_2^{(j)}\right)
    - m_j\left(X_{3}^{(j)}\right)\right)
    -I(j)\right\}^2\\  
      &
    \nonumber
     \left\{
     \left(\nonumber
     m_j\left(X_3^{(j)}\right)
    - m_j\left(X_{4}^{(j)}\right)\right)^2
    -I(j)\right\}\\
    &
    \quad\quad\quad
     \left\{
     \left(\nonumber
     m_j\left(X_{4}^{(j)}\right)
    - m_j\left(X_{\pi(4)}^{(j)}\right)\right)^2
    -I(j)\right\} \bigg|\pi(1)\notin\{2,3,4\},\pi(4)\neq 1
    \bigg]\\
               &+
   \widetilde c_{14}\mathbb{P}\left(\pi(1)\notin\{2,3,4\},\pi(2)\neq 1,
    \pi(3)=4,\pi(4)=2
    \right)\\
    &   
    \mathbb{E}\bigg[
     \left\{
     \left(\nonumber
     m_j\left(X_1^{(j)}\right)
    - m_j\left(X_{\pi(1)}^{(j)}\right)\right)^2+2\varepsilon_1\left(\nonumber
     m_j\left(X_1^{(j)}\right)
    - m_j\left(X_{\pi(1)}^{(j)}\right)\right)
    -I(j)\right\}^2\\
    &
    \nonumber
     \left\{
     \left(\nonumber
     m_j\left(X_2^{(j)}\right)
    - m_j\left(X_{\pi(2)}^{(j)}\right)\right)^2+2\varepsilon_2\left(\nonumber
     m_j\left(X_2^{(j)}\right)
    - m_j\left(X_{\pi(2)}^{(j)}\right)\right)
    -I(j)\right\}^2\\  
      &
    \nonumber
     \left\{
     \left(\nonumber
     m_j\left(X_3^{(j)}\right)
    - m_j\left(X_{4}^{(j)}\right)\right)^2
    -I(j)\right\}\\
    &
    \quad\quad\quad
     \left\{
     \left(\nonumber
     m_j\left(X_{4}^{(j)}\right)
    - m_j\left(X_{2}^{(j)}\right)\right)^2
    -I(j)\right\} \bigg|\pi(1)\notin\{2,3,4\},\pi(2)\neq 1
    \bigg]\\
                  &+
   \widetilde c_{15}\mathbb{P}\left(\pi(1)\notin\{2,3,4\},\pi(2)=3,
    \pi(3)=4,\pi(4)=2
    \right)\\
    &
    \mathbb{E}\bigg[
     \left\{
     \left(\nonumber
     m_j\left(X_1^{(j)}\right)
    - m_j\left(X_{\pi(1)}^{(j)}\right)\right)^2+2\varepsilon_1\left(\nonumber
     m_j\left(X_1^{(j)}\right)
    - m_j\left(X_{\pi(1)}^{(j)}\right)\right)
    -I(j)\right\}^2\\
    &
    \nonumber
     \left\{
     \left(\nonumber
     m_j\left(X_2^{(j)}\right)
    - m_j\left(X_{3}^{(j)}\right)\right)^2+2\varepsilon_2\left(\nonumber
     m_j\left(X_2^{(j)}\right)
    - m_j\left(X_{3}^{(j)}\right)\right)
    -I(j)\right\}^2\\  
      &
    \nonumber
     \left\{
     \left(\nonumber
     m_j\left(X_3^{(j)}\right)
    - m_j\left(X_{4}^{(j)}\right)\right)^2
    -I(j)\right\}\\
    &
    \quad\quad\quad
     \left\{
     \left(\nonumber
     m_j\left(X_{4}^{(j)}\right)
    - m_j\left(X_{2}^{(j)}\right)\right)^2
    -I(j)\right\} \bigg|\pi(1)\notin\{2,3,4\}
    \bigg]
    \end{align*}

    \begin{align*}
      &+
   \widetilde c_{16}\mathbb{P}\left(\pi(1)\notin\{2,3,4\},\pi(2)=3,
    \pi(3)\notin\{1,4\},\pi(4)=2
    \right)\\
    &
    \mathbb{E}\bigg[
     \left\{
     \left(\nonumber
     m_j\left(X_1^{(j)}\right)
    - m_j\left(X_{\pi(1)}^{(j)}\right)\right)^2+2\varepsilon_1\left(\nonumber
     m_j\left(X_1^{(j)}\right)
    - m_j\left(X_{\pi(1)}^{(j)}\right)\right)
    -I(j)\right\}^2\\
    &
    \nonumber
     \left\{
     \left(\nonumber
     m_j\left(X_2^{(j)}\right)
    - m_j\left(X_{3}^{(j)}\right)\right)^2+2\varepsilon_2\left(\nonumber
     m_j\left(X_2^{(j)}\right)
    - m_j\left(X_{3}^{(j)}\right)\right)
    -I(j)\right\}^2\\  
      &
    \nonumber
     \left\{
     \left(\nonumber
     m_j\left(X_3^{(j)}\right)
    - m_j\left(X_{\pi(3)}^{(j)}\right)\right)^2
    -I(j)\right\}\\
    &
    \quad\quad\quad
     \left\{
     \left(\nonumber
     m_j\left(X_{4}^{(j)}\right)
    - m_j\left(X_{2}^{(j)}\right)\right)^2
    -I(j)\right\} \bigg|\pi(1)\notin\{2,3,4\},\pi(3)\notin\{1,4\}
    \bigg]\\
      &+
   \widetilde c_{17}\mathbb{P}\left(\pi(1)=2,\pi(2)=3,
    \pi(3)=4,\pi(4)\notin\{1,2\}
    \right)\\
    & 
    \mathbb{E}\bigg[
     \left\{
     \left(\nonumber
     m_j\left(X_1^{(j)}\right)
    - m_j\left(X_{2}^{(j)}\right)\right)^2+2\varepsilon_1\left(\nonumber
     m_j\left(X_1^{(j)}\right)
    - m_j\left(X_{2}^{(j)}\right)\right)
    -I(j)\right\}^2\\
    &
    \nonumber
     \left\{
     \left(\nonumber
     m_j\left(X_2^{(j)}\right)
    - m_j\left(X_{3}^{(j)}\right)\right)^2+2\varepsilon_2\left(\nonumber
     m_j\left(X_2^{(j)}\right)
    - m_j\left(X_{3}^{(j)}\right)\right)
    -I(j)\right\}^2\\  
      &
    \nonumber
     \left\{
     \left(\nonumber
     m_j\left(X_3^{(j)}\right)
    - m_j\left(X_{4}^{(j)}\right)\right)^2
    -I(j)\right\}\\
    &
    \quad\quad\quad
     \left\{
     \left(\nonumber
     m_j\left(X_{4}^{(j)}\right)
    - m_j\left(X_{\pi(4)}^{(j)}\right)\right)^2
    -I(j)\right\} \bigg|\pi(4)\notin\{1,2\}
    \bigg]\\
        &+
   \widetilde c_{18}\mathbb{P}\left(\pi(1)=2,\pi(2)=3,
    \pi(3)=4,\pi(4)=1
    \right)\\
    &\quad   
    \mathbb{E}\bigg[
     \left\{
     \left(\nonumber
     m_j\left(X_1^{(j)}\right)
    - m_j\left(X_{2}^{(j)}\right)\right)^2+2\varepsilon_1\left(\nonumber
     m_j\left(X_1^{(j)}\right)
    - m_j\left(X_{2}^{(j)}\right)\right)
    -I(j)\right\}^2\\
    &
    \nonumber
     \left\{
     \left(\nonumber
     m_j\left(X_2^{(j)}\right)
    - m_j\left(X_{3}^{(j)}\right)\right)^2+2\varepsilon_2\left(\nonumber
     m_j\left(X_2^{(j)}\right)
    - m_j\left(X_{3}^{(j)}\right)\right)
    -I(j)\right\}^2\\  
      &
    \nonumber
     \left\{
     \left(\nonumber
     m_j\left(X_3^{(j)}\right)
    - m_j\left(X_{4}^{(j)}\right)\right)^2
    -I(j)\right\}
     \left\{
     \left(\nonumber
     m_j\left(X_{4}^{(j)}\right)
    - m_j\left(X_{1}^{(j)}\right)\right)^2
    -I(j)\right\} 
    \bigg]\\
    &=\mathcal{O}\left(\frac{1}{\gamma_n}\right),
\end{align*}

\noindent because all the expectations are finite and

\begin{align*}
    \max\bigg\{&\mathbb{P}\left(\pi(1)\notin\{3,4\},\pi(2)\notin\{3,4\},
    \pi(3)=4,\pi(4)\notin\{1,2\}
    \right),~ \\
    &\mathbb{P}\left(\pi(1)=3,\pi(2)=4,
    \pi(3)\notin\{2,4\},\pi(4)\notin\{1,3\}
    \right),\\
    &\mathbb{P}\left(\pi(1)=3,\pi(2)\notin\{1,3\},
    \pi(3)\notin\{2,4\},\pi(4)=2
    \right),~\\
     &
    \mathbb{P}\left(\pi(1)\notin\{2,4\},\pi(2)\notin\{1,3\},
    \pi(3)=1,\pi(4)=2
    \right),\\
    &
    \mathbb{P}\left(\pi(1)\notin\{2,3,4\},\pi(2)=3,
    \pi(3)=4,\pi(4)\neq 1
    \right),~
    \end{align*}

    \begin{align*} 
        &
    \mathbb{P}\left(\pi(1)\notin\{2,3,4\},\pi(2)\neq 1,
    \pi(3)=4,\pi(4)=2
    \right)\\
        &\mathbb{P}\left(\pi(1)\notin\{2,3,4\},\pi(2)=3,
    \pi(3)=4,\pi(4)=2
    \right),~\\
    &
    \mathbb{P}\left(\pi(1)\notin\{2,3,4\},\pi(2)=3,
    \pi(3)\notin\{1,4\},\pi(4)=2
    \right)\\
    &
    \mathbb{P}\left(\pi(1)=2,\pi(2)=3,
    \pi(3)=4,\pi(4)\notin\{1,2\}
    \right),~\\
    &
    \mathbb{P}\left(\pi(1)=2,\pi(2)=3,
    \pi(3)=4,\pi(4)=1
    \right)\bigg\}\\
    &\leq
    \mathbb{P}\left(
    \pi(3)=4
    \right)
    =
    \mathcal{O}\left(\frac{1}{\gamma_n}\right).
\end{align*}

\noindent In addition
\begin{align*}
      &\mathbb{E}\bigg[
     \!\left\{ \!
     \left(\nonumber
     m_j\left(X_1^{(j)}\right)
    - m_j\left(X_{\pi(1)}^{(j)}\right)\right)^2+2\varepsilon_1\left(\nonumber
     m_j\left(X_1^{(j)}\right)
    - m_j\left(X_{\pi(1)}^{(j)}\right)\right)
    \!-\!I(j)\right\}^2\\
    &
    \nonumber
     \left\{
     \left(\nonumber
     m_j\left(X_2^{(j)}\right)
    - m_j\left(X_{\pi(2)}^{(j)}\right)\right)^2
    \!-\!I(j)\right\}\!
     \left\{
     \left(\nonumber
     m_j\left(X_3^{(j)}\right)
    - m_j\left(X_{\pi(3)}^{(j)}\right)\right)^2
    \!-\!I(j)\right\}\\
      &
    \nonumber
     \left\{
     \left(\nonumber
     m_j\left(X_{4}^{(j)}\right)
    - m_j\left(X_{\pi(4)}^{(j)}\right)\right)^2
    \!-\!I(j)\right\}
     \left\{
     \left(\nonumber
     m_j\left(X_{5}^{(j)}\right)
    - m_j\left(X_{\pi(5)}^{(j)}\right)\right)^2
    \!-\!I(j)\right\}\!
    \bigg]\\[0.5em]
     &=
    \widetilde c_{19}\mathbb{P}\left(\pi(1)\notin\{2,3,4,5\},\pi(2)=3,
    \pi(3)\notin\{1,4,5\},\pi(4)=5,\pi(5)\notin\{1,2,3\}
    \right)\\
    &
    \mathbb{E}\bigg[
     \left\{
     \left(\nonumber
     m_j\left(X_1^{(j)}\right)
    - m_j\left(X_{\pi(1)}^{(j)}\right)\right)^2+2\varepsilon_1\left(\nonumber
     m_j\left(X_1^{(j)}\right)
    - m_j\left(X_{\pi(1)}^{(j)}\right)\right)
    -I(j)\right\}^2\\
    &
    \nonumber
     \left\{
     \left(\nonumber
     m_j\left(X_2^{(j)}\right)
    - m_j\left(X_{3}^{(j)}\right)\right)^2
    -I(j)\right\}
     \left\{
     \left(\nonumber
     m_j\left(X_3^{(j)}\right)
    - m_j\left(X_{\pi(3)}^{(j)}\right)\right)^2
    -I(j)\right\}\\
      &
    \nonumber
     \left\{
     \left(\nonumber
     m_j\left(X_{4}^{(j)}\right)
    - m_j\left(X_{5}^{(j)}\right)\right)^2
    -I(j)\right\} \left\{
     \left(\nonumber
     m_j\left(X_{5}^{(j)}\right)
    - m_j\left(X_{\pi(5)}^{(j)}\right)\right)^2
    -I(j)\right\}\\
    &
    \quad \bigg|
    \pi(1)\notin\{2,3,4,5\},
 \pi(3)\notin\{1,4,5\},
\pi(5)\notin\{1,2,3\}
    \bigg]\\
    &+
      \widetilde c_{20}\mathbb{P}\left(\pi(1)\notin\{2,3,4,5\},\pi(2)=3,
    \pi(3)=4,\pi(4)=5,\pi(5)\neq 1
    \right)\\
    &
    \mathbb{E}\bigg[
     \left\{
     \left(\nonumber
     m_j\left(X_1^{(j)}\right)
    - m_j\left(X_{\pi(1)}^{(j)}\right)\right)^2+2\varepsilon_1\left(\nonumber
     m_j\left(X_1^{(j)}\right)
    - m_j\left(X_{\pi(1)}^{(j)}\right)\right)
    -I(j)\right\}^2\\
    &
    \nonumber
     \left\{
     \left(\nonumber
     m_j\left(X_2^{(j)}\right)
    - m_j\left(X_{3}^{(j)}\right)\right)^2
    -I(j)\right\}
     \left\{
     \left(\nonumber
     m_j\left(X_3^{(j)}\right)
    - m_j\left(X_{4}^{(j)}\right)\right)^2
    -I(j)\right\}\\
      &
    \nonumber
     \left\{
     \left(\nonumber
     m_j\left(X_{4}^{(j)}\right)
    - m_j\left(X_{5}^{(j)}\right)\right)^2
    -I(j)\right\}\\
    &
     \left\{
     \left(\nonumber
     m_j\left(X_{5}^{(j)}\right)
    - m_j\left(X_{\pi(5)}^{(j)}\right)\right)^2
    -I(j)\right\}\bigg|
    \pi(1)\notin\{2,3,4,5\},
\pi(5)\notin\{1\}
    \bigg]\\
        &+
     \widetilde c_{21}\mathbb{P}\left(\pi(1)\notin\{3,4,5\},\pi(2)=1,
    \pi(3)=4,\pi(4)=5,\pi(5)\notin\{1,2\}
    \right)\\
    &
    \mathbb{E}\bigg[
     \left\{
     \left(\nonumber
     m_j\left(X_1^{(j)}\right)
    - m_j\left(X_{\pi(1)}^{(j)}\right)\right)^2+2\varepsilon_1\left(\nonumber
     m_j\left(X_1^{(j)}\right)
    - m_j\left(X_{\pi(1)}^{(j)}\right)\right)
    -I(j)\right\}^2
\end{align*}

\begin{align*}
    &
    \nonumber
     \left\{
     \left(\nonumber
     m_j\left(X_2^{(j)}\right)
    - m_j\left(X_{1}^{(j)}\right)\right)^2
    -I(j)\right\}
     \left\{
     \left(\nonumber
     m_j\left(X_3^{(j)}\right)
    - m_j\left(X_{4}^{(j)}\right)\right)^2
    -I(j)\right\}\\
      &
    \nonumber
     \left\{
     \left(\nonumber
     m_j\left(X_{4}^{(j)}\right)
    - m_j\left(X_{5}^{(j)}\right)\right)^2
    -I(j)\right\}\\
    &
     \left\{
     \left(\nonumber
     m_j\left(X_{5}^{(j)}\right)
    - m_j\left(X_{\pi(5)}^{(j)}\right)\right)^2
    -I(j)\right\}\bigg|
    \pi(1)\notin\{3,4,5\},
\pi(5)\notin\{1,2,3\}
    \bigg]\\
     &+
     \widetilde c_{22}\mathbb{P}\left(\pi(1)=2,\pi(2)\notin\{3,4,5\},
    \pi(3)=4,\pi(4)=5,\pi(5)\notin\{1,2\}
    \right)\\
    &
    \mathbb{E}\bigg[
     \left\{
     \left(\nonumber
     m_j\left(X_1^{(j)}\right)
    - m_j\left(X_{2}^{(j)}\right)\right)^2+2\varepsilon_1\left(\nonumber
     m_j\left(X_1^{(j)}\right)
    - m_j\left(X_{2}^{(j)}\right)\right)
    -I(j)\right\}^2\\
    &
    \nonumber
     \left\{
     \left(\nonumber
     m_j\left(X_2^{(j)}\right)
    - m_j\left(X_{\pi(2)}^{(j)}\right)\right)^2
    -I(j)\right\}
     \left\{
     \left(\nonumber
     m_j\left(X_3^{(j)}\right)
    - m_j\left(X_{4}^{(j)}\right)\right)^2
    -I(j)\right\}\\
      &
    \nonumber
     \left\{
     \left(\nonumber
     m_j\left(X_{4}^{(j)}\right)
    - m_j\left(X_{5}^{(j)}\right)\right)^2
    -I(j)\right\}\\
    &
     \left\{
     \left(\nonumber
     m_j\left(X_{5}^{(j)}\right)
    - m_j\left(X_{\pi(5)}^{(j)}\right)\right)^2
    -I(j)\right\}\bigg|
    \pi(2)\notin\{3,4,5\},
\pi(5)\notin\{1,2,3\}
    \bigg]\\
    &+
     \widetilde c_{23}\mathbb{P}\left(\pi(1)=2,\pi(2)=3,
    \pi(3)\notin\{4,5\},\pi(4)=5,\pi(5)\notin\{1,2,3\}
    \right)\\
    &
    \mathbb{E}\bigg[
     \left\{
     \left(\nonumber
     m_j\left(X_1^{(j)}\right)
    - m_j\left(X_{2}^{(j)}\right)\right)^2+2\varepsilon_1\left(\nonumber
     m_j\left(X_1^{(j)}\right)
    - m_j\left(X_{2}^{(j)}\right)\right)
    -I(j)\right\}^2\\
    &
    \nonumber
     \left\{
     \left(\nonumber
     m_j\left(X_2^{(j)}\right)
    - m_j\left(X_{3}^{(j)}\right)\right)^2
    -I(j)\right\}
     \left\{
     \left(\nonumber
     m_j\left(X_3^{(j)}\right)
    - m_j\left(X_{4}^{(j)}\right)\right)^2
    -I(j)\right\}\\
    &
    \nonumber
     \left\{
     \left(\nonumber
     m_j\left(X_{4}^{(j)}\right)
    - m_j\left(X_{5}^{(j)}\right)\right)^2
    -I(j)\right\}\\
    &
     \left\{
     \left(\nonumber
     m_j\left(X_{5}^{(j)}\right)
    - m_j\left(X_{\pi(5)}^{(j)}\right)\right)^2
    -I(j)\right\}\bigg|
    \pi(3)\notin\{4,5\},
\pi(5)\notin\{1,2,3\}
    \bigg]\\
     &+
     \widetilde c_{24}\mathbb{P}\left(\pi(1)\notin\{4,5\},\pi(2)=1,
    \pi(3)=2,\pi(4)=5,\pi(5)\notin\{1,2,3\}
    \right)\\
    &
    \mathbb{E}\bigg[
     \left\{
     \left(\nonumber
     m_j\left(X_1^{(j)}\right)
    - m_j\left(X_{\pi(1)}^{(j)}\right)\right)^2+2\varepsilon_1\left(\nonumber
     m_j\left(X_1^{(j)}\right)
    - m_j\left(X_{\pi(1)}^{(j)}\right)\right)
    -I(j)\right\}^2\\
    &
    \nonumber
     \left\{
     \left(\nonumber
     m_j\left(X_2^{(j)}\right)
    - m_j\left(X_{1}^{(j)}\right)\right)^2
    -I(j)\right\}
     \left\{
     \left(\nonumber
     m_j\left(X_3^{(j)}\right)
    - m_j\left(X_{2}^{(j)}\right)\right)^2
    -I(j)\right\}\\
    &
    \nonumber
     \left\{
     \left(\nonumber
     m_j\left(X_{4}^{(j)}\right)
    - m_j\left(X_{5}^{(j)}\right)\right)^2
    -I(j)\right\}\\
    &
     \left\{
     \left(\nonumber
     m_j\left(X_{5}^{(j)}\right)
    - m_j\left(X_{\pi(5)}^{(j)}\right)\right)^2
    -I(j)\right\}\bigg|
    \pi(1)\notin\{4,5\},
\pi(5)\notin\{1,2,3\}
    \bigg]\\
       &+
     \widetilde c_{25}\mathbb{P}\left(\pi(1)=2,\pi(2)=3,
    \pi(3)=4,\pi(4)=5,\pi(5)\neq 1
    \right)\\
    &
    \mathbb{E}\bigg[
     \left\{
     \left(\nonumber
     m_j\left(X_1^{(j)}\right)
    - m_j\left(X_{2}^{(j)}\right)\right)^2+2\varepsilon_1\left(\nonumber
     m_j\left(X_1^{(j)}\right)
    - m_j\left(X_{2}^{(j)}\right)\right)
    -I(j)\right\}^2\\
    &
    \nonumber
     \left\{
     \left(\nonumber
     m_j\left(X_2^{(j)}\right)
    - m_j\left(X_{3}^{(j)}\right)\right)^2
    -I(j)\right\}
     \left\{
     \left(\nonumber
     m_j\left(X_3^{(j)}\right)
    - m_j\left(X_{4}^{(j)}\right)\right)^2
    -I(j)\right\}\\
    &
    \nonumber
     \left\{
     \left(\nonumber
     m_j\left(X_{4}^{(j)}\right)
    - m_j\left(X_{5}^{(j)}\right)\right)^2
    -I(j)\right\}\\
     &\left\{
     \left(\nonumber
     m_j\left(X_{5}^{(j)}\right)
    - m_j\left(X_{\pi(5)}^{(j)}\right)\right)^2
    -I(j)\right\}\bigg|
    \pi(5)\neq 1
    \bigg]
    \end{align*}
    
    \begin{align*}
         &+
     \widetilde c_{26}\mathbb{P}\left(\pi(1)\notin\{2,3,4,5\},\pi(2)=1,
    \pi(3)=2,\pi(4)=3,\pi(5)=4
    \right)\\
    &
    \mathbb{E}\bigg[
     \left\{
     \left(\nonumber
     m_j\left(X_1^{(j)}\right)
    - m_j\left(X_{\pi(1)}^{(j)}\right)\right)^2+2\varepsilon_1\left(\nonumber
     m_j\left(X_1^{(j)}\right)
    - m_j\left(X_{\pi(1)}^{(j)}\right)\right)
    -I(j)\right\}^2\\
    &
    \nonumber
     \left\{
     \left(\nonumber
     m_j\left(X_2^{(j)}\right)
    - m_j\left(X_{1}^{(j)}\right)\right)^2
    -I(j)\right\}
     \left\{
     \left(\nonumber
     m_j\left(X_3^{(j)}\right)
    - m_j\left(X_{2}^{(j)}\right)\right)^2
    -I(j)\right\}\\
    &
\quad\quad
    \nonumber
     \left\{
     \left(\nonumber
     m_j\left(X_{4}^{(j)}\right)
    - m_j\left(X_{3}^{(j)}\right)\right)^2
    -I(j)\right\}\\
    &\quad\quad\quad\quad
     \left\{
     \left(\nonumber
     m_j\left(X_{5}^{(j)}\right)
    - m_j\left(X_{4}^{(j)}\right)\right)^2
    -I(j)\right\}\bigg|
    \pi(1)\notin\{2,3,4,5\}
    \bigg]\\
           &+
     \widetilde c_{27}\mathbb{P}\left(\pi(1)=2,\pi(2)=3,
    \pi(3)=4,\pi(4)=5,\pi(5)= 1
    \right)\\
    &
    \mathbb{E}\bigg[
     \left\{
     \left(\nonumber
     m_j\left(X_1^{(j)}\right)
    - m_j\left(X_{2}^{(j)}\right)\right)^2+2\varepsilon_1\left(\nonumber
     m_j\left(X_1^{(j)}\right)
    - m_j\left(X_{2}^{(j)}\right)\right)
    -I(j)\right\}^2\\
    &
    \nonumber
     \left\{
     \left(\nonumber
     m_j\left(X_2^{(j)}\right)
    - m_j\left(X_{3}^{(j)}\right)\right)^2
    -I(j)\right\}
     \left\{
     \left(\nonumber
     m_j\left(X_3^{(j)}\right)
    - m_j\left(X_{4}^{(j)}\right)\right)^2
    -I(j)\right\}\\
    &
    \nonumber
     \left\{
     \left(\nonumber
     m_j\left(X_{4}^{(j)}\right)
    - m_j\left(X_{5}^{(j)}\right)\right)^2
    -I(j)\right\}
     \left\{
     \left(\nonumber
     m_j\left(X_{5}^{(j)}\right)
    - m_j\left(X_{1}^{(j)}\right)\right)^2
    -I(j)\right\}
    \bigg]\\
    &=
    \mathcal{O}\left(\frac{1}{\gamma_n^2}\right),
\end{align*}

\noindent because
\begin{align*}
    \max\Big\{&\mathbb{P}\left(\pi(1)\notin\{2,3,4,5\},\pi(2)=3,
    \pi(3)\notin\{1,4,5\},\pi(4)=5,\pi(5)\notin\{1,2,3\}
    \right),\\
    &\mathbb{P}\left(\pi(1)\notin\{2,3,4,5\},\pi(2)=3,
    \pi(3)=4,\pi(4)=5,\pi(5)\neq 1
    \right),\\
    &\mathbb{P}\left(\pi(1)\notin\{3,4,5\},\pi(2)=1,
    \pi(3)=4,\pi(4)=5,\pi(5)\notin\{1,2\}
    \right),\\
    &\mathbb{P}\left(\pi(1)=2,\pi(2)=3,
    \pi(3)\notin\{4,5\},\pi(4)=5,\pi(5)\notin\{1,2,3\}
    \right),\\
    &\mathbb{P}\left(\pi(1)\notin\{4,5\},\pi(2)=1,
    \pi(3)=2,\pi(4)=5,\pi(5)\notin\{1,2,3\}
    \right),\\
    &\mathbb{P}\left(\pi(1)=2,\pi(2)=3,
    \pi(3)=4,\pi(4)=5,\pi(5)\neq 1
    \right)\\
    &\mathbb{P}\left(\pi(1)\notin\{2,3,4,5\},\pi(2)=1,
    \pi(3)=2,\pi(4)=3,\pi(5)=4
    \right),\\
    &\mathbb{P}\left(\pi(1)=2,\pi(2)=3,
    \pi(3)=4,\pi(4)=5,\pi(5)= 1
    \right)\bigg\}\\
    &\leq\mathbb{P}\left(\pi(2)=3,\pi(4)=5
    \right)=\frac{!(\gamma_n-2)}{!\gamma_n}
    =\mathcal{O}\left(\frac{1}{\gamma_n}\right)
\end{align*}
as in (\ref{theorem 2.2 permu schnitt}) and (\ref{theorem 2.2 permu vereinigung}).
For the expected value resulting from the sixfold sum, we get
\begin{align*}
       &\mathbb{E}\bigg[\!
     \left\{\!
     \left(\nonumber
     m_j\left(X_1^{(j)}\right)
    - m_j\left(X_{\pi(1)}^{(j)}\right)\right)^2
    \!-\!I(j)\right\}\!\left\{
     \left(\nonumber
     m_j\left(X_2^{(j)}\right)
    - m_j\left(X_{\pi(2)}^{(j)}\right)\right)^2
    \!-\!I(j)\!\right\}\\
      &
     \left\{\!
     \left(\nonumber
     m_j\left(X_3^{(j)}\right)
    - m_j\left(X_{\pi(3)}^{(j)}\right)\right)^2
    \!-\!I(j)\right\}\!
     \left\{
     \left(\nonumber
     m_j\left(X_{4}^{(j)}\right)
    - m_j\left(X_{\pi(4)}^{(j)}\right)\right)^2
    \!-\!I(j)\right\}
\end{align*}
\begin{align*}
    &
    \nonumber
     \left\{\!
     \left(\nonumber
     m_j\left(X_{5}^{(j)}\right)
    - m_j\left(X_{\pi(5)}^{(j)}\right)\right)^2
    \!-\!I(j)\right\}\!
         \left\{
     \left(\nonumber
     m_j\left(X_{6}^{(j)}\right)
    - m_j\left(X_{\pi(6)}^{(j)}\right)\right)^2
    \!-\!I(j)\right\}\!
    \bigg]\\
    &=
    \widetilde c_{28} \mathbb{P}\big(\pi(1)=2,\pi(2)\notin\{3,4,5,6\},
    \pi(3)=4,\\
    &\quad\quad\quad\quad
    \quad\quad\quad\quad
    \quad\quad\quad\quad
    \pi(4)\notin\{1,2,5,6\},\pi(5)=6,\pi(6)\notin\{1,2,3,4\}
    \big)\\
    &
    \mathbb{E}\bigg[\!
     \left\{\!
     \left(\nonumber
     m_j\left(X_1^{(j)}\right)
    - m_j\left(X_{2}^{(j)}\right)\right)^2
    \!-\!I(j)\right\}\left\{
     \left(\nonumber
     m_j\left(X_2^{(j)}\right)
    - m_j\left(X_{\pi(2)}^{(j)}\right)\right)^2
    -I(j)\right\}\\
        &
     \left\{
     \left(\nonumber
     m_j\left(X_3^{(j)}\right)
    - m_j\left(X_{4}^{(j)}\right)\right)^2
    -I(j)\right\}
     \left\{
     \left(\nonumber
     m_j\left(X_{4}^{(j)}\right)
    - m_j\left(X_{\pi(4)}^{(j)}\right)\right)^2
    -I(j)\right\}\\
      &
    \nonumber
     \left\{
     \left(\nonumber
     m_j\left(X_{5}^{(j)}\right)
    - m_j\left(X_{6}^{(j)}\right)\right)^2
    -I(j)\right\}
         \left\{
     \left(\nonumber
     m_j\left(X_{6}^{(j)}\right)
    - m_j\left(X_{\pi(6)}^{(j)}\right)\right)^2
    -I(j)\right\}\\
    &
    \quad\quad
    \bigg|
    \pi(2)\notin\{3,4,5,6\},
\pi(4)\notin\{1,2,5,6\},
\pi(6)\notin\{1,2,3,4\}
    \bigg]\\
        &+
    \widetilde c_{29} \mathbb{P}\left(\pi(1)=2,\pi(2)=3,
    \pi(3)\notin\{4,5,6\},\pi(4)=5,\pi(5)=6,\pi(6)\notin\{1,2,3\}
    \right)\\
    &
    \mathbb{E}\bigg[
    \! \left\{\!
     \left(\nonumber
     m_j\left(X_1^{(j)}\right)
    - m_j\left(X_{2}^{(j)}\right)\right)^2
    \!-\!I(j)\right\}\!\left\{
     \left(\nonumber
     m_j\left(X_2^{(j)}\right)
    - m_j\left(X_{3}^{(j)}\right)\right)^2
    \!-\!I(j)\right\}\\
        &
     \left\{
     \left(\nonumber
     m_j\left(X_3^{(j)}\right)
    - m_j\left(X_{\pi(3)}^{(j)}\right)\right)^2
    \!-\!I(j)\right\}\!
     \left\{
     \left(\nonumber
     m_j\left(X_{4}^{(j)}\right)
    - m_j\left(X_{5}^{(j)}\right)\right)^2
    \!-\!I(j)\right\}\\
      &
    \nonumber
     \left\{
     \left(\nonumber
     m_j\left(X_{5}^{(j)}\right)
    - m_j\left(X_{6}^{(j)}\right)\right)^2
    \!-\!I(j)\right\}\!
         \left\{
     \left(\nonumber
     m_j\left(X_{6}^{(j)}\right)
    - m_j\left(X_{\pi(6)}^{(j)}\right)\right)^2
    \!-\!I(j)\right\}\\
    &
    \quad\quad\quad\quad
    \bigg|
    \pi(3)\notin\{4,5,6\},
\pi(6)\notin\{1,2,3\}
    \bigg]\\
        &+
    \widetilde c_{30} \mathbb{P}\left(\pi(1)=2,\pi(2)=3,
    \pi(3)=4,\pi(4)\notin\{5,6\},\pi(5)=6,\pi(6)\notin\{1,2,3,4\}
    \right)\\
    &
    \mathbb{E}\bigg[
     \!\left\{\!
     \left(\nonumber
     m_j\left(X_1^{(j)}\right)
    - m_j\left(X_{2}^{(j)}\right)\right)^2
    \!-\!I(j)\right\}\!\left\{
     \left(\nonumber
     m_j\left(X_2^{(j)}\right)
    - m_j\left(X_{3}^{(j)}\right)\right)^2
    \!-\!I(j)\right\}\\
        &
     \left\{
     \left(\nonumber
     m_j\left(X_3^{(j)}\right)
    - m_j\left(X_{4}^{(j)}\right)\right)^2
    \!-\!I(j)\right\}\!
     \left\{
     \left(\nonumber
     m_j\left(X_{4}^{(j)}\right)
    - m_j\left(X_{\pi(4)}^{(j)}\right)\right)^2
    \!-\!I(j)\right\}\\
      &
    \nonumber
     \left\{
     \left(\nonumber
     m_j\left(X_{5}^{(j)}\right)
    - m_j\left(X_{6}^{(j)}\right)\right)^2
    -I(j)\right\}
         \left\{
     \left(\nonumber
     m_j\left(X_{6}^{(j)}\right)
    - m_j\left(X_{\pi(6)}^{(j)}\right)\right)^2
    -I(j)\right\}\\
    &
    \quad\quad\quad\quad\quad
    \bigg|
    \pi(4)\notin\{5,6\},
\pi(6)\notin\{1,2,3,4\}
    \bigg]\\
        &+
    \widetilde c_{31} \mathbb{P}\left(\pi(1)=2,\pi(2)=3,
    \pi(3)=4,\pi(4)=5,\pi(5)=6,\pi(6)\notin\{1,2,3,4,5\}
    \right)\\
    &
    \mathbb{E}\bigg[\!
     \left\{
     \left(\nonumber
     m_j\left(X_1^{(j)}\right)
    - m_j\left(X_{2}^{(j)}\right)\right)^2
    -I(j)\right\}\!\left\{
     \left(\nonumber
     m_j\left(X_2^{(j)}\right)
    - m_j\left(X_{3}^{(j)}\right)\right)^2
    -I(j)\right\}\\
        &
     \left\{
     \left(\nonumber
     m_j\left(X_3^{(j)}\right)
    - m_j\left(X_{4}^{(j)}\right)\right)^2
    -I(j)\right\}\!
     \left\{
     \left(\nonumber
     m_j\left(X_{4}^{(j)}\right)
    - m_j\left(X_{5}^{(j)}\right)\right)^2
    -I(j)\right\}\\
      &
    \nonumber
     \left\{
     \left(\nonumber
     m_j\left(X_{5}^{(j)}\right)
    - m_j\left(X_{6}^{(j)}\right)\right)^2
    -I(j)\right\}\!
         \left\{
     \left(\nonumber
     m_j\left(X_{6}^{(j)}\right)
    - m_j\left(X_{\pi(6)}^{(j)}\right)\right)^2
    -I(j)\right\}\\
    &
    \quad\quad\quad\quad\quad 
    \bigg|
\pi(6)\notin\{1,2,3,4,\}
    \bigg]\\
    &=\mathcal{O}\left(\frac{1}{\gamma_n^3}\right),
\end{align*}
since 
\begin{align*}
    &\max\Big\{\mathbb{P}\big(\pi(1)=2,\pi(2)\notin\{3,4,5,6\},
    \pi(3)=4,\\
    &\quad\quad\quad\quad\quad\quad
    \pi(4)\notin\{1,2,5,6\},\pi(5)=6,\pi(6)\notin\{1,2,3,4\}
    \big)\!,\\
    &
     \mathbb{P}\left(\pi(1)=2,\pi(2)=3,
    \pi(3)\notin\{4,5,6\},\pi(4)=5,\pi(5)=6,\pi(6)\notin\{1,2,3\}
    \right)\!,\\
    &
    \mathbb{P}\left(\pi(1)=2,\pi(2)=3,
    \pi(3)=4,\pi(4)\notin\{5,6\},\pi(5)=6,\pi(6)\notin\{1,2,3,4\}
    \right)\!,\\
    &
     \mathbb{P}\left(\pi(1)=2,\pi(2)=3,
    \pi(3)=4,\pi(4)=5,\pi(5)=6,\pi(6)\notin\{1,2,3,4,5\}
    \right)
    \Big\}\\
    &\leq 
    \mathbb{P}\left(\pi(1)=2,
    \pi(3)=4,\pi(5)=6
    \right)=\frac{!(\gamma_n-3)}{\gamma_n}=\frac{1}{\gamma_n^3}+o(1).
    \end{align*}
This results in 
\begin{align*}
    \mathbb{E}\left[
    \left\vert h_{\gamma_n}-I(j)\right\vert^6
    \right]
    =
    \mathcal{O}\left(\frac{1}{\gamma_n^3}\right),
\end{align*}
which implies that

\[
        \frac{\mathbb{E}\left[\vert h_{\gamma_n}-I(j) \vert^{6}
    \right]}{
    \mathbb{E}\left[\vert h_{\gamma_n}-I(j) \vert^{3} 
    \right]^2}
\]
is uniformly bounded. Hence,
\begin{align*}
        \frac{\mathbb{E}\left[\vert h_{\gamma_n}-I(j) \vert^{2k}
    \right]}{
    \mathbb{E}\left[\vert h_{\gamma_n}-I(j) \vert^{k} 
    \right]^2}
\end{align*}
is uniformly bounded for $k=2,3$.
Due to (A1)-(A3), \Cref{Lemma 2} can be applied
delivering
\begin{align*}
    \frac{\gamma_n}{n}\frac{\zeta_{\gamma_n}}{\gamma_n\zeta_{1,\gamma_n}}
    \leq
     \frac{\gamma_n^2}{n}\frac{\widetilde K}{\gamma_n^2\zeta_{1,\gamma_n}}
\end{align*}
for a $\widetilde K >0$. \Cref{Lemma 1} and $\gamma_n/\sqrt{n}\rightarrow 0$ as $n\rightarrow\infty$
result in
\begin{align*}
    \lim_{n\rightarrow\infty} \frac{\gamma_n}{n}\frac{\zeta_{\gamma_n}}{\gamma_n\zeta_{1,\gamma_n}}=0.
\end{align*}
Hence, \Cref{Theorem 2.1 von Peng} can be applied, which completes the proof.
\end{proof}

\section{Proof of Theorem 4.3}\label{appB}
\begin{proof}
    We start with the decomposition
    \begin{align*}
      \frac{ I_{n,\widehat M}^{OOB}(j)-I(j)}{\sqrt{\gamma_n^2\zeta_{1,\gamma_n}/n+\zeta_{\gamma_n}/M_n}}  
      =
        \frac{ I_{n,\widehat M}^{OOB}(j)-\widetilde I_{n,\widehat M}^{OOB}(j)}{\sqrt{\gamma_n^2\zeta_{1,\gamma_n}/n+\zeta_{\gamma_n}/M_n}}  
        +
          \frac{ \widetilde I_{n,\widehat M}^{OOB}(j)-I(j)}{\sqrt{\gamma_n^2\zeta_{1,\gamma_n}/n+\zeta_{\gamma_n}/M_n}}  .
    \end{align*}
    Due to \Cref{Theorem 3}, we already know that the second summand is asymptotically normal. Hence it remains to show that the first term vanishes for growing $M$ and $n$. The first term in the denominator is asymptotically equivalent to $\kappa/n$ by \Cref{Lemma 1}. From the proof of \Cref{Theorem 3}, we can observe that there is a constant $c_1$ such that
    $\zeta_{\gamma_n}=\mathbb{V}\text{ar}(h_{\gamma_n})=\mathbb{E}[\vert h_{\gamma_n}-I(j)\vert^2]\sim\frac{c_1}{\gamma_n}$. Thus
$
\frac{\zeta_{\gamma_n}}{M}\sim\frac{c_1}{\gamma_nM}
$
 holds true. The assumption $\frac{M\gamma_n\log(a_n)^2\log(a_np)}{a_n^{\phi(\lambda)}}\rightarrow 0$ implies
 $\frac{M\gamma_n}{a_n}\rightarrow 0$ as well as $a_n\sim n$, since $a_n=n-\gamma_n$.
Hence,
 \begin{align*}
     \sqrt{\frac{\gamma_n^2\zeta_{1,\gamma_n}}{n}+\frac{\zeta_{\gamma_n}}{M}}\sim\sqrt{\frac{c_1}{\gamma_n M}}.
 \end{align*}
For the numerator, we observe
     \begin{align}
     \nonumber
         &\left(I_{n,M}^{OOB}(j)-\widetilde I_{n,M}^{OOB}(j)\right)\\
         =&
         \nonumber
            \frac{1}{M\gamma_n}
    \sum_{t=1}^M
    \sum_{i\in\mathcal{D}_n^{-(t)}}
    \Big\{
        \big(
        Y_i-{m}_{n,1}(\textbf{X}_i^{\pi_{j,t}},
        \mathbf{\Theta}_t,\mathcal{D}_n)
        \big)^2
        -
        \big(
        Y_i-{m}_{n,1}(\textbf{X}_i,
          \mathbf{\Theta}_t,\mathcal{D}_n)
        \big)^2
    \Big\}
    \Big\}\\
    &-\nonumber
          \frac{1}{M\gamma_n}
    \sum_{t=1}^M
    \sum_{i\in\mathcal{D}_n^{-(t)}}
    \Big\{
        \big(
        Y_i-\widetilde{m}(\textbf{X}_i^{\pi_{j,t}})
        \big)^2
        -
        \big(
        Y_i-\widetilde{m}(\textbf{X}_i)
        \big)^2
    \Big\}
    \Big\}\\
    &=
    \nonumber
    \frac{1}{M\gamma_n}
    \sum_{t=1}^M
    \sum_{i\in\mathcal{D}_n^{-(t)}}
    \bigg\{
    {m}_{n,1}(\textbf{X}_i^{\pi_{j,t}},
          \mathbf{\Theta}_t,\mathcal{D}_n)^2
          -
          \widetilde{m}(\textbf{X}_i^{\pi_{j,t}})^2
          -
           {m}_{n,1}(\textbf{X}_i,
          \mathbf{\Theta}_t,\mathcal{D}_n)^2\\
          &\quad\quad\quad\quad\quad\quad\quad\quad\quad\quad+
              \widetilde{m}(\textbf{X}_i)^2
    +
    2Y_i\left(\nonumber
     \widetilde{m}(\textbf{X}_i^{\pi_{j,t}})
     -
    {m}_{n,1}(\textbf{X}_i^{\pi_{j,t}},
          \mathbf{\Theta}_t,\mathcal{D}_n)
    \right)
    \\
    &\quad\quad\quad\quad\quad\quad\quad\quad\quad\quad\quad\quad\quad
    -2Y_i\left(
     \widetilde{m}(\textbf{X}_i)
     -
    {m}_{n,1}(\textbf{X}_i,
          \mathbf{\Theta}_t,\mathcal{D}_n)
    \right)\bigg\}.
    \label{Theorem 4 diff decomposed}
     \end{align}
We will handle the terms in 
(\ref{Theorem 4 diff decomposed}) seperately
and start with (A3), (A4) to get
\begin{align*}
&\left\vert\widetilde{m}(\mathbf{X}_i)^2-m_{n,1}(\mathbf{X}_i,\mathbf{\Theta}_t,\mathcal{D}_n)^2 \right\vert\\
&= \left\vert\widetilde{m}(\mathbf{X}_i)+m_{n,1}(\mathbf{X}_i,\mathbf{\Theta}_t,\mathcal{D}_n) \right\vert
\cdot\left\vert\widetilde{m}(\mathbf{X}_i)-m_{n,1}(\mathbf{X}_i,\mathbf{\Theta}_t,\mathcal{D}_n) \right\vert\\
&\leq
(2K+K_{\varepsilon})\cdot\left\vert\widetilde{m}(\mathbf{X}_i)-m_{n,1}(\mathbf{X}_i,\mathbf{\Theta}_t,\mathcal{D}_n) \right\vert
\end{align*}
with $K$ from (A3) and $K_{\varepsilon}$ from (A4). Let $\xi>0$, then, using the Markov inequality for
the joint porbability measure $\mathbb{P}=\mathbb{P}_{\mathbf{X},\mathbf{\Theta}}$ of $\mathbf{\Theta}$ and $\mathbf{X}$,
we get
\begin{align}
    &\mathbb{P}\left(\sqrt{\gamma_n M}\left\vert\widetilde{m}(\mathbf{X}_i)^2-m_{n,1}(\mathbf{X}_i,\mathbf{\Theta}_t,\mathcal{D}_n)^2 \right\vert >\xi\right)\nonumber\\
    &\leq
    \mathbb{P}\left(
    (2K+K_{\varepsilon})\cdot 
    \sqrt{\gamma_n M}\left\vert\widetilde{m}(\mathbf{X}_i)-m_{n,1}(\mathbf{X}_i,\mathbf{\Theta}_t,\mathcal{D}_n) \right\vert >\xi\right)\nonumber\\
    &=
    \mathbb{P}\left(
    \sqrt{\gamma_n M}\left\vert\widetilde{m}(\mathbf{X}_i)-m_{n,1}(\mathbf{X}_i,\mathbf{\Theta}_t,\mathcal{D}_n) \right\vert >\frac{\xi}{2K+K_{\varepsilon}}\right)\nonumber\\
    &\leq (2K+K_{\varepsilon})^2\cdot \gamma_nM\cdot
    \frac{\mathbb{E}_{\mathbf{X},\mathbf{\Theta}}\left[\left\vert\widetilde{m}(\mathbf{X}_i)-m_{n,1}(\mathbf{X}_i,\mathbf{\Theta}_t,\mathcal{D}_n) \right\vert^2\right]}{\xi^2}.
\label{Theorem 4 Markov}
\end{align}
To handle this, we are going to use Theorem 2.3 from \cite{mazumder2024convergence}, which, under (A3), (A4) and (A7), states that for $\mathbf{X}_i\in[0,1]^p$, there is a $\delta\in(0,1)$
such that with probability at least $1-\delta$,
\begin{align}
    \mathbb{E}_{\mathbf{X}}\!\left[\left\vert\widetilde{m}(\mathbf{X}_i)\!-\!m_{n,1}(\mathbf{X}_i,\mathbf{\Theta}_t,\mathcal{D}_n) \right\vert^2\right]
   \leq
    C_{\lambda, K, K_{\varepsilon}}\!\frac{\log(a_n)^2\log(a_np)\!+\!\log(a_n)\log(1/\delta)}{a_n^{\phi(\lambda)}}    \label{Theorem 4 Baum ungleichung}
\end{align}
holds
for $a_n$ large enough if the tree depth $d=\lceil\log_2(a_n)/(1-\log_2(1-\lambda))\rceil$, where $\phi(\lambda):=\frac{-\log_2(1-\lambda)}{1-\log_2(1-\lambda)}$.
$C_{\lambda, K, K_{\varepsilon}}$ is a constant that only depends on $K, K_{\varepsilon}$ and $\lambda$ from (A7). In this case 
'$a_n$ large enough' means that $$\max\left\{\frac{2\exp(1)^2d}{a_n},\frac{\log(72p^d(n+1^{2d}/\delta))}{a_n}\right\}<\frac{3}{4}$$ holds true. Now, by choosing
$\delta=\frac{1}{a_n^{\phi(\lambda)}}$, combining (\ref{Theorem 4 Markov}), (\ref{Theorem 4 Baum ungleichung}) and the independence of $\mathbf{\Theta}_t$
and $\mathbf{X}_i$ delivers
\begin{align}
    &\gamma_nM\cdot\nonumber
        \mathbb{E}_{\mathbf{X},\mathbf{\Theta}}\left[\left\vert\widetilde{m}(\mathbf{X}_i)-m_{n,1}(\mathbf{X}_i,\mathbf{\Theta}_t,\mathcal{D}_n) \right\vert^2\right]\\
        &=
        \nonumber
      \gamma_nM
        \int
        \left\vert\widetilde{m}(\mathbf{x}_i)-m_{n,1}(\mathbf{x}_i,\mathbf{\theta},\mathcal{D}_n) \right\vert^2d\mathbb{P}_{\mathbf{X},\mathbf{\Theta}}(\mathbf{x},\mathbf{\theta})\\
        &=
          \nonumber
        \gamma_nM
        \int\int
        \left\vert\widetilde{m}(\mathbf{x}_i)-m_{n,1}(\mathbf{x}_i,\mathbf{\theta},\mathcal{D}_n) \right\vert^2d\mathbb{P}_{\mathbf{X}}(\mathbf{x})d\mathbb{P}_{\mathbf{\Theta}}(\mathbf{\theta})\\
          &\leq
          \nonumber
           \gamma_nM  C_{\lambda, K, K_{\varepsilon}}\frac{\log(a_n)^2\log(a_np)-\phi(\lambda)\log(a_n)\log(a_n)}{a_n^{\phi(\lambda)}}
            +
            \frac{\gamma_nM(2K+K_{\varepsilon})}{a_n^{\phi(\lambda)}}\\
            &\leq
                \gamma_nMC_{\lambda, K, K_{\varepsilon}}\frac{\log(a_n)^2\log(a_np)}{a_n^{\phi(\lambda)}}
            +
            \frac{\gamma_nM(2K+K_{\varepsilon})}{a_n^{\phi(\lambda)}}\label{Theorem 4 Erwartung abgeschätzt}
\end{align}
for $a_n$ large enough. The assumption $\frac{M\gamma_n\log(a_n)^2\log(a_np)}{a_n^{\phi(\lambda)}}\rightarrow 0$ implies that (\ref{Theorem 4 Erwartung abgeschätzt})
becomes arbitrary small. Thus, with (\ref{Theorem 4 Erwartung abgeschätzt}) we can see that
$\gamma_nM(\widetilde{m}(\mathbf{X}_i)^2-m_{n,1}(\mathbf{X}_i,\mathbf{\Theta},\mathcal{D}_n)^2)\overset{\mathbb{P}}{\longrightarrow}0$. The calculations in (\ref{Theorem 4 Markov}) can also be conducted for
\begin{align*}
    &\mathbb{P}_{\pi,\mathbf{X},\mathbf{\Theta}}\left(\sqrt{\gamma_n M}\left\vert\widetilde{m}(\mathbf{X}^{\pi_{j,t}}_i)^2-m_{n,1}(\mathbf{X}^{\pi_{j,t}}_i,\mathbf{\Theta}_t,\mathcal{D}_n)^2 \right\vert >\xi\right)\\
    &\leq
    (2K+K_{\varepsilon})^2\cdot \gamma_nM\cdot
    \frac{\mathbb{E}_{\pi,\mathbf{X},\mathbf{\Theta}}\left[\left\vert\widetilde{m}(\mathbf{X}^{\pi_{j,t}}_i)-m_{n,1}(\mathbf{X}^{\pi_{j,t}}_i,\mathbf{\Theta}_t,\mathcal{D}_n) \right\vert^2\right]}{\xi^2}.
\end{align*}
for the joint probability measure $\mathbb{P}_{\pi,\mathbf{X},\mathbf{\Theta}}$ of $\pi,\mathbf{X}$ and $\mathbf{\Theta}$.
Let $\mathfrak{D}_{\gamma_n}$ be the set of all derrangements of $\{1,\ldots, \gamma_n\}$
and for $(i_1,\ldots,i_{\gamma_n})\in \mathfrak{D}_{\gamma_n}$ denote $\mathbf{X}_i^{\pi_{j,t}}$ by $\mathbf{X}_i^{(i_1,\ldots, i_{\gamma_n})}$ when
$\pi_{j,t}$ takes the realization $(i_1,\ldots,i_{\gamma_n})$.
Then, the inependence of $\pi_{j,t},~\mathbf{X}_i$ and $\mathbf{\Theta}_t$ leads to
\begin{align*}
    &\gamma_nM\cdot
        \mathbb{E}_{\mathbf{X},\mathbf{\Theta}}\left[\left\vert\widetilde{m}(\mathbf{X}^{\pi_{j,t}}_i)-m_{n,1}(\mathbf{X}^{\pi_{j,t}}_i,\mathbf{\Theta}_t,\mathcal{D}_n) \right\vert^2\right]\\
        &=
        \nonumber
      \gamma_nM
        \int
        \left\vert\widetilde{m}\left(\mathbf{x}^{\pi_j}_i\right)-m_{n,1}\left(\mathbf{x}^{\pi_j}_i,\mathbf{\theta},\mathcal{D}_n\right) \right\vert^2d\mathbb{P}_{\pi,\mathbf{X},\mathbf{\Theta}}(\pi_j,\mathbf{x}_i,\mathbf{\theta})\\
        &=
       \nonumber
      \gamma_nM
        \int\int\int 
        \left\vert\widetilde{m}\left(\mathbf{x}^{\pi_j}_i\right)-m_{n,1}\left(\mathbf{x}^{\pi_j}_i,\mathbf{\theta},\mathcal{D}_n\right) \right\vert^2
        d\mathbb{P_\pi}(\pi_j)
        d\mathbb{P}_{\mathbf{X}}(\mathbf{x})d\mathbb{P}_{\mathbf{\Theta}}(\mathbf{\theta})\\
        &=
        \nonumber
      \gamma_nM
        \int\int\sum_{(i_1,\ldots,i_{\gamma_n})\in \mathfrak{D}_{\gamma_n}} \frac{1}{!\gamma_n}\\
        &\quad\quad\quad\quad\quad\cdot
        \left\vert\widetilde{m}\left(\mathbf{x}^{(i_1,\ldots,i_{\gamma_n})}_i\right)-m_{n,1}\left(\mathbf{x}^{(i_1,\ldots,i_{\gamma_n})}_i,\mathbf{\theta},\mathcal{D}_n\right) \right\vert^2
        d\mathbb{P}_{\mathbf{X}}(\mathbf{x})d\mathbb{P}_{\mathbf{\Theta}}(\mathbf{\theta})\\
        &=
         \frac{\gamma_nM}{!\gamma_n}
        \sum_{(i_1,\ldots,i_{\gamma_n})\in \mathfrak{D}_{\gamma_n}}\\
        &\quad\quad\quad\quad\int\int
        \left\vert\widetilde{m}\left(\mathbf{x}^{(i_1,\ldots,i_{\gamma_n})}_i\right)-m_{n,1}\left(\mathbf{x}^{(i_1,\ldots,i_{\gamma_n})}_i,\mathbf{\theta},\mathcal{D}_n\right) \right\vert^2
        d\mathbb{P}_{\mathbf{X}}(\mathbf{x})d\mathbb{P}_{\mathbf{\Theta}}(\mathbf{\theta})\\
         &=
        \gamma_nM
       \mathbb{E}_{\mathbf{X},\mathbf{\Theta}}\left[
        \left\vert\widetilde{m}\left(\mathbf{X}_{j,i}\right)-m_{n,1}\left(\mathbf{X}_{j,i},\mathbf{\Theta},\mathcal{D}_n\right) \right\vert^2\right].
\end{align*}
 Since $\mathbf{X}_{j,i}\in[0,1]^p$, the inequality (\ref{Theorem 4 Baum ungleichung}) can also be applied to $$ \mathbb{E}_{\mathbf{X},\mathbf{\Theta}}\left[
        \left\vert\widetilde{m}\left(\mathbf{X}_{j,i}\right)-m_{n,1}\left(\mathbf{X}_{j,i},\mathbf{\Theta},\mathcal{D}_n\right) \right\vert^2\right].$$
        Hence (\ref{Theorem 4 Erwartung abgeschätzt}) holds for $\mathbb{E}_{\mathbf{X},\mathbf{\Theta}}\left[
        \left\vert\widetilde{m}\left(\mathbf{X}_{j,i}\right)-m_{n,1}\left(\mathbf{X}_{j,i},\mathbf{\Theta},\mathcal{D}_n\right) \right\vert^2\right]$ as well and thus 
$$\gamma_nM(\widetilde{m}(\mathbf{X}^{\pi_{j,t}}_i)^2-m_{n,1}(\mathbf{X}^{\pi_{j,t}}_i,\mathbf{\Theta},\mathcal{D}_n)^2)\overset{\mathbb{P}}{\longrightarrow}0.$$
It remains to show that $$\sqrt{\gamma_nM}2Y_i(\widetilde{m}(\mathbf{X}_i)-m_{n,1}(\mathbf{X}_i,\mathbf{\Theta},\mathcal{D}_n))$$ and 
$$\sqrt{\gamma_nM}2Y_i(\widetilde{m}(\mathbf{X}^{\pi_{j,t}}_i)-m_{n,1}(\mathbf{X}^{\pi_{j,t}}_i,\mathbf{\Theta},\mathcal{D}_n))$$ are also vanishing for growing $n$.
We use (A3) and (A4) to get
\begin{align*}
   \vert 2Y_i(\widetilde{m}(\mathbf{X}_i)-m_{n,1}(\mathbf{X}_i,\mathbf{\Theta},\mathcal{D}_n))\vert
   \leq2(K+K_{\varepsilon})\cdot\vert\widetilde{m}(\mathbf{X}_i)-m_{n,1}(\mathbf{X}_i,\mathbf{\Theta},\mathcal{D}_n)\vert.
\end{align*}
As before, Markov inequality can be applied, which leads to
\begin{align}
    &\mathbb{P}\left(\sqrt{\gamma_n M}\cdot 2\vert Y_i\vert \cdot\left\vert\widetilde{m}(\mathbf{X}_i)-m_{n,1}(\mathbf{X}_i,\mathbf{\Theta}_t,\mathcal{D}_n) \right\vert >\xi\right)\nonumber\\
    &\leq 
    \mathbb{P}\left(2(K+K_{\varepsilon})\cdot\sqrt{\gamma_n M}\left\vert\widetilde{m}(\mathbf{X}_i)-m_{n,1}(\mathbf{X}_i,\mathbf{\Theta}_t,\mathcal{D}_n) \right\vert >\xi\right)\nonumber\\
    &=
    \mathbb{P}\left(\sqrt{\gamma_n M}\left\vert\widetilde{m}(\mathbf{X}_i)-m_{n,1}(\mathbf{X}_i,\mathbf{\Theta}_t,\mathcal{D}_n) \right\vert >\frac{\xi}{2(K+K_{\varepsilon})}\right)\nonumber\\
    &\leq
    4(K+K_{\varepsilon})^2\gamma_nM
     \frac{\mathbb{E}_{\mathbf{X},\mathbf{\Theta}}\left[\left\vert\widetilde{m}(\mathbf{X}_i)-m_{n,1}(\mathbf{X}_i,\mathbf{\Theta}_t,\mathcal{D}_n) \right\vert^2\right]}{\xi^2}.
     \label{Theorem 4 Markov2}
\end{align}
Since (\ref{Theorem 4 Markov}) tends to $0$ as $n$ grows large, (\ref{Theorem 4 Markov2}) also converges to 0. Therefore,
\begin{align*}
    2Y_i(\widetilde{m}(\mathbf{X}_i)-m_{n,1}(\mathbf{X}_i,\mathbf{\Theta}_t,\mathcal{D}_n))\overset{\mathbb{P}}{\longrightarrow}0
\end{align*}
holds true.
Note that (\ref{Theorem 4 Markov2}) also holds for $\mathbf{X}_i^{\pi_{j,t}}$ instead of $\mathbf{X}_i$ with the exact same arguments,
which implies 
$  2Y_i(\widetilde{m}(\mathbf{X}^{\pi_{j,t}}_i)-m_{n,1}(\mathbf{X}^{\pi_{j,t}}_i,\mathbf{\Theta}_t,\mathcal{D}_n))\overset{\mathbb{P}}{\longrightarrow}0$.
Hence,
 \begin{align*}
         \frac{ I_{n,M}^{OOB}(j)-\widetilde I_{n,M}^{OOB}(j)}{\sqrt{\gamma_n^2\zeta_{1,\gamma_n}/n+\zeta_{\gamma_n}/M}}  
         \overset{\mathbb{P}}{\longrightarrow}0
 \end{align*}
 holds and completes the proof. 
\end{proof}

\end{appendix}


\begin{funding}
The first author was supported by the Deutsche Forschungsgemeinschaft (DFG, German Research Foundation) - 314838170, GRK 2297 MathCoRe. 
\end{funding}

\end{document}